\documentclass[11pt]{article}
\usepackage{a4wide}

\usepackage{amssymb}
\usepackage{amsmath}
\usepackage{amsthm}
\usepackage{tikz-cd}
\usepackage{mathrsfs}
\usepackage{graphicx} 
\usepackage[font=small]{subcaption}
\usepackage{pgfplots}

\usepackage{mathtools}
\usepackage{hyperref}
\usepackage{enumerate}

\usepackage{thmtools}
\usepackage{thm-restate}
\usepackage{cleveref}
\usepackage{mathbbol}
\usepackage{verbatim}
\usepackage{romannum}

\usepackage{latexsym}
\usepackage{amscd}
\usepackage{graphics}
\usepackage{amsmath}
\usepackage{amssymb}
\usepackage{esint}
\usepackage{mathrsfs}
\usepackage{bbm}
\usepackage{color}
\usepackage{accents}

\newtheorem{theorem}{Theorem}[section]
\newtheorem{prop}[theorem]{Proposition}

\newtheorem{definition}[theorem]{Definition}
\newtheorem{cor}[theorem]{Corollary}
\newtheorem{lemma}[theorem]{Lemma}
\newtheorem{lem}[theorem]{Lemma}

\newtheorem{remark}[theorem]{Remark}
\newtheorem{rmk}[theorem]{Remark}

\DeclareMathOperator{\I}{I}
\DeclareMathOperator{\II}{II}
\DeclareMathOperator{\III}{III}
\DeclareMathOperator{\IV}{IV}

\renewcommand{\ker}{\operatorname{ker}}

\renewcommand{\dim}{\operatorname{dim}}

\newcommand{\ind}{\operatorname{ind}}

\newcommand{\PD}{\mathrm{PD}}
\newcommand{\ev}{\mathrm{ev}}

\renewcommand{\H}{\mathrm{H}}

\newcommand{\CF}{\mathrm{CF}}
\newcommand{\HF}{\mathrm{HF}}
\newcommand{\SC}{\mathrm{SC}}
\newcommand{\SH}{\mathrm{SH}}
\newcommand{\CZ}{\mathrm{CZ}}

\newcommand{\CM}{\mathrm{CM}}

\newcommand{\QH}{\mathrm{QH}}

\newcommand{\reg}{\mathrm{reg}}
\newcommand{\crit}{\mathrm{crit}}

\newcommand{\R}{\mathbb{R}}
\newcommand{\Q}{\mathbb{Q}}
\newcommand{\Z}{\mathbb{Z}}

\newcommand{\N}{\mathbb{N}}
\newcommand{\C}{\mathbb{C}}
\newcommand{\Ham}{\mathrm{Ham}}

\DeclareFontFamily{U}{mathx}{\hyphenchar\font45}
\DeclareFontShape{U}{mathx}{m}{n}{
	<5> <6> <7> <8> <9> <10>
	<10.95> <12> <14.4> <17.28> <20.74> <24.88>
	mathx10
}{}
\DeclareSymbolFont{mathx}{U}{mathx}{m}{n}
\DeclareFontSubstitution{U}{mathx}{m}{n}
\DeclareMathAccent{\widecheck}{0}{mathx}{"71}
\DeclareMathAccent{\wideparen}{0}{mathx}{"75}

\numberwithin{equation}{section}
\AtEndDocument{\bigskip{\footnotesize
\noindent\textsc{Seoul National University, Department of Mathematical Sciences, Research Institute in Mathematics, 
	08826 Seoul, South Korea} \par  
 \noindent \textit{E-mail address}: \texttt{\href{mailto:jungsoo.kang@snu.ac.kr}{ghdrnjs2@snu.ac.kr }} \par

  \addvspace{\medskipamount}
\noindent\textsc{Seoul National University, Department of Mathematical Sciences, Research Institute in Mathematics, 
	08826 Seoul, South Korea} \par  
 \noindent \textit{E-mail address}: \texttt{\href{mailto:jungsoo.kang@snu.ac.kr}{jungsoo.kang@snu.ac.kr}} \par
   \addvspace{\medskipamount}

}}

\title{Biran decomposition and relative symplectic cohomology}
\author{Hong-kwon Jo and Jungsoo Kang}
\date{}

\begin{document}
	\pagenumbering{arabic}
	
	\maketitle
	\begin{abstract}
		Let $(M,\omega)$ be a closed monotone integral symplectic manifold with monotonicity constant $\tau>0$, and let $D\subset M$ be a symplectic hyperplane section of degree $0<k<\tau$. By Biran's decomposition theorem, $M$ is the union of a symplectic disk bundle over $D$ and the skeleton of the complement $M\setminus D$ equipped with its  Liouville structure. In this paper, we compute the relative symplectic cohomology of circle subbundles of this disk bundle over $\Z$ in terms of the quantum cohomology of $D$. 
		This relative symplectic cohomology depends on the radius of the circle bundle relative to a certain critical radius. For a broad class of pairs $(M,D)$, we prove that the relative symplectic cohomology over $\Z$ does not vanish for supercritical radii. Therefore, circle bundles of supercritical radii are heavy subsets, which yields infinite dimensional quasi-flats in the Hamiltonian diffeomorphism group with respect to both the Hofer metric and the spectral metric over $\Z$. We also establish that if the degree $k$ of $D$ is greater than one, the skeleton of $M\setminus D$ is SH-full over $\Z/k\Z$, and hence non-displaceable in $M$ by any symplectomorphism. 
	\end{abstract}

	\tableofcontents
	\section{Introduction}\label{sec:Intro}
	
	Let $(M,\omega)$ be a closed symplectic manifold, and let $\Ham(M,\omega)$ denote the group of Hamiltonian diffeomorphisms. One central conjecture in the study of the geometry of $\Ham(M,\omega)$ states that it has infinite diameter with respect to the Hofer metric, which is a bi-invariant Finsler type metric, see \cite[Problem 20]{MS17}. This conjecture has been confirmed in several cases, including surfaces, symplectically aspherical manifolds, $\mathbb{CP}^n$ with Fubini-Study form, degree $d$ hypersurfaces in $\mathbb{CP}^n$ with $d>n+1$. It also has been settled under various structural assumptions, such as the existence of certain Lagrangian submanifolds, the presence of a Hamiltonian function with specific dynamical properties, or conditions on the (quantum) cohomology. Furthermore, in many of these cases, it is shown that $\R^N$ with arbitrarily large $N$ is quasi-isometrically embedded into $\Ham(M,\omega)$. A non-exhaustive list of relevant works includes \cite{LM95,Pol98,Ost03,Py08,McD09,Ush13,Zap13,PS23,CGHS23,Mai24,Sun24,FZ25,AACS25}.

	Corresponding questions have been investigated for the spectral metric, which is another bi-invariant metric on $\Ham(M,\omega)$ defined via Floer homology \cite{Vit92,Sch00,Oh05}. This is directly relevant to the aforementioned conjecture since the Hofer norm is bounded from below by the spectral norm. Indeed, some of the references cited above establish that $\Ham(M,\omega)$ has infinite diameter with respect to the spectral metric. Moreover, a folklore conjecture states that the spectral diameter of $\Ham(M,\omega)$ is infinite when $\omega(\pi_2(M))=0$. It is important to note that the spectral metric depends on the choice of coefficient ring for Floer homology. For example, for $\mathbb{CP}^n$ with the Fubini-Study form $\omega_\mathrm{FS}$, the spectral diameter over field coefficients is finite, see \cite{EP03,KS21}. However, it is recently discovered in \cite{KS24} that the spectral diameter of $\Ham(\mathbb{CP}^n,\omega_\mathrm{FS})$ for $n\geq 2$ is infinite over $\Z$-coefficients. In light of this development, one could ambitiously conjecture that, for every closed symplectic  manifold $(M,\omega)$, the spectral diameter of $\Ham(M,\omega)$ is infinite over $\Z$.

	Introduced in \cite{EP09}, heavy subsets are fundamental objects exhibiting rigidity properties. For instance, Lagrangian submanifolds with nonvanishing Floer homology are heavy, see \cite{FOOO19}. Furthermore, the presence of disjoint heavy subsets gives rise to quasi-flats in $\Ham(M,\omega)$ or its universal cover. Therefore, detecting heavy subsets is a central problem,  and the relative symplectic cohomology \cite{Var21} serves as a powerful  Floer-theoretic invariant for this purpose. Indeed, a recent remarkable result \cite{MSV24} (see also \cite{DGPZ24}) shows that a compact subset $K\subset M$ is heavy if and only if the relative symplectic cohomology $\SH_M^*(K)$ is nonzero. 
	
	\medskip
	
	We now describe the geometric setup used throughout this paper. Let $(M,\omega)$ be a closed integral symplectic manifold of dimension $2n$. We assume that it is spherically monotone, i.e.~there exists $\tau>0$ such that $c_1^{TM}(A)=\tau\omega(A)$ for every $A\in\pi_2(M)$, where $c_1^{TM}$ denotes the first Chern class of $M$. We also assume that there exists a symplectic hyperplane section $D$ of degree $0<k<\tau$. More precisely, $D\subset M$ is a closed codimension two symplectic submanifold whose homology class  is Poincar\'e dual to $k[\omega]_\Z$ for some positive integer $k<\tau$, where $[\omega]_\Z\in\H^2(M;\Z)$ is an integral lift of $[\omega]\in\H^2(M;\R)$. We refer to  Remark \ref{rem:high_deg} for the case $k\geq \tau$.
	
	By Biran's decomposition theorem \cite{Bir01,Gir17}, the complement $M\setminus D$ admits a Liouville structure. We denote by $\mathrm{sk}(M\setminus D)$ the skeleton of the Liouville flow. Moreover, $M\setminus \mathrm{sk}(M\setminus D)$ is symplectomorphic to a symplectic open unit disk bundle $E$ of $D$ whose fiber disk has area $1/k$. Let $\psi:E\to M$ denote this symplectic embedding. For  $0<\rho< 1$, we write $E_\rho\subset E$ for the open disk subbundle of radius $\sqrt{\rho}$, and set $U_\rho:=\psi(E_\rho)$. The main result of this paper is to compute the relative symplectic cohomology $\SH_M^*(\partial U_\rho)$ of $\partial U_\rho$ with $\Z$-coefficients. It is crucial to use $\Z$-coefficients for nonvanishing of $\SH_M^*(\partial U_\rho)$, and hence to find a family of disjoint heavy sets, which yields an infinite-dimensional quasi-flat in $\Ham(M,\omega)$.

	\subsection{Computations of relative symplectic cohomology}
	
	We compute $\SH_M^*(\partial U_\rho)$ in terms of 
	the quantum cohomology ring $\QH^*(D;\Gamma)=\H^*(D;\Z)\otimes_\Z\Gamma$ of $D$, where $\Gamma=\Z[T,T^{-1}]$ if $\omega(\pi_2(M))\neq0$, and $\Gamma=\Z$ otherwise.\footnote{As discussed in Section \ref{sec:rel_sh}, our relative symplectic cohomology corresponds to the version in \cite{Var21} defined over the universal Novikov ring $\Lambda^\mathrm{univ}$ in \eqref{eq:universal} rather than the version over the Novikov ring with nonnegative exponents, which is also considered in \cite{Var21}.} In contrast to the usual convention, we define the degree of $T$ to be $2(\tau-k)m_M$, where $m_M$ denotes the positive generator of $\omega(\pi_2(M))=m_M\Z$. In fact, $\SH_M^*(\partial U_\rho)$ depends on the size of $\rho$ compared to the critical radius 
	\[
	\rho_\crit:=\frac{k}{\tau}.
	\]
	This is precisely the radius where the index and action of every capped periodic Reeb orbit on $\partial U_{\rho_\crit}$ are proportional via the same constant  $2\tau$, see \eqref{eq:crit_radius}. It is also the unique radius for which the Lagrangian lift of a monotone Lagrangian submanifold in $D$ is again monotone, see \cite[Proposition 6.4.1]{BC6}. Such a  critical radius appeared also in \cite{BSV22}.
	To capture this phenomenon, we introduce an extension of $\Gamma$:
	\begin{equation*}
		\widehat\Gamma_\rho :=\begin{cases}
			\Gamma((S)) & \rho>\rho_{\crit},\\
			\Gamma[S,S^{-1}] & \rho=\rho_{\crit},\\
			 \Gamma((S^{-1})) & \rho<\rho_{\crit},
			\end{cases}
	\end{equation*}
	where the degree of $S$ is $2$.  We also define $\QH^*(D;\widehat\Gamma_\rho)$ 
	whose product structure is inherited from $\QH^*(D;\Gamma)$.	 Let $[\omega_D]_\Z\in\H^2(D;\Z)$ denote the pullback of $[\omega]_\Z$ under the inclusion $D\subset M$. 
  We consider a degree two class 
	\begin{equation}\label{eq:e_M}
			e_M:=k[\omega_D]_\Z+{k}N_1(M,D) T\in \QH^2(D; \Gamma).
		\end{equation}
	Here, $N_1(M,D)\in\Z$ is the signed count of $J$-holomorphic spheres $v$ in $M$ that intersect $D$ exactly at one point with contact order $km_M$. It is nonzero only if $(\tau-k)m_M=1$, or equivalently, $\deg T=2$. When $\Gamma=\Z$, $N_1(M,D)=0$. 
	
	\begin{theorem}\label{thm:sh}
		For every $0<\rho<1$, there is a  $\Z$-module isomorphism 
		\[
		\SH^*_M(\partial U_\rho) \cong \QH^{*+n}(D;\widehat\Gamma_\rho) / (e_M+S),
		\]
		where $(e_M+S)$ denotes the ideal generated by $e_M+S \in \QH^{2}(D;\widehat\Gamma_\rho)$. 
	\end{theorem}
	
	\begin{cor}\label{cor:sh=0}
		We have $\SH^*_M(\partial U_\rho)=0$  if and only if any of the following conditions holds:
		\begin{enumerate}[(a)]
			\item $\rho <\rho_\crit$,
			\item $\rho=\rho_\crit$ and the class $e_M\in \QH^2(D; \Gamma)$ is nilpotent,
			\item $\rho>\rho_\crit$ and the sequence of ideals 
			\begin{equation}\label{eq:ideal}
				\QH^*(D;\Gamma)\supseteq (e_M)\supseteq (e_M^{\star 2}) \supseteq (e_M^{\star 3})\supseteq \cdots
			\end{equation}
			stabilizes, where $e_M^{\star l}$ denotes the $l$-fold quantum product of $e_M$ with itself.
		\end{enumerate}
	\end{cor}

	We recall the following properties of relative symplectic cohomology that are relevant to our results. We refer to  \cite{Var21,AGV24} for other fundamental properties.
	\begin{enumerate}[(1)]
		\item $\SH_M^*(M)\cong \QH^*(M;\Lambda)$, where $\Lambda=\Z[\mathfrak{t},\mathfrak{t}^{-1}]$ with $ \mathfrak{t}$ of degree twice the minimal Chern number of $M$ if $\omega(\pi_2(M))\neq0$, and $\Lambda=\Z$ otherwise.
		\item For any compact subsets $K\subset K'$ of $M$, there is a restriction map $\SH_M^*(K')\to \SH_M^*(K)$. Moreover, $\SH_M^*(K')=0$ implies $\SH_M^*(K)=0$.		
		\item If $K$ is stably displaceable, then $\SH_M^*(K)=0$.
		\item The Mayer-Vietoris property of relative symplectic cohomology holds for any pair of compact subsets $K,K'\subset M$ such that there are Poisson commuting functions $f,f'$ on $M$ with $f^{-1}(0)=K$, $f'^{-1}(0)=K'$.
	\end{enumerate}
	
	Recall that a compact subset $K\subset M$ is called stably non-displaceable if the product of $K$ with the zero section of $T^*S^1$ is non-displaceable in $M\times T^*S^1$ by any compactly supported Hamiltonian diffeomorphism.
	
	Let $\overline{U}_\rho$ denote the closure of $U_\rho$. Since $D$ is stably displaceable, $\overline{U}_\rho$ is also stably displaceable for sufficiently small $\rho>0$, which implies  $\SH_M^*(\overline{U}_\rho)=0$. 
	
	\begin{prop}\label{prop:annuli}
		For any $0<\rho_1\leq\rho_2 <\rho_\crit <\rho_3\leq \rho_4< 1$, the following hold.
		\begin{enumerate}[(a)]
			\item $\SH^*_M( \overline{U}_{\rho_2} \setminus U_{\rho_1})\cong\SH^*_M( \partial{U}_{\rho})=0$ for any $\rho_1\leq \rho\leq\rho_2$.
			\item $\SH^*_M( \overline{U}_{\rho_3} \setminus U_{\rho_2})\cong \SH^*_M(\partial U_{\rho_\crit})$.
			\item $\SH^*_M( \overline{U}_{\rho_4} \setminus U_{\rho_3})\cong \SH^*_M(\partial U_{\rho})$ for any $\rho_3\leq \rho\leq \rho_4$.
			\item $\SH^*_M(\overline{U}_{\rho_1})=0$.
			\item $\SH^*_M(\overline{U}_{\rho_3})\cong\SH^*_M(\partial U_{\rho_\crit})$.
			\item $\SH^*_M(M\setminus U_{\rho_3})\cong \SH^*_M(M\setminus U_{\rho_4})$.
			\item If $\mathrm{sk}(M\setminus D)$ is stably displaceable, then $\SH^*_M(M\setminus U_{\rho_3})=0$ and 
			 \[
			 \QH^*(M;\Lambda)\cong \SH_M^*(\partial U_{\rho_\crit})\cong \SH_M^*(\overline{U}_{\rho_4})\cong \SH_M^*(M\setminus U_{\rho_2}).
			 \]
		\end{enumerate}	
		\end{prop}
		\begin{proof}
			It is clear from the proof of Theorem \ref{thm:sh} that the computation of $\SH^*_M( \partial{U}_{\rho})$ adapts straightforwardly to $\SH^*_M( \overline{U}_{\rho'} \setminus U_{\rho})$, with the resulting cohomology depending on the position of $\rho\leq\rho'$ relative to $\rho_\crit$. More precisely, the generators and the Floer (continuation) cylinders for the chain complex of $\SH^*_M( \overline{U}_{\rho'} \setminus U_{\rho})$ are identical up to an action adjustment depending on $\rho\leq\rho'$, while the completion of the chain complex is determined by the position of $[\rho,\rho']$ relative to $\rho_\crit$. This proves the isomorphisms in (a)-(c). Furthermore, by Proposition \ref{prop:inc}, the isomorphisms are given by the restriction maps. 
			
			The statements (d)-(g) follow from the Mayer-Vietoris property. Note that (d) was proved in \cite{BSV22} using the Fukaya trick, and the other isomorphisms can plausibly also be proved by the same technique.  
		\end{proof}
	
	\begin{rmk}\label{rem:high_deg}
		In the case of $k\ge \tau$, it is established in \cite[Theorem D]{BSV22} that $\mathrm{sk}(M\setminus D)$ is $\SH$-full, meaning that every compact subset $K'$ disjoint from $\mathrm{sk}(M\setminus D)$ has $\SH_M^*(K')=0$. Therefore, we obtain $\SH^*_M(\partial U_\rho)=\SH^*_M(\overline{U}_\rho)=0$ for all $0<\rho<1$.
	\end{rmk}

	\begin{remark}\label{rem:coeff}
		The above results continue to hold if the ground ring $\Z$ is replaced by any commutative ring, or if the Novikov ring $\Gamma=\Z[T,T^{-1}]$ is replaced by a Novikov field $\mathbb{K}((T))$. In the latter case, the sequence \eqref{eq:ideal} always stabilizes since $\QH^*(D,\mathbb{K}((T)))$ is a finite dimensional vector space over $\mathbb{K}((T))$. 
	\end{remark}
	
	\begin{remark}
	 Theorem \ref{thm:sh} can be viewed as synthesizing the results of  \cite{AK23,BKK24}. Removing either the divisor $D$ or the skeleton $\mathrm{sk}(M\setminus D)$ yields a symplectic disk bundle $E$ or a Liouville manifold $M\setminus D$, respectively. The Rabinowitz Floer homologies of the dual of $E$ and of $M\setminus D$ were computed in \cite{AK16,Ven18,Ven21,AK23} and \cite{BKK24}, respectively. Theorem \ref{thm:sh} integrates these two  computational perspectives.
	\end{remark}

	\subsection{Heavy and superheavy subsets}

		The concepts of heaviness and superheaviness for compact subsets were introduced in \cite{EP09}. Among  other fundamental results, they proved that superheaviness implies heaviness, that every heavy set is stably non-displaceable, and that every superheavy set intersects every heavy set. The last property  implies that any superheavy set is strongly non-displaceable. 
		Recall that a compact subset $K\subset M$ is called strongly non-displaceable if it is non-displaceable by any symplectomorphism of $M$. The notion of (super)heaviness is defined via the spectral invariants of the Floer homology of $M$ associated with an idempotent element of the quantum cohomology of $M$. In this paper, we take this idempotent element to be the unit element and, unless otherwise stated, work primarily  with $\Z$-coefficients. 
		
		 The following remarkable properties of $\SH_M^*(K)$ established in \cite{MSV24} show that it is an effective  invariant to detect (super)heavy sets.
		
		\begin{enumerate}[(1)]
			\item  A compact subset $K\subset M$ is heavy (with respect to the unit of $\QH^*(M;\Lambda)$ over $\Z$) if and only if $\SH_M^*(K)\neq0$.
			\item If $K$ is superheavy, then $K$ is SH-full,  namely every compact subset $K'\subset M$ disjoint from $K$ has $\SH_M^*(K')=0$.
			\item If $K$ is SH-full, then $K$ is superheavy provided that, for some nonnegative $H\in C^\infty(M)$ with $H|_K=0$ and any $m\in\N$, $c(1;mH)=c([\omega^{\wedge n}];-mH)$ holds and the boundary depth of the Floer chain complex for $mH$ is uniformly bounded.
		\end{enumerate}
		Here $c(\eta;H)$ denotes the spectral invariant of $H$ with respect to $\eta\in\QH^*(M;\Lambda)$, see Section \ref{sec:spec}.
		Although we do not know whether property (3) holds unconditionally, the invariance $\SH_M^*(K)\cong \SH_M^*(\varphi(K))$ under symplectomorphisms $\varphi$ implies the following.
		\begin{enumerate}[(1)] 
			\item[(4)] If $K$ is heavy and SH-full (with any ground ring), then $K$ is strongly non-displaceable.	
		\end{enumerate}

	\begin{theorem}\label{thm:heavy}
	The following hold.
	\begin{enumerate}[(a)]

		\item  If $\SH^*_M(M\setminus U_{\rho})\neq0$ for some $\rho>\rho_\crit$, then $\SH^*_M(\mathrm{sk}({M\setminus D}))\neq0$.
		\item If  $\SH_M^*(\partial U_{\rho})\neq 0$ for some $\rho>\rho_\crit$, then $\SH_M^*(\partial U_{\rho_\crit})\neq 0$.

		\item At least one of $\mathrm{sk}({M\setminus D})$ and $\partial U_{\rho_\crit}$ is heavy.
		\item If $\mathrm{sk}({M\setminus D})$ is not heavy, then the restriction map $\QH^*(M;\Lambda)\to \SH_M^*(\partial U_{\rho_\crit})$ is an isomorphism. Moreover, $\partial U_{\rho_\crit}$ is SH-full and strongly non-displaceable.

		\item If $\partial U_{\rho_\crit}$ is not heavy, then the restriction map $\QH^*(M;\Lambda)\to \SH_M^*(M\setminus U_\rho)$ is an isomorphism for any $\rho>\rho_\crit$. Moreover, $\mathrm{sk}({M\setminus D})$ is SH-full and strongly non-displaceable.

		\item If $e_M\in\QH^*(D;\Gamma)$ is not nilpotent, then $\partial U_{\rho_\crit}$ is heavy. 
		\item If the sequence in \eqref{eq:ideal} does not stabilize, then, $\partial U_\rho$ is heavy for every $\rho> \rho_\crit$. Consequently, $\mathrm{sk}({M\setminus D})$ and $\partial U_{\rho_\crit}$ are also heavy.
 
		\item If the degree of $D$ satisfies $k\geq 2$, then the restriction map $\QH^*(M;\Lambda)\to \SH_M^*(\partial U_{\rho_\crit})$ is not an isomorphism. Consequently,  by (d), $\mathrm{sk}(M\setminus D)$ is heavy. (cf.~Theorem \ref{thm:superheavy}.(a))

	\end{enumerate}
	\end{theorem}
	\begin{proof}
		
		We first prove (a). By the assumption together with Proposition \ref{prop:annuli}.(f), for every $\rho>\rho_\crit$, $\SH^*_M(M\setminus U_{\rho})\neq0$ and thus $M\setminus U_{\rho}$ is heavy. This implies that $\mathrm{sk}({M\setminus D}) =\bigcap_{\rho>\rho_\crit} M\setminus U_{\rho}$ is also heavy, see \cite[Lemma 3.7]{MSV24}. Hence, $\SH^*_M(\mathrm{sk}({M\setminus D}))\neq0$.
	
		Statement (b) follows from Corollary \ref{cor:sh=0}. Suppose $\SH_M^*(\partial U_{\rho_\crit})=0$. Then, $e_M$ is nilpotent. This implies that the sequence in \eqref{eq:ideal} stabilizes and hence $\SH_M^*(\partial U_{\rho})=0$ for every $\rho>\rho_\crit$. One can also see this alternatively as follows. If $\partial U_{\rho}$ is heavy for some $\rho>\rho_\crit$, then $\partial U_{\rho'}$ is heavy for every $\rho'>\rho_\crit$ by Proposition \ref{prop:annuli}.(c). This yields that $\partial U_{\rho_\crit}$ is also heavy. 
	
		We take any $0<\rho_1<\rho_2<\rho_\crit<\rho_3<\rho_4<1$. By Proposition \ref{prop:annuli}.(d), $\SH_M^*(\overline{U}_{\rho_2})=0$. This together with the Mayer-Vietoris property implies that at least one of $\SH^*_M( \overline{U}_{\rho_4} \setminus U_{\rho_1})$ and $\SH^*_M( M \setminus U_{\rho_3})$ is nonzero. In the former case, $\partial U_{\rho_\crit}$ is heavy by Proposition \ref{prop:annuli}.(b). In the latter case, $\mathrm{sk}({M\setminus D})$ is heavy by (a). This proves (c). 
		
		We now prove (d). By (a), $\SH_M^*(\mathrm{sk}({M\setminus D}))=0$ implies $\SH^*_M(M\setminus U_{\rho_3})=0$. Thus, Proposition \ref{prop:annuli} and the Mayer-Vietoris property yield that the restriction map $\QH^*(M;\Lambda)\to\SH^*_M( \overline{U}_{\rho_4} \setminus U_{\rho_1})\cong \SH^*_M(\partial U_{\rho_\crit})$ is an isomorphism. Moreover, $\SH_M^*(\overline{U}_{\rho_2} \cup (M\setminus U_{\rho_3}))\cong \SH_M^*(\overline{U}_{\rho_2})\oplus \SH_M^*(M\setminus U_{\rho_3})=0$. This proves that $\partial U_{\rho_\crit}$ is SH-full.

		Statement (e) follows analogously again from Proposition \ref{prop:annuli} and the Mayer-Vietoris property.	
		Statements (f) and (g) are direct consequences of Corollary \ref{cor:sh=0}. For the latter claim in (g), we use (a) and (b).  
	
		We consider the ring 
		$
		R:= \frac{\QH^*(D;\Gamma)\otimes_\Z\Z[S,S^{-1}]}{(e_M+S)},
		$
		which is isomorphic to $\SH_M^*(\partial U_{\rho_\crit})$ as $\Z$-module by Theorem \ref{thm:sh}. Since $-e_MS^{-1}$ equals the unit element of $R$,  where $e_M=k([\omega_D]_\Z +N_1(M,D)T)$, we have $R/kR=0$. However, $\QH^*(M;\Lambda)$ has a free factor $\Z$, and thus $\QH^*(M;\Lambda)/k\QH^*(M;\Lambda)$ contains $\Z/k\Z$. Therefore, for $k\geq 2$, $R\not\cong \QH^*(M;\Lambda)$ as $\Z$-modules. This proves (h). 
	\end{proof}

	We stress again that working with $\Z$-coefficients is essential to maximize the nonvanishing of $\SH_M^*(\partial U_\rho)$ for $\rho\geq \rho_\crit$. Indeed, heaviness over some commutative ring $R$  implies heaviness over $\Z$, but the converse does not hold in general. On the other hand, employing other coefficient rings is more effective when seeking strongly non-displaceable (or SH-full) subsets. Indeed, superheaviness over $\Z$ implies superheaviness over any commutative ring $R$, see Section \ref{sec:spec}. In our setting, taking coefficients in $\Z/k\Z$ is a natural choice.

	\begin{theorem}\label{thm:superheavy}
	Let $(M,D)$ be as before, where $\PD([D])=k[\omega]_\Z$.
	\begin{enumerate}[(a)]
		\item For $k\geq 2$, $\mathrm{sk}({M\setminus D})$ is strongly non-displaceable.
		\item Let $D'\subset M$ be another symplectic hyperplane section of degree $0<k'<\tau$. Assume that $\gcd(k,k')\geq2$. Then, $\mathrm{sk}({M\setminus D})\cap \mathrm{sk}({M\setminus D'})\neq \emptyset$.
		\item For $k=1$, $\mathrm{sk}({M\setminus D})\cup \partial U_{\rho_\crit}$ is strongly non-displaceable. (For $k\geq 2$, see (a).)
	\end{enumerate} 
	\end{theorem}
	\begin{proof}	
	To prove statement (a), we work with $\Z/k\Z$-coefficients. Then, the class $e_M$ is zero and  $\SH_M^*(\partial U_{\rho_\crit})=0$ over $\Z/k\Z$ by Corollary \ref{cor:sh=0}. Arguing as in the proof of Theorem \ref{thm:heavy}.(d,e), we can show that the restriction map $\QH^*(M;\Lambda)\to \SH_M^*(\mathrm{sk}({M\setminus D}))$ is an isomorphism and $\mathrm{sk}({M\setminus D})$ is SH-full over $\Z/k\Z$. Therefore, $\mathrm{sk}({M\setminus D})$ is strongly non-displaceable.

	Statement (b) is proved analogously by working with $\Z/{\gcd(k,k')}\Z$-coefficients.

	To prove (c), we work with Novikov fields $\mathbb{K}((T))$ and $\mathbb{K}((\mathfrak{t}))$, rather than $\Z[T,T^{-1}]$ and $\Z[\mathfrak{t},\mathfrak{t}^{-1}]$. Then, by Corollary \ref{cor:sh=0} and Remark \ref{rem:coeff}, $\SH_M^*(\partial U_\rho)=0$ over $\mathbb{K}$ for any $\rho>\rho_\crit$. By Proposition \ref{prop:annuli} and the Mayer-Vietoris sequence, the restriction map $\QH^*(M;\mathbb{K}((\mathfrak{t})))\to \SH_M^*((M\setminus U_\rho)\cup \partial U_{\rho_\crit})$ for any $\rho>\rho_\crit$ is an isomorphism over $\mathbb{K}$. Moreover, $\mathrm{sk}({M\setminus D})\cup \partial U_{\rho_\crit}$ is SH-full over $\mathbb{K}$, and hence strongly non-displaceable.
	\end{proof}
	
	Theorem \ref{thm:superheavy}.(a) extends the following earlier results. In the K\"ahler setting without the monotonicity assumption, it is proved in \cite{GSZ21} that $M\setminus D$ is not subcritical for $k\geq 2$, confirming a conjecture in \cite{BC01}. In the present  setting, it is shown in \cite{BKK24} that the symplectic homology of $M\setminus D$ does not vanish when $k\geq 2$. 
	
	A few examples for which $M\setminus D$  admits a subcritical Stein structure  can be found in \cite{BJ13}. In this case, $\mathrm{sk}({M\setminus D})$ is displaceable, and hence $\SH_M^*(\mathrm{sk}({M\setminus D}))=0$. Theorem \ref{thm:heavy}.(e,h) and Theorem \ref{thm:superheavy} complement the results in \cite{TV21,Sun24} asserting that if $M$ is negatively monotone or $c_1^{TM}(\pi_2(M))=0$, then $\mathrm{sk}(M\setminus D)$ is superheavy.

	\begin{remark}
		As we recalled above, two conditions are needed to conclude that $K$ is superheavy from its SH-fullness. One concerns the spectral invariants, and the other is a bound on the  boundary depth. The second condition holds in our setting because, in the course of the proof of the main result, we classify Floer cylinders contributing to the boundary operator, and their energies are uniformly bounded. The first condition is satisfied, for example for $(\mathbb{CP}^n,\omega_\mathrm{FS})$, see \cite[Remark 3.12]{MSV24}. 
		
		Thus, the above result can be used to find superheavy subsets in $\mathbb{CP}^n$. For $D=\mathbb{CP}^{n-1}$, we can conclude from Theorem \ref{thm:heavy}.(d) that $\partial U_{\rho_\crit}$ is superheavy. This was known because $\partial U_{\rho_\crit}$ contains the Clifford torus, which is superheavy, see \cite[Proposition 3.7.1]{KS24}. Applying Theorem \ref{thm:superheavy}.(a) to  $D\subset \mathbb{CP}^n$ with $k\geq 2$,   we  deduce that $\mathrm{sk}(M\setminus D)$ is superheavy over $\Z/k\Z$. When $k=2$, taking $D$ to be a smooth quadric yields $\mathrm{sk}(M\setminus D)=\mathbb{RP}^n$. This recovers the fact that $\mathbb{RP}^n\subset \mathbb{CP}^n$ is superheavy over $\Z/2\Z$, which was established in \cite{EP09}.
	\end{remark}

	\subsection{The large scale geometry of $\Ham(M,\omega)$}
	
	Let $\gamma$ be the spectral norm on $\Ham(M,\omega)$ descended from the spectral pseudo-norm $\widetilde\gamma$ on the universal cover ${\widetilde{\Ham}}(M,\omega)$. These spectral (pseudo)-norms also depend on the choice of coefficients, and we take coefficients in $\Z$. In the setting of this paper, the projection ${\widetilde{\Ham}}(M,\omega)\to \Ham(M,\omega)$ given by $[H]\mapsto \phi_H^1$ is a quasi-isometry:~there exists $C>0$ such that 
	\[
	\widetilde\gamma([H])-C\leq \gamma(\phi_H^1)\leq \widetilde\gamma([H])
	\]
	see Proposition \ref{prop:ham}. Therefore, the presence of disjoint heavy subsets implies the existence of quasi-flats in $\Ham(M,\omega)$.

	\begin{theorem}\label{thm:diameter0}
		Assume that $(M,\omega)$ admits a symplectic hyperplane section $D$ of degree $2\leq k<\tau$, and $e_M\in\QH^*(D;\Gamma)$ is not nilpotent. Then, there exists a quasi-isometric embedding of $\R$ into ${\Ham}(M,\omega)$ with respect to both the spectral norm and the Hofer norm.
	\end{theorem}
	\begin{proof}
		By Theorem \ref{thm:heavy}.(f,h), the assumption yields that there are two disjoint heavy sets, namely  $\partial U_{\rho_\crit}$ and $\mathrm{sk}(M\setminus D)$. Then, the claim follows by adapting the proof of \cite[Lemma 2.5]{Sun24}, see also the proof of Theorem  \ref{thm:diameter}.
	\end{proof}

	\begin{theorem}\label{thm:diameter}		
		Assume that $(M,\omega)$ admits a symplectic hyperplane section $D$ such that the sequence of ideals in \eqref{eq:ideal} does not stabilize. Then, there exists a quasi-isometric embedding of $(C_c^\infty((0,1)),\|\cdot\|_{C_0})$ into ${\Ham}(M,\omega)$ with respect to both the spectral norm and the Hofer norm. In particular, the claim holds for $(M,\omega)$ listed in Proposition  \ref{prop:list}.
	\end{theorem}
	\begin{proof}
		The assumption implies that $\partial U_\rho$ is heavy for every $\rho\geq \rho_\crit$. Then, we argue as in \cite[Lemma 2.6]{FZ25}, which we recall for completeness. We fix a diffeomorphism $\chi:(\rho_\crit,1)\to(0,1)$. Given $h\in C^\infty_c((0,1))$, let $H$ be the Hamiltonian on $M$ supported on $U_1\setminus U_{\rho_\crit}$ defined by $H(x)=h\circ \chi(\rho_x)$ for $x\in U_1\setminus U_{\rho_\crit}$, where $\rho_x\in(\rho_\crit,1)$ is the constant satisfying $x\in \partial U_{\rho_x}$. Then, the embedding
		\[
		\jmath: C_c^\infty((0,1))\to \widetilde{\Ham}(M,\omega),\qquad h\mapsto[H]
		\]
		satisfies $\jmath(h+h')=[H+H']$. Moreover,
		\[
		\|h\|_{C^0}\leq \widetilde{\gamma}([H])=\max h-\min h = \|H\|_\mathrm{Hofer} \leq 2\|h\|_{C^0}
		\]
		where the first equality uses the fact that $\partial U_\rho$ for $\rho>\rho_\crit$ are heavy. Therefore, 
		\[
		\|h\|_{C^0}-C\leq \gamma(\phi_H^1) \leq \|\phi_H^1\|_{\mathrm{Hofer}} \leq \|H\|_\mathrm{Hofer}=\widetilde{\gamma}([H]) \leq \gamma(\phi_H^1)+C\leq 2\|h\|_{C^0}+C.
		\]
		This establishes the claimed quasi-isometric embedding.
	\end{proof}

	It is therefore an important question to determine which $(M,\omega)$ fulfill the assumption of Theorem \ref{thm:diameter}. We show that this is the case for a broad class of examples.
	
	\begin{lem}\label{lem:stab}
	The sequence of ideals in \eqref{eq:ideal} does not stabilize when $N_1(M,D)=0$ (e.g.~$\tau-k\geq 2$),  $k[\omega_D]_\Z$ is not primitive (e.g.~$k\geq 2$), and  $[\omega_D]_\Z$ is not nilpotent in the free part $\QH^*_\mathrm{free}(D;\Gamma)$ of $\QH^*(D;\Gamma)$.
	\end{lem}
	The use of $\Z$-coefficients is crucial in this lemma, see Lemma \ref{lem:non_stab} for the proof. To frame the non-nilpotency condition of $[\omega_D]_\Z$ within a more well-studied context, let us assume that $D$ is monotone,  i.e.~$c_1^{TD}$ and $[\omega_D]_\Z$ are positively proportional. Then, $[\omega_D]_\Z$ is non-nilpotent if and only if $c_1^{TD}$ is non-nilpotent. The latter holds whenever the quantum multiplication operator by $c_1^{TD}$ in the quantum cohomology of $D$ with $\C$-coefficients has a nonzero eigenvalue. This property is conjectured to hold for all Fano manifolds, see \cite{GHIKLS24}. In fact, we learned from Hiroshi Iritani that such property is inherited from $M$.

	\begin{lemma}\label{lem:fano}
		Let $M$ be a Fano manifold with a primitive ample integral class $[\omega]_{\Z}$. Assume that $c_1^{TM}=\tau[\omega]_\Z$ for $\tau\geq 4$ and quantum multiplication by $c_1^{TM}$ has a nonzero eigenvalue. Assume that there is a Fano hypersurface $D\subset M$ such that $\PD([D])=k[\omega]_\Z$ with $2\leq k\leq \tau -2$. Then, the sequence of ideals in \eqref{eq:ideal} does not stabilize.
	\end{lemma}
	\begin{proof}
		By \cite[Corollary 27]{GI17}, the quantum multiplication by $c_1^{TD}$ on the quantum cohomology of $D$ with $\C$-coefficients also has a nonzero eigenvalue provided $1\leq k\leq\tau-2$. Thus, the claim follows from Lemma \ref{lem:stab}. 
	\end{proof}

	\begin{prop}\label{prop:list}
		The manifolds $(M,\omega)$ in the following list admit a symplectic hyperplane section $D\subset M$ such that the sequence of ideals in \eqref{eq:ideal} does not stabilize. 
	\begin{enumerate}[(i)]
		\item degree $d$ hypersurfaces $X_d\subset \mathbb{CP}^n$ with $n\geq d+3$,
		\item complete intersections $X_{(d_1,\dots,d_l)}$ of degree $(d_1,\dots,d_l)$ in $\mathbb{CP}^{n}$ with $n\geq d_1+\dots +d_l+3$,
		\item complex Grassmannians $\mathrm{Gr}(k,r)$ with $r\geq 4$,
		\item products of manifolds in (i)-(iii) with the same Fano index, e.g.~$\mathbb{CP}^m\times X_{d_1,\dots,d_l}\times\mathrm{Gr}(k,r)$ with $m+1=n+1-(d_1+\dots+d_l)=r\geq 4$,
		\item $\mathbb{CP}^2=\mathrm{Gr}(1,3)$ and quadrics $Q^2,\,Q^3$ ($\mathbb{CP}^{m\geq 3}$ and $Q^{m\geq 4}$ are covered by (i)),
	\end{enumerate}
		with symplectic forms naturally inherited from the Fubini-Study form $\omega_\mathrm{FS}$.
	\end{prop}
	\begin{proof}
				
		We first recall a basic fact about Fano manifolds. Let $X$ be a Fano manifold with $\dim_\C X\geq 2$. Its Fano index $\tau\in\mathbb{N}$ is the largest positive integer for which the anticanonical line bundle $-K_X=\det(TX)$ satisfies $-K_X\cong L^{\otimes \tau}$ for some line bundle $L\to X$. Since $-K_X$ is ample by the Fano condition, $L$ is also ample, and thus the first Chern class $c_1^L$ of $L$ is represented by an integral K\"ahler form $\omega$. Assume that $L$ is very ample. Then, by the Bertini theorem, for every $k\in\N$, there is a smooth hypersurface $D$ such that $\mathcal{O}_X(D)\cong L^{\otimes k}$, and in turn $\PD([D])=k\,c_1(L)=k[\omega]$. If $k<\tau$, $D$ is also Fano by the adjunction formula:~$K_D\cong (K_X\otimes\mathcal{O}_X(D))|_D \cong
 L^{\otimes(k-\tau)}|_D$, and thus $-K_D\cong L^{\otimes(\tau-k)}|_D$ is ample. 
			
		We show that the manifolds $M$ appearing in (i), (ii), and (iii) satisfy the assumption of Lemma \ref{lem:fano}. Indeed, the restriction $\mathcal{O}(1)|_M$ of $\mathcal{O}(1)$ to $M\subset \mathbb{CP}^n$ is very ample (we use the Pl\"ucker embedding for $\mathrm{Gr}(k,r)$) and $-K_M\cong (O(1)|_M)^{\otimes \tau}$ with $\tau\geq 4$. Thus, $c_1^{TM}=\tau c_1^{\mathcal{O}(1)|_M}=\tau [\omega_\mathrm{FS}|_M]$, and there exists a Fano hypersurface $D$ satisfying $\PD([D])=2[\omega_\mathrm{FS}|_M]$. It remains to verify the eigenvalue condition. The quantum cohomology of $M=\mathrm{Gr}(k,r)$ with $\C$-coefficients is semi-simple, and thus the quantum multiplication $\star\,c_1^{TM}$ has a nonzero eigenvalue. For the manifolds $M$ in (ii), $\star\,c_1^{TM}$ has a nonzero eigenvalue by \cite{HKLY21,Ke23}. Case (i) is a special case of (ii), and it also follows from \cite[Corollary 27]{GI17} mentioned above.
				
		Let $(M,\omega)$ be a manifold in (iv). We write $M=M_1\times \cdots\times M_j$. It admits a Fano hypersurface $D$ satisfying $\PD([D])=2[\omega]$ since the tensor product of the pullback of the very ample line bundles of $M_i$ to $M$ is again very ample. Moreover, the quantum cohomology of $M$ is isomorphic to the tensor product of the quantum cohomologies of $M_i$.  Therefore, the sum of eigenvalues of $\star\,c_1^{TM_i}$ is an eigenvalue of $\star\,c_1^{TM}$. Since the spectrum of $\star\,c_1^{TM_i}$ is invariant under the multiplication of $e^{2\pi i/\tau_i}$, where $\tau_i$ is the Fano index of $M_i$, see \cite[Remark 3.1.3]{GGI16}, certain eigenvalues of $\star\, c_1^{TM_i}$ sum to a nonzero eigenvalue of $\star\, c_1^{TM}$.

		The manifolds in (v) have Fano index $\tau\leq 3$. The claim for  $Q^2=\mathbb{CP}^1\times \mathbb{CP}^1$ and $\mathbb{CP}^2$ follows from Propositions \ref{prop:q_2} and \ref{prop:quadric}. For $Q^3$, we can take a divisor $D\subset Q^3$ satisfying $\PD([D])=2[\omega_\mathrm{FS}|_{Q^3}]$ as a complete intersection of two quadrics in $\mathbb{CP}^4$. Then, as mentioned above, $c_1^{TD}$ has nonzero eigenvalue. Thus, Lemma \ref{lem:stab} can be applied to $(Q^3,D)$.
	\end{proof}

	The list of manifolds in Proposition  \ref{prop:list} is far from exhaustive. Note that if the conjecture that $\star\,c_1^{TD}$ has a nonzero eigenvalue for every Fano manifold $D$ holds, then the conclusion of Theorem \ref{thm:diameter} applies to every Fano manifold $M$ of Fano index $\tau\geq 3$ that admits a Fano hypersurface of degree $k\geq2$. 
	
	The spectral diameter of $\Ham(M,\omega)$ for $(M,\omega)$ in the list of Proposition  \ref{prop:list} with coefficients other than $\Z$ could be finite. As mentioned above, this is the case for $\mathbb{CP}^n$ with $n\geq 2$, and also for some Grassmannians, see \cite{AACS25}.

	It seems that it is unknown whether the spectral diameter of $\Ham(\mathbb{CP}^1,\omega_\mathrm{FS})$ is infinite. On the other hand, a quasi-isometric embedding of $(C^\infty_c((0,1)),\|\cdot\|_{C^0})$ into $\Ham(\mathbb{CP}^1,\omega_\mathrm{FS})$ with the Hofer norm was established in \cite{PS23}.

		An inspiring suggestion in \cite[Remark 3.2.4]{KS24} is to seek an alternative proof of their theorem that the spectral diameter of $\Ham(\mathbb{CP}^n,\omega_\mathrm{FS})$ for $n\geq 2$ over $\Z$ is infinite by finding a monotone Lagrangian torus  $L$ in $Q^{n-1}\subset \mathbb{CP}^n$ whose Lagrangian lift $\widetilde L$ to $\partial U_{\rho_\crit}\subset \mathbb{CP}^n$ is heavy. In the case of $\mathbb{CP}^3$, the Lagrangian torus $L_{1,1}^P$ from \cite{OU16} is such a Lagrangian lift. Recall that the skeleton $\mathrm{sk}(\mathbb{CP}^n\setminus Q^{n-1})=\mathbb{RP}^n$ is a heavy set disjoint from $\partial U_{\rho_\crit}$. Rather than finding such a Lagrangian torus $\widetilde{L}$, our approach is to show directly that $\partial U_{\rho_\crit}$ (indeed $\partial U_{\rho}$ for every $\rho\geq\rho_\crit$) is heavy.

		According to Theorem \ref{thm:sh}, $\SH_M^*(\partial U_\rho)$ can be computed whenever the quantum cohomology of $D$ is well understood. 
		In Section \ref{sec:computation}, we provide explicit computations for some examples. In view of Proposition \ref{prop:annuli}, the computation of $\SH_M^*(\partial U_\rho)$ also yields $\SH_M^*(\overline{U}_\rho)$ and $\SH_M^*(M\setminus U_\rho)$ (the latter under the assumption $\SH_M^*(\mathrm{sk}(M\setminus D))=0$).  For the standard ball $B_r\subset \mathbb{CP}^n$ of capacity $r>0$, $\SH_{\mathbb{CP}^n}^*(B_r)$ and $\SH_{\mathbb{CP}^n}^*(\mathbb{CP}^n\setminus B_r)$ were computed in \cite{DG26}, and our result recovers their computation. They worked over a filtered Novikov ring and computed restriction maps, and such persistence structures should also be computable in our setting. Furthermore, several interesting directions for relative symplectic cohomology are suggested in \cite[Section 1.7]{DG26}. While developed independently, our work addresses one of the questions posed there.

	\paragraph{Acknowledgments} 
This work is supported by National Research Foundation of Korea grants RS-2023-00211186, RS-2023-00301976, and RS-2025-02217108.
 We are grateful to Egor Shelukhin for Proposition \ref{prop:ham}. We thank Hiroshi Iritani for helpful correspondence regarding the eigenvalues of quantum multiplication by the first Chern class. We also thank Adi Dickstein, Yaniv Ganor, Egor Shelukhin, and Umut Varolgunes for helpful feedback.
 
	 \paragraph{AI disclosure} 
We used LLMs for  proofreading, literature searches, and simplifying algebraic computations in Propositions \ref{prop:q_2} and \ref{prop:quadric}.
	
	\section{Floer cohomology and relative symplectic cohomology}
Let $(M,\omega)$ be a closed symplectic manifold, which is spherically monotone, meaning that there exists $\tau>0$ such that $c_1^{TM}=\tau\omega$ on $\pi_2(M)$.
	\subsection{Floer cohomology}
	We briefly recall the construction of Hamiltonian Floer cohomology and refer to \cite{HS95} for details.
	Let $\mathcal{L}M$ be the space of contractible smooth 1-periodic loops in $M$. Let $\widetilde{\mathcal{L}M}$ denote the covering space of $\mathcal{L}M$ with deck transformation group $\pi_2(M)/\ker\omega$. We write elements in $\widetilde{\mathcal{L}M}$ as equivalence classes $[x,\bar x]$ where $x\in \mathcal{L}M$ and $\bar{x}$ is a capping disk of $x$, i.e.~$\bar{x}:D^2\to M$ with $\bar{x}(e^{2\pi it})=x(t)$ for $t\in S^1=\R/\Z$. Two pairs $(x,\bar x)$ and $(x',\bar x')$ are said to be equivalent if $x=x'$ and $\int_{\bar x}\omega=\int_{\bar{x}'}\omega$. The action functional associated with a smooth Hamiltonian $H:S^1\times M\to \R$ is defined by
	\[
		\mathcal{A}_H:\widetilde{\mathcal{L}M}\to \mathbb{R},\qquad 
		\mathcal{A}_H([x,\bar{x}])= \int_{D^2}\bar{x}^*\omega+\int_{0}^{1}H_t(x(t))dt.
	\]
	Then $[x,\bar{x}]$ is a critical point of $\mathcal{A}_H$ if and only if $\dot x(t)=X_{H_t}(x(t))$ holds, where the Hamiltonian vector field $X_{H_t}$ is defined by
		$dH_t=\omega(X_{H_t},\cdot)$.
	For $(a,b)\subset\R$, we denote by $\widetilde{\mathcal{P}}_{(a,b)}(H)$ the set of critical points $[x,\bar x]$ of $\mathcal{A}_H$ satisfying $a<\mathcal{A}_H([x,\bar{x}])<b$. Assume that $H$ is nondegenerate. 
	The Floer chain complex associated with $H$ is defined by 
	\begin{equation*}
		\CF^*_{(a,b)}(H,J)=\bigoplus_{[x,\bar{x}]\in\widetilde{\mathcal{P}}_{(a,b)}(H)}\mathbb{Z}\langle[x,\bar{x}]\rangle
	\end{equation*}
	with grading given by the Conley-Zehnder index $\CZ(x,\bar{x})\in\Z$, where $J=(J_t)_{t\in S^1}$ is an $S^1$-family of compatible almost complex structures used to define the boundary operator. We consider the moduli space  
	\[
	{\mathcal{M}}([x_-,\bar{x}_-],[x_+,\bar{x}_+],H,J),\qquad [x_\pm,\bar{x}_\pm]\in \widetilde{\mathcal{P}}(H)
	\] 
	that consists of smooth solutions $u:\R\times S^1\to M$ of the positive $L^2$-gradient flow equation
	\[
	\partial_su+J_t(u)(\partial_tu-X_{H_t}(u))=0
	\]
	with asymptotic conditions 
	\begin{equation}\label{eq:asymptotic_condition}
	\ev_\pm(u):=\lim_{s\rightarrow\pm\infty}u(s,t)=x_\pm(t),\qquad \int_{D^2} \bar{x}_-^*\omega+\int_{\mathbb{R}\times S^1} u^*\omega=\int_{D^2}\bar{x}_+^*\omega.
	\end{equation}
	The latter condition is equivalent to $[x_-,\bar{x}_-\#u]=[x_-,\bar{x}_+]$, where $\#$ denotes  gluing along their common boundary. 
	For a generic choice of $J$, these moduli spaces are cut out transversely and have dimension $\CZ(x_+,\bar{x}_+)-\CZ(x_-,\bar{x}_-)$. We denote by $\mathcal{J}_\reg(H)$ the space of such $J$. Unless $[x_-,\bar{x}_-]=[x_+,\bar{x}_+]$, there is a free $\mathbb{R}$-action on the moduli space given by translation in the $s$-direction. We denote the quotient space by $\overline{\mathcal{M}}([x_-,\bar{x}_-],[x_+,\bar{x}_+],H,J)$. It is a compact zero-dimensional manifold when $\CZ(x_+,\bar{x}_+)-\CZ(x_-,\bar{x}_-)=1$. The boundary operator $d_{(H,J)}:\CF^*_{(a,b)}(H,J)\rightarrow \CF^{*+1}_{(a,b)}(H,J)$ is defined by
	\begin{equation*}
		d_{(H,J)}([x_-,\bar{x}_-])=\sum\#\overline{\mathcal{M}}([x_-,\bar{x}_-],[x_+,\bar{x}_+],H,J)\cdot[x_+,\bar{x}_+],
	\end{equation*}
	where the sum ranges over all the generators $[x_+,\bar{x}_+]$ of $\CF^{*+1}_{(a,b)}(H,J)$. 
	The signed count $\#\overline{\mathcal{M}}([x_-,\bar{x}_-],[x_+,\bar{x}_+],H,J)\in\Z$ is taken with respect to a chosen system of coherent orientations as in \cite{FH93}. The homology of this complex $\HF^*_{(a,b)}(H)$ is called the Floer cohomology of $H$, and it is independent of the choice of $J$.

	Since the boundary operator increases the action, for $a_2\leq a_1$, the inclusion and the projection
	\begin{equation}\label{eq:incl_proj}
	\begin{split}
	\iota_{a_2,a_1}:\CF^*_{(a_1,\infty)}(H,J)\longrightarrow \CF^*_{(a_2,\infty)}(H,J),\\[1ex]
	\pi_{a_2,a_1}:\CF^*_{(-\infty,a_1)}(H,J)\longrightarrow \CF^*_{(-\infty,a_2)}(H,J)
	\end{split}
	\end{equation}
	are chain maps.

	Next, let $H_\pm:S^1\times M\to \R$ be two smooth Hamiltonians satisfying  $H_-\le H_+$. We choose a smooth homotopy $(\widetilde{H}_s)_{s\in\R}$ of smooth Hamiltonians $\widetilde{H}_s:S^1\times M\to \R$ such that $\widetilde{H}_{s}=H_\pm$ for $\pm s>1$ and $\partial_s \widetilde{H}_s \geq 0$. For $J_\pm\in\mathcal{J}_\reg(H_\pm)$, we also choose a smooth homotopy $(\widetilde J_s)_{s\in\R}$ of $S^1$-families of compatible almost complex structures such that $\widetilde J_s=J_\pm$ for $\pm s>1$. The moduli space
	\[
	{\mathcal{M}}([x_-,\bar{x}_-],[x_+,\bar{x}_+],\widetilde{H},\widetilde{J}),\qquad [x_\pm,\bar{x}_\pm]\in \widetilde{\mathcal{P}}(H_\pm)
	\]
	consists of smooth maps $u:\R\times S^1\to M$ solving 
	\[
	\partial_su+\widetilde{J}_{s,t}(u)(\partial_tu-X_{\widetilde{H}_{s,t}}(u))=0,\quad  
	\]
	with the same asymptotic conditions in \eqref{eq:asymptotic_condition}.
	For generic $\widetilde{H}$ and $\widetilde{J}$, it is a smooth manifold of dimension $\CZ(x_+,\bar{x}_+)-\CZ(x_-,\bar{x}_-)$. The continuation homomorphism $\Phi:\CF^*_{(a,b)}(H_-,J_-)\to \CF^*_{(a,b)}(H_+,J_+)$ is defined by 
	\[
	\Phi([x_-,\bar{x}_-])=\sum\#{\mathcal{M}}([x_-,\bar{x}_-],[x_+,\bar{x}_+],\widetilde{H},\widetilde{J})\cdot [x_+,\bar{x}_+]
	\]
	where the sum ranges over all the generators $[x_+,\bar{x}_+]$ of $\CF^*_{(a,b)}(H_+,J_+)$. By the condition $\partial_s \widetilde{H} \geq 0$, a Floer continuation cylinder increases the action and thus $\Phi$ is a chain map. Furthermore, $\Phi$ commutes with the maps $\iota_{a_2,a_1}$ and $\pi_{a_2,a_1}$ in \eqref{eq:incl_proj}.

	For the full action window $(a,b)=\R$, we simply write $\CF^*$ and $\HF^*$. In this case, the Floer chain complex admits a module structure over a Novikov ring $\Lambda=\Z[\mathfrak{t},\mathfrak{t}^{-1}]$ induced by
	\begin{equation}\label{eq:novikov_action}
	\Lambda\times \widetilde{\mathcal{P}}(H)\to \widetilde{\mathcal{P}}(H),\qquad \mathfrak{t}[x,\bar{x}]=[x,\bar{x}\#A],	
	\end{equation}
	where $A\in\pi_2(M)/\ker\omega$ such that $\omega(A)=m_M$. The constant $m_M\in\N$ is a positive  generator of $\omega(\pi_2(M))=m_M\Z$. By $\bar{x}\#A$, we mean a capping of $x$ induced by gluing $\bar{x}$ and a sphere representing $A$. The formal variable $\mathfrak{t}$ has degree $2c_M$, where $c_M$ denotes the minimal Chern number of $(M,\omega)$. By monotonicity, $c_M=\tau m_M$.
	If $\omega(\pi_2(M))=0$, we set $\Lambda=\Z$. 
	
	\subsection{Spectral norm}\label{sec:spec}
	Let $\QH^*(M;\Lambda)$ denote the quantum cohomology of $M$, that is isomorphic to $\H^*(M;\Z)\otimes\Lambda$ as $\Lambda$-module. Let $\mathrm{PSS}_H:\QH^*(M;\Lambda)\to \HF^*(H)$ be the PSS map established in \cite{PSS96}. It is a $\Lambda$-algebra isomorphism, where $\HF^*(H)$ is equipped with the pair of pants product. The spectral invariant of $H$ associated with a nonzero class $\eta\in \QH^*(M;\Lambda)$ is defined as 
	\[
	c(\eta;H):=\sup \big\{a\in \R \mid \mathrm{PSS}_H(\eta) \in \mathrm{im}\big((\iota_{-\infty,a})_*:\HF^*_{(a,\infty)}(H)\to \HF^*(H)\big)\big\}.
	\]
	Among many of its valuable properties, we collect some of them. 
	\begin{enumerate}[(i)]
		\item $c(\eta;H)\in \mathrm{spec}(H):=\mathcal{A}_H(\widetilde{\mathcal{P}}(H))$	
		\item $\displaystyle \int_{S^1}\min_M (H_t-K_t)dt\leq c(\eta;H)-c(\eta;K) \leq \int_{S^1}\max_M (H_t-K_t)dt$ 
		\item If $H$ and $K$ are normalized Hamiltonians with $[H]=[K]$, then $c(\eta;H)=c(\eta;K)$.  
		\item $c(\eta\star\mu;[H][K])\geq c(\eta;[H])+c(\mu;[K])$
	\end{enumerate}
	By (ii), we can define $c(\eta;H)$ even for degenerate $H$ (also for  continuous $H$). By being normalized we mean  $\int_{S^1}\int_M H_t\omega^ndt=0$ holds. We denote by $\widetilde{\Ham}(M,\omega)$  the universal cover of ${\Ham}(M,\omega)$, and write $[H]=[(\phi_H^t)_{t\in[0,1]}]\in\widetilde{\Ham}(M,\omega)$ and $[H][K]:=[(\phi_H^t\phi_K^t)_{t\in[0,1]})]$. The operation $\star$ denotes the quantum  product.
	By (iii), we can define
	\[
	c:\widetilde{\Ham}(M,\omega)\to \R ,\qquad c(\eta;[H]):=c(\eta;H),
	\]
	for normalized Hamiltonians $H$.

	For the unit element $1_M\in \QH^*(M)$, we simply write $c(H):=c(1_M;H)$.	
	The spectral (pseudo-)norm on $\widetilde{\Ham}(M,\omega)$ is defined by 
	\[
	\widetilde\gamma([H]):=-c([H])-c([\overline{H}])
	\]
	where $\overline{H}_t(x):=-H_t(\phi^t_H(x))$ is the Hamiltonian generating the path $t\mapsto (\phi_H^t)^{-1}$. It is an open problem whether $\widetilde\gamma([H])=0$ if and only if $[H]$ is the identity. This pseudo-norm  descends to a genuine norm on ${\Ham}(M,\omega)$:
	\[
	\gamma:{\Ham}(M,\omega)\to[0,\infty),\qquad \gamma(\phi):=\inf\{\widetilde\gamma([H])\mid \phi_H^1=\phi\}.
	\]
	We denote the induced bi-invariant metric by 
	\[
	 d_\gamma(\phi,\psi):=\gamma(\phi^{-1}\psi).
	\]
	By property (iv), it satisfies the triangle inequality. By property (ii), the Hofer-norm of $H$
	\[
	\|H\|_\mathrm{Hofer}:=\int_0^1 \big(\max_{M}H(t,\cdot)-\min_{M}H(t,\cdot)\big)dt
	\] 
	is bounded from below by $\widetilde{\gamma}([H])$, and therefore 
	\[
	\gamma(\phi)\leq \|\phi\|_{\mathrm{Hofer}}:= \inf\big\{ \|H\|_\mathrm{Hofer} \mid \phi_H^1=\phi\big\}.
	\]
	
 	A compact subset $K\subset M$ is called heavy or superheavy (with respect to $1_M$) if for any smooth autonomous function $H:M\to\R$
	\[
	\lim_{m\to\infty}\frac{c(mH)}{m} \leq \max_K H \quad\text{or}\quad \lim_{m\to\infty}\frac{c(mH)}{m} \geq \min_K H
	\] 
	holds, respectively. We remark that the spectral invariant $c$ and hence also the spectral norm $\gamma$ depend on the choice of the coefficient ring for Floer homology. As observed in \cite[Proposition 3.1.5]{KS24}, if $c_R$ denotes the spectral invariant with respect to a commutative unital ring $R$, then $c_R(H)\geq c(H)$ holds, where $c=c_\Z$. Thus, if $K$ is heavy with respect to $R$, then it is heavy with respect to $\Z$. Moreover, if $\gamma_R$ denotes the norm induced by $c_R$, then $\gamma\geq \gamma_R$ holds. The  inequality $c_R\geq c$ also implies that if $K$ is superheavy with respect to $\Z$, then it is superheavy with respect to $R$.
	
	A key observation we learned from Egor Shelukhin is that, in the spherically monotone case, $\widetilde\gamma([H])-\gamma(\phi_H^1)$ is uniformly bounded. Thus, the projection from $(\widetilde{\Ham}(M,\omega),\widetilde{\gamma})$ to $({\Ham}(M,\omega),\gamma)$ is a quasi-isometry.
	
	\begin{prop}\label{prop:ham}
		For every $H:S^1\times M\to\R$, we have $\widetilde\gamma([H])-\frac{2n m_M}{c_M}\leq \gamma(\phi_H^1)\leq \widetilde\gamma([H])$.
	\end{prop}
	\begin{proof}
	We consider the quantum cohomology  $\QH^*(M;\Lambda^\mathrm{univ})$ over the universal Novikov ring
	\begin{equation}\label{eq:universal}
			\Lambda^\mathrm{univ}:=\left\{\sum_{i=1}^\infty a_i\mathfrak{t}^{\lambda_i} \mid a_i\in\Z,\;\lambda_i\in\R,\; \lambda_i \nearrow\infty\right\}.
	\end{equation}
	We have $\QH^*(M;\Lambda)\hookrightarrow \QH^*(M;\Lambda^\mathrm{univ})$. The valuation 
		\[
			\nu:\Lambda^{\mathrm{univ}}\to \R\cup\{\infty\},\qquad \nu\left(\sum_{i=1}^\infty a_i\mathfrak{t}^{\lambda_i} \right):=\min\{\lambda_i\mid a_i\neq 0\}
		\] 	
		naturally induces a valuation $\nu:\QH^*(M;\Lambda^\mathrm{univ})\to \R$.
		For $[K]\in\pi_1(\Ham(M,\omega))$, we write the associated Seidel element $\mathcal{S}([K])\in\QH^*(M;\Lambda^\mathrm{univ})$, which is an invertible homogeneous element. By \cite[Proposition 4.1]{EP03}, we have
		\[
		c([K])=c(1_M;[K])=c(\mathcal{S}([K]);[0])=m_M\nu(\mathcal{S}([K])).
		\]
		As observed in \cite[Proposition 2.4]{MSV24}, the spectral invariant defined over $\Lambda$ coincides with the one over $\Lambda^\mathrm{univ}$.
		Let $d\in\R$ be the degree of $\mathcal{S}([K])=\sum \beta_i\otimes \mathfrak{t}^{\lambda_i}$, i.e.~$\deg \beta_i + 2c_M\lambda_i=d$ for all $i$. Since $0\leq\deg\beta_i\leq 2n$, we have $\frac{d-2n}{2c_M}\leq \lambda_i\leq \frac{d}{2c_M}$, and in turn
		\[
		\frac{d-2n}{2c_M}\leq \nu (\mathcal{S}([K])).
		\]
		Since $\mathcal{S}([\overline{K}])$ is the inverse of $\mathcal{S}([K])$, the degree of $\mathcal{S}[\overline{K}]$ is $-d$, and thus 
		\[
		\frac{-d-2n}{2c_M}\leq \nu (\mathcal{S}([\overline{K}])).
		\]
		Summing up these two inequalities, we obtain
		\[
		\widetilde\gamma([K])=-m_M\nu(\mathcal{S}([K]))-m_M\nu(\mathcal{S}([\overline{K}]))\leq \frac{2nm_M}{c_M}.
		\]
		This together with the triangle inequality yields
		\[
		\widetilde\gamma([H])\leq \widetilde\gamma([H][K])+\widetilde\gamma([\overline{K}]) \leq \widetilde\gamma([H][K])+\frac{2n m_M}{c_M}
		\]
		for any $H:S^1\times M\to\R$. The proposition follows from
		\[
		\widetilde\gamma([H])-\frac{2n m_M}{c_M}\leq \gamma(\phi_H^1)=\inf\{\widetilde\gamma([H][K])\mid  [K]\in\pi_1(\Ham(M,\omega))\} \leq \widetilde\gamma([H]).
		\]	
	\end{proof}
		
	\subsection{Relative symplectic cohomology}\label{sec:rel_sh}
	In this section, we sketch the construction of the relative symplectic cohomology of a compact subset $K\subset M$ and refer to \cite{Var21} for full details. We choose a sequence of nondegenerate Hamiltonians $H_1\leq H_2\leq \cdots $, where $H_i:S^1\times M\to\R$, such that 
	\[
	H_i|_{S^1\times K}<0,\qquad \lim_{i\to\infty} H_i(t,x)=
	\begin{cases}
		0 \quad &(t,x)\in S^1\times K,\\[.5ex]
		\infty&(t,x)\in S^1\times (M\setminus K).
	\end{cases}
	\]
	We also choose a monotone homotopy between $H_i$ and $H_{i+1}$ and a generic homotopy between $J_i\in\mathcal{J}_\reg(H_i)$ and $J_{i+1}\in\mathcal{J}_\reg(H_{i+1})$ for every $i\in\N$. The pair of sequences $(\mathcal{H},\mathcal{J})=(\{H_i\}_{i\in\N},\{J_i\}_{i\in\N})$, together with the homotopies, is called acceleration data for $K$.	This defines one-ray of Floer complexes 
	\[
	\begin{tikzcd}[column sep=20pt]
		\CF^*_{(a,b)}(H_1,J_1)\ar[r,"\Phi_1"]&\CF^*_{(a,b)}(H_2,J_2)\ar[r,"\Phi_2"]&\cdots
	\end{tikzcd}
	\]
	for any $-\infty\leq  a<b<\infty$, where the maps are continuation homomorphisms. We also view this as a direct system by defining maps between nonadjacent terms by compositions of $\Phi_i$. We denote the limit chain complex of this direct system by
	\[
	\SC^*_{(a,b)}(K,\mathcal{H},\mathcal{J}):=\varinjlim_{i\rightarrow\infty}\CF^*_{(a,b)}(H_i,J_i).
	\] 
 	Moreover, for any $b_1\geq b_2$, we have a commutative map between two one-rays 
	\[
	\begin{tikzcd}[column sep=20pt]
		\CF^*_{(a,b_1)}(H_1,J_1)\ar[r,"\Phi_1"]\ar[d]&\CF^*_{(a,b_1)}(H_2,J_2)\ar[r,"\Phi_2"]\ar[d]&\cdots\\
		\CF^*_{(a,b_2)}(H_1,J_1)\ar[r,"\Phi_1"]&\CF^*_{(a,b_2)}(H_2,J_2)\ar[r,"\Phi_2"]&\cdots
	\end{tikzcd}
	\]
	induced by projections in \eqref{eq:incl_proj}. This defines an inverse system of chain complexes
	\[
	\cdots\longrightarrow	\SC^*_{(a,b_j)}(K,\mathcal{H},\mathcal{J})\longrightarrow \SC^*_{(a,b_{j+1})}(K,\mathcal{H},\mathcal{J}) \longrightarrow\cdots.
	\]
	where $\cdots>b_j>b_{j+1}>\cdots$.
	Taking $a = -\infty$ and taking the inverse limit over $b_j \nearrow \infty$, we obtain the chain complex
	\begin{equation*}
		\SC^*_M(K,\mathcal{H},\mathcal{J})=\varprojlim_{b\rightarrow\infty}\SC^*_{(-\infty,b)}(K,\mathcal{H},\mathcal{J})=\varprojlim_{b\rightarrow\infty} \varinjlim_{i\rightarrow\infty}\CF^*_{(-\infty,b)}(H_i,J_i).
	\end{equation*}
	We write 
	\[
	\Phi_{l,i}=\Phi_{l-1}\circ \cdots \circ\Phi_i:\CF^*(H_i,J_i)\longrightarrow \CF^*(H_l,J_l)
	\]
	for the composition of continuation homomorphisms and 
	\[
	 \Phi_{\infty,i}:\CF^*(H_{i},J_{i})\longrightarrow \varinjlim_{l\rightarrow\infty}\CF^*(H_{l},J_{l})
	\] 
	for the induced map to the direct limit. Then, any element in $\SC_M^*(K,\mathcal{H},\mathcal{J})$ takes the form 
	\begin{equation}\label{eq:eltsofsc}
			\sum_{i=1}^\infty \Phi_{\infty,i}(\beta_i),\qquad \beta_i \in \CF^*(H_i,J_i),\quad a_i:=\lim_{l\to\infty}\mathcal{A}_{H_l}(\Phi_{l,i}(\beta_i))<\infty,\quad \lim_{i\to\infty} a_i=\infty,\quad 
	\end{equation}
	where $\mathcal{A}_{H_l}(\Phi_{l,i}(\beta_i))$ denotes the minimum action among the generators appearing with non-zero coefficients in the chain $\Phi_{l,i}(\beta_i)$.

	The relative symplectic cohomology of $K$ is defined by
	\begin{equation*}
		\SH^*_M(K)=\H^*(\SC_M^*(K,\mathcal{H},J)).
	\end{equation*}
	The product structure investigated in \cite{TV21,MSV24,AGV24} and the $\Lambda$-module structure induced by \eqref{eq:novikov_action} endow $\SH^*_M(K)$ with a $\Lambda$-algebra structure. Our definition of $\SH^*_M(K)$ is not identical to the original one \cite{Var21}, where the universal Novikov ring was used. We compare these two cohomologies in Remark \ref{rmk:original}.

	\begin{rmk}\label{rmk:Jchange}
	For another acceleration data $(\mathcal{H}',\mathcal{J}')=(\{H_i'\},\{J_i'\})$ with $H_i\leq H_i'$ for all $i\in\N$, we have a morphism between the two respective one-rays,
		\[
	\begin{tikzcd}[column sep=20pt]
		\CF^*_{(-\infty,b)}(H_1,J_1)\ar[r]\ar[d]\ar[rd]&\CF^*_{(-\infty,b)}(H_2,J_2)\ar[r]\ar[d]\ar[rd]&\cdots\\
		\CF^*_{(-\infty,b)}(H_1',J_1')\ar[r,]&\CF^*_{(-\infty,b)}(H_2',J_2')\ar[r,]&\cdots
	\end{tikzcd}
	\]
	induced by continuation homomorphisms, where the diagonal arrows indicate that each square commutes up to homotopy. This morphism induces a quasi-isomorphism from $\SC^*_M(K,\mathcal{H},\mathcal{J})$ to $\SC^*_M(K,\mathcal{H}',\mathcal{J}')$, 
		see \cite[Proposition 3.3.2]{Var21}. Even if $H_i\leq H_i'$ does not hold, the statement remains true since one can find a third acceleration data $(\mathcal{H''}=\{H_i''\},\mathcal{J}'')$ satisfying both $H_i\leq H_i''$ and $H_i'\leq H_i''$, which yields $\SC^*_M(K,\mathcal{H},\mathcal{J})\simeq \SC^*_M(K,\mathcal{H}'',\mathcal{J}'')\simeq \SC^*_M(K,\mathcal{H}',\mathcal{J}')$.
	\end{rmk}

	\begin{remark}\label{rem:limit} 
	We can also define the filtered relative symplectic cohomology of $K$
	\[
	\SH^*_{M,(a,\infty)}(K):=\H^*\Big(\varprojlim_{b\rightarrow\infty}\SC^*_{(a,b)}(K,\mathcal{H},\mathcal{J})\Big).
	\]
	One can readily see that the equalities 
			\begin{equation*}
		\SC^*_M(K,\mathcal{H},\mathcal{J}) =\varprojlim_{b\rightarrow\infty}\varinjlim_{a\rightarrow-\infty}\SC^*_{(a,b)}(K,\mathcal{H},\mathcal{J}) =\varinjlim_{a\rightarrow-\infty}\varprojlim_{b\rightarrow\infty}\SC^*_{(a,b)}(K,\mathcal{H},\mathcal{J})
	\end{equation*}
	hold by observing that each of these modules consists of the same elements, see  
	\cite[Lemmas 4.2 and 4.3]{CF11}. 
	Since the direct limit commutes with the homology functor, we conclude 
	\[
	\SH^*_M(K)\cong \varinjlim_{a\rightarrow-\infty}\SH^*_{M,(a,\infty)}(K).
	\]
	\end{remark}
	
	Let $K\subset K'$ be  compact subsets of $M$. By choosing acceleration data $(\mathcal{H}=\{H_i\},\mathcal{J})$ for $K$ and $(\mathcal{H}'=\{H_i'\},\mathcal{J}')$ for $K'$ satisfying $H_i'\leq H_i$ for all $i$, we have a sequence of continuation homomorphisms $\CF^*_{(-\infty,b)}(H_i')\to \CF^*_{(-\infty,b)}(H_i)$. This in turn yields the  restriction map 
	\[
	\mathrm{res}_{K,K'}:\SH_M^*(K')\longrightarrow\SH_M^*(K),
	\]
	which is a $\Lambda$-algebra homomorphism, see \cite{Var21,AGV24}. When $K'=M$, we have a $\Lambda$-algebra isomorphism $\QH^*(M;\Lambda) \cong \SH_M^*(M)$, see \cite{Var21,MSV24}. Composing this with the restriction map $\mathrm{res}_{K,M}$, we obtain a $\Lambda$-algebra homomorphism
	\begin{equation*}
		\mathrm{PSS}_K:\QH^*(M;\Lambda)\longrightarrow\SH_M^*(K).
	\end{equation*}

Apart from the following proposition which relies on a  chain-level product structure, we do not use the product structure on relative symplectic cohomology elsewhere in this paper.
	
	\begin{prop}\label{prop:heavy}
		A compact subset $K\subset M$ is heavy if and only if $\SH_M^*(K)\neq 0$.
	\end{prop}
	\begin{proof}
	This was proved in \cite[Theorem 1.7]{MSV24}. See also Remark \ref{rmk:original}. 
	\end{proof}

	\begin{rmk}\label{rmk:original}
			To explain how our definition of relative symplectic cohomology is related to the original one in \cite{Var21}, we recall the universal Novikov ring $\Lambda^{\mathrm{univ}}$ and the valuation  $\nu:\Lambda^{\mathrm{univ}}\to (-\infty,\infty]$ from \eqref{eq:universal}. 
			For $b\in\R$, we denote $\Lambda^{\mathrm{univ}}_{\geq b}:=\nu^{-1}([b,\infty])$. The chain complex introduced in \cite{Var21} is
			\begin{equation*}
				\SC^*_\mathrm{Var}(K,\mathcal{H},\mathcal{J};\Lambda^{\mathrm{univ}}):=\left(\varprojlim_{b\rightarrow\infty}\varinjlim_{i\rightarrow\infty}\CF^*(H_i;\Lambda^{\mathrm{univ}}_{\ge 0}/\Lambda^{\mathrm{univ}}_{\ge b})\right)\otimes_{\Lambda^{\mathrm{univ}}_{\ge 0}}\Lambda^{\mathrm{univ}},
			\end{equation*}
			where $\CF^*(H_i;\Lambda^{\mathrm{univ}}_{\ge 0}/\Lambda^{\mathrm{univ}}_{\ge b})$ is generated by $1$-periodic orbits of $X_{H_i}$ over $\Lambda^{\mathrm{univ}}_{\ge 0}/\Lambda^{\mathrm{univ}}_{\ge b}$. We write $\SH^*_{M,\mathrm{Var}}(K;\Lambda^\mathrm{univ})$ for the homology of this complex. 
			
			The action of $\Lambda=\Z[\mathfrak{t},\mathfrak{t}^{-1}]$  in \eqref{eq:novikov_action} defines a $\Lambda$-module structure on $\SC^*(K,\mathcal{H},\mathcal{J})$. Also, the filtered one $\SC^*_{(0,\infty)}(K,\mathcal{H},\mathcal{J})$ is a $\Z[\mathfrak{t}]$-module. We define
			\[
				\SC^*(K,\mathcal{H},\mathcal{J};\Lambda^{\mathrm{univ}}):
				=\left(\varprojlim_{b\rightarrow\infty}\SC^*_{(0,\infty)}(K,\mathcal{H},\mathcal{J})\otimes_{\Z[\mathfrak{t}]}(\Lambda^{\mathrm{univ}}_{\ge 0}/\Lambda^{\mathrm{univ}}_{\ge b})\right)\otimes_{\Lambda^{\mathrm{univ}}_{\ge 0}}\Lambda^{\mathrm{univ}}.
			\]
			We write $\SH^*_M(K;\Lambda^{\mathrm{univ}})$ for its homology. 
			Then, we have a chain isomorphism
			\[
				\SC^*(K,\mathcal{H},\mathcal{J};\Lambda^{\mathrm{univ}}) \longrightarrow \SC^*_\mathrm{Var}(K,\mathcal{H},\mathcal{J};\Lambda^{\mathrm{univ}})
			\]
			determined by $[x,\bar{x}] \mathfrak{t}^\lambda\mapsto x\,\mathfrak{t}^{\mathcal{A}_{H_i}(x,\bar{x})+m_M\lambda}$, and hence
			\begin{equation*}
				\SH^*_M(K;\Lambda^{\mathrm{univ}})\cong\SH^*_{M,\mathrm{Var}}(K;\Lambda^{\mathrm{univ}}).
			\end{equation*}
			
			We now observe the following equivalence:
			\[
			\SH^*_M(K;\Lambda^{\mathrm{univ}})\neq 0\quad \Longleftrightarrow \quad\SH^*_{M,\mathrm{Var}}(K;\Lambda^\mathrm{univ})\neq 0
			\]
			Since there is a natural $\Lambda$-linear inclusion 
			\[
				\SC^*(K,\mathcal{H},\mathcal{J})\to\SC^*(K,\mathcal{H},\mathcal{J};\Lambda^\mathrm{univ}),\qquad 
				\alpha\mapsto \alpha \mathfrak{t}^l \otimes \mathfrak{t}^{-l},
			\]
			where $l\gg0$ such that $\alpha\mathfrak{t}^l \in \SC^*_{(0,\infty)}(K,\mathcal{H},\mathcal{J})$, and a chain retraction 
			\[
				\SC^*(K,\mathcal{H},\mathcal{J};\Lambda^\mathrm{univ})\to\SC^*(K,\mathcal{H},\mathcal{J}),\qquad 
				\alpha\otimes \mathfrak{t}^\lambda \mapsto\begin{cases}
					\alpha\cdot \mathfrak{t}^\lambda &\lambda\in\mathbb{Z}\\
					0 &\mathrm{otherwise},
				\end{cases}
			\]
			it follows that $\SH_M^*(K)$ is a direct summand of $\SH_M^*(K;\Lambda^\mathrm{univ})$. Thus $\SH_M^*(K)\neq 0$ implies that $\SH^*_{M}(K;\Lambda^\mathrm{univ})\neq 0$. To show the converse, we consider the commutative diagram 
			\[\begin{tikzcd}
				\QH^*(M;\Lambda)\ar[rr]\ar[d,"\mathrm{PSS}_K"]&&\QH^*(M;\Lambda^\mathrm{univ})\ar[d,"\mathrm{PSS}_K^\mathrm{univ}"]\\
				\SH_M^*(K)\ar[r] &\SH_M^*(K;\Lambda^\mathrm{univ})\ar[r,"\cong"] &\SH^*_{M,\mathrm{Var}}(K;\Lambda^\mathrm{univ})
			\end{tikzcd}\]
			where the upper horizontal map is the inclusion. Suppose that $\SH^*_{M,\mathrm{Var}}(K;\Lambda^\mathrm{univ})\neq0$. Since $\mathrm{PSS}_K^\mathrm{univ}$ is a ring homomorphism, the unit $1_M\in \QH^*(M;\Lambda)\subset \QH^*(M;\Lambda^\mathrm{univ})$ is mapped to the unit $\mathrm{PSS}_K^\mathrm{univ}(1_M)\in \SH^*_{M,\mathrm{Var}}(K;\Lambda^\mathrm{univ})$, which is nonzero.  Therefore, $\mathrm{PSS}_K(1_M)\neq 0$, which implies $\SH_M^*(K)\neq 0$.
			\end{rmk}

	\section{Relative symplectic cohomology for circle bundles}\label{sec:sh}
	Now we describe the setting for our main results.  
	Let $(M,\omega)$ be a closed symplectic manifold of dimension $2n$. We assume that $\omega$ is integral, i.e.~there exists a lift $[\omega]_\Z\in\H^2(M;\Z)$ of $[\omega]\in \H^2(M;\R)$ and that $M$ is spherically monotone, i.e.~$c_1^{TM}=\tau\omega$ on $\pi_2(M)$ for some  $\tau>0$. We also assume that there exists a closed  codimension two symplectic submanifold $D\subset M$ whose homology class $[D]\in\H_{2n-2}(M;\Z)$ is Poincar\'e dual to $k[\omega]_\Z$ for some positive integer $k<\tau$. 
	
	We denote by $\omega_D$ and $[\omega_D]_\Z$ the pullbacks of $\omega$ and $[\omega]_\Z$ under the inclusion $D\hookrightarrow M$. 
	Let $\pi:\widehat{E}\to D$ be a complex line bundle with first Chern class $k[\omega_D]_\Z$. We equip $\widehat{E}$ with any Hermitian metric.  Let $\alpha$ be a connection 1-form on $\widehat{E}\setminus D$ satisfying $\pi^*(-k\omega_D)=d\alpha$. Let $E\subset \widehat{E}$ denote the open unit disk subbundle. We equip ${E}$ with the symplectic form \[
	\pi^*\omega_D+\frac{1}{k}d(r^2\alpha)
	\] 
	where $r$ denotes the fiberwise radial coordinate. The symplectic form admits the primitive 1-form $\frac{1}{k}(r^2-1)\alpha$ on $E\setminus D$.  The celebrated decomposition theorem of Biran \cite{Bir01} (see \cite{Gir17} for generalizations) yields that 
	\begin{itemize}
		\item there is a Liouville domain $(W,\lambda)$ with a smooth map $q:W\to M$ such that the interior $\mathrm{int}(W)$  of $W$ is symplectomorphic to $M\setminus D$ via $q$  and $q$ maps the boundary $\partial W$ to $D$, where $q:(\partial W,-k\lambda|_{\partial W})\to D$ is strictly contactomorphic to the prequantization bundle $\pi:(\partial E,\alpha)\to D$;
		\item there is a symplectomorphism $\psi:E\to M\setminus \mathrm{sk}(W,\lambda)$, where $\mathrm{sk}(W,\lambda)$ denotes the skeleton of the Liouville flow of $(W,\lambda)$.
	\end{itemize}
	Moreover, through the symplectomorphism $\mathrm{int}(W)\setminus \mathrm{sk}(W,\lambda) \cong E\setminus D$, the Liouville 1-form $\lambda$ corresponds to $\frac{1}{k}(r^2-1)\alpha$, and the Liouville vector field corresponds to $Z=(r^2-1)\partial_{r^2}$. The connection 1-form $\alpha$ restricted to each fiber of $E$ is $\frac{d\theta}{2\pi}$, where $(r,\theta)$ denote the polar coordinates, and accordingly the Reeb vector field $R_\alpha$ equals  $2\pi\partial_\theta$.  
	
	We define
	\[
	U_\rho:=\psi\big(\{e\in E\mid r^2(e)< \rho\}\big) \subset M.
	\] 
	To ease notation, we often identify $E$ with $U_1$ and regard $E$ as a subset of $M$. In this section, we introduce and investigate the Morse-Bott chain complex for the relative symplectic cohomology of the boundary $\partial U_\rho$ of $U_\rho$.
	
	\subsection{A specific choice of acceleration data}
	
	We choose a monotone sequence $\{H_\mu\}_{\mu\in\R_+\setminus\Z}$ of autonomous  Hamiltonians on $M$ as follows. We fix $0<\rho<1$. We choose $\epsilon>0$ such that $2\epsilon<1-\rho$ and   $\frac{\rho\mu}{\epsilon }\notin\Z$. Let $h_\mu^\mathrm{lin}:[0,1]\rightarrow \mathbb{R}$ be a piecewise linear function given by 
	\begin{equation*}
		h_\mu^\mathrm{lin}(r)=\begin{cases}
			-\frac{\mu}{k}(r^2-\rho) &r^2\le \rho,\\[.5ex]
			\frac{\rho}{\epsilon}\frac{\mu}{k}(r^2-\rho) &\rho \le r^2 \le \rho+\epsilon,\\[.5ex]
			\frac{\mu}{k} \rho &r^2\ge \rho+\epsilon.
		\end{cases}
	\end{equation*}
\begin{figure}[htbp]
	\centering
	\begin{tikzpicture}
	\def\rhoval{0.38}
	\def\epsval{0.18}
	\pgfmathsetmacro{\rhoeps}{\rhoval+\epsval}

	\begin{axis}[
		width=8cm,
		height=6cm,
		axis lines=middle,
		xmin=0, xmax=1.05,
		ymin=-0.12, ymax=1.2,
		xlabel={$r^2$},
		ylabel={},
		xtick={0,\rhoval,\rhoeps,1},
		xticklabels={$0$,$\rho$,$\rho+\epsilon$,$1$},
		ytick={0,1},
		yticklabels={$0$,$\frac{\mu}{k}\rho$},
		clip=false,
	]
		\addplot[very thick, black] coordinates {
			(0,1)
			(\rhoval,0)
			(\rhoeps,1)
			(1,1)
		};
		\pgfmathsetmacro{\xmid}{(\rhoeps+1)/2}
		\node[above=4pt] at (axis cs:\xmid,1) {\(\mathrm{I}\)};
		\node[above=4pt] at (axis cs:\rhoeps,1) {\(\mathrm{II}\)};
		\node[above=17pt] at (axis cs:\rhoval-0.02,0) {\(\mathrm{III}\)};
		\node[above right=4pt] at (axis cs:0,1) {\(\mathrm{IV}\)};
	\end{axis}
\end{tikzpicture}
\caption{graph of $h_\mu^\mathrm{lin}(r^2)$}
	\end{figure}
	
	Let $h_\mu$ be a smooth approximation of $h_\mu^\mathrm{lin}$ such that it vanishes on $r^2={\rho}$, and the function $r^2\mapsto h_\mu(r^2)$ is strictly convex near $r^2=\rho$ and strictly concave near $r^2=\rho+\epsilon$. The Hamiltonian $H_\mu:M\to\R$ is defined by
	\begin{equation}\label{eq:Hamiltonian}
	H_\mu(p)=\begin{cases}
		h_\mu(r^2(p)) & p\in E\cong M\setminus\mathrm{sk}(W,\lambda)\\[.5ex]
		\frac{\mu}{k}\rho & p \in \mathrm{sk}(W,\lambda).
	\end{cases}
	\end{equation} 
	Note that away from $\mathrm{sk}(W,\lambda)\cup D$ we have
	\begin{equation}\label{eq:Ham=-Reeb} 
	X_H=-k\partial_{r^2}h  R_\alpha.		
	\end{equation}

	Next, we introduce a special class of almost complex structures on $M$, cf.~\cite{DL19a,DL19b,AK23,BKK24}. Let $J_D$ be a compatible almost complex structure on $(D,\omega_D)$. We decompose $T\widehat{E}$ into vertical and horizontal tangent subbundles with respect to the connection form $\alpha$,
	\begin{equation*}
		T_{e}\widehat{E}=T_e^\mathrm{ver}\widehat{E}\oplus T_e^\mathrm{hor}\widehat{E}\cong \widehat{E}_{\pi(e)}\oplus T_{\pi(e)}D,
	\end{equation*}
	where  $\widehat{E}_{\pi(e)}$ denotes the fiber of $\widehat{E}\to D$ over $\pi(e)$. We consider a compatible almost complex structure $J_{\widehat{E}}$ on $\widehat{E}$, which has the form 
	$J_{\widehat{E}}=\begin{pmatrix}
			j &0\\
			0 & J_D
		\end{pmatrix}$
		with respect to the above decomposition,  
	where 
	\begin{equation}\label{eq:j_fiber}
		\begin{cases}
			j(e)=i &\text{for } r^2(e)<\frac{\rho}{3}\\[.5ex]
			j(e)Z=-R_\alpha &\text{for } r^2(e)>\frac{2\rho}{3}\\[.5ex]
			j(e)\partial_{r^2}=g(r)R_\alpha &\text{for } r(e)\neq 0
		\end{cases}
	\end{equation}
	for some smooth function $g:(0,\infty)\rightarrow (0,\infty)$. Here $i$ denotes the complex structure of $\widehat{E}\to D$, and recall $Z=(r^2-1)\partial_{r^2}$. We also consider a compatible almost complex structure $J$ on $M$ such that, on $U_{\rho+2 \epsilon}\subset U_1\cong E$, it satisfies the first and the last condition of \eqref{eq:j_fiber}.
	
	By the Liouville flow, $(W\setminus\mathrm{sk}(W,\lambda),\lambda)$ is exact symplectomorphic to $((0,1)\times\partial{W},\mathfrak{r}\lambda|_{\partial W})$, where the coordinate $\mathfrak{r}\in(0,1)$ corresponds to $1-r^2$.  We define the completion $\widehat{W}$ by attaching $((1,\infty)\times \partial{W},\mathfrak{r}\lambda|_{\partial W})$ to $W$. Let $J_{\widehat{W}}$ be a compatible almost complex structure on $\widehat{W}$ being cylindrical at infinity, i.e.,~outside a bounded subset, $J_{\widehat{W}}$ preserves $\ker\alpha$, is independent of $\mathfrak{r}$, and sends $Z$ to $-R_\alpha=R_{\lambda|_{\partial W}}$, where the latter is the Reeb vector field associated with $\lambda|_{\partial W}$.

	\begin{lem}\label{lem:mp}
		Let $(\R_+\times \Sigma, d(\mathfrak{r}\mathfrak{\sigma}))$ be the symplectization of a contact manifold $(\Sigma,\sigma)$. Let $J$ be a compatible almost complex structure satisfying $\sigma\circ J=f d\mathfrak{r}$ near $\{\mathfrak{r}=\mathfrak{r}_0\}$ for some positive function $f$. Let $H$ be a Hamiltonian on $\R_+\times \Sigma$ of the form $H(\mathfrak{r},z)=h(\mathfrak{r})$ near $\{\mathfrak{r}=\mathfrak{r}_0\}$. Let $u:\R\times S^1\to \R_+\times \Sigma$ be a Floer cylinder for $(H,J)$. 
		 If both asymptotic orbits of $u$ are contained in $\{\mathfrak{r}\leq \mathfrak{r}_0\}$, then $u$ is entirely contained in that region. This continues to hold for Floer continuation cylinders for a homotopy $\widetilde{H}_s(r,z)=h_s(r)$ provided $\partial_sh_s\geq0$ and $\partial_s(-\mathfrak{r}h'_s(\mathfrak{r})+h_s(\mathfrak{r}))\leq 0$ for $\mathfrak{r}$ near $\mathfrak{r}_0$.
	\end{lem}
	Lemma \ref{lem:mp} was proved in \cite[Lemma 2.2]{CO18} (originally in \cite[Lemma 7.2]{AS10}), up to a sign reversal coming from our  convention $\omega(X_H,\cdot)=dH$. 
	
	\begin{lem}\label{lem:ab}
		Assume we are in the setting of Lemma \ref{lem:mp}. Assume further that $H(\mathfrak{r},z)=h(\mathfrak{r})$ on $\R_+\times \Sigma$, and $J$ is independent of $\mathfrak{r}$, preserves $\ker \sigma$, and maps $\partial_\mathfrak{r}$ to the Reeb vector field of $\sigma$.
		Let $u:\R_\pm\times S^1\to \R_+\times \Sigma$ be a Floer cylinder for $(H,J)$ such that $\displaystyle\ev_\pm(u):=\lim_{s\to\pm\infty}u(s,\cdot)= x_\pm$ exist. Then, the following hold.  
		\begin{enumerate}[(a)]
			\item Assume $\partial_{\mathfrak{r}}^2h(\mathfrak{r}(x_-))<0$. Then either there exists $(s_0,t_0)\in \mathbb{R}\times S^1$ such that $\mathfrak{r}(u(s_0,t_0))>\mathfrak{r}(x_-)$, or $u$ is constant equal to $x_-$.
			\item Assume $\partial_{\mathfrak{r}}^2h(\mathfrak{r}(x_+))>0$. Then either there exists $(s_0,t_0)\in \mathbb{R}\times S^1$ such that $\mathfrak{r}( u(s_0,t_0))>\mathfrak{r}(x_+)$, or $u$ is constant equal to $x_+$.
		\end{enumerate}
	\end{lem}

  	Lemma \ref{lem:ab} was proved in \cite[Lemma 2.3]{CO18}. The sign conditions on  $\partial_{\mathfrak{r}}^2h(\mathfrak{r}(x_\pm))$ are reversed because of our convention \eqref{eq:Ham=-Reeb}. The lemma applies to a Floer continuation cylinder $u$, as it solves the ordinary Floer equation asymptotically. In this case, $u\equiv x_\pm$ in the statement is replaced by $u|_{(-\infty,-1)\times S^1}\equiv x_-$ and $u|_{(1,\infty)\times S^1}\equiv x_+$.

	\subsection{Morse-Bott chain complex}
	
	The Hamiltonian $H_\mu$ is constant on $M\setminus U_{\rho+\epsilon}$. We choose a smooth Morse function 
	\begin{equation}\label{eq:f_W}
	f_W:M\setminus U_{\rho+\epsilon} \longrightarrow \mathbb{R}		
	\end{equation}
	such that $\partial_r f_W>0$ near $\partial U_{\rho+\epsilon}$. We then perturb $H_\mu$ by adding $\epsilon f_W$ for $\epsilon>0$ and denote the new Hamiltonian again by $H_\mu$. We take $\epsilon$ small enough that $\epsilon f_W$ is $C^2$-small, possibly $\epsilon\searrow0$  as $\mu\nearrow\infty$.  Then, the space of 1-periodic orbits $\mathcal{P}(H_\mu)$ is of Morse-Bott type, and each orbit $x\in \mathcal{P}(H_\mu)$ falls into one of the following four types:
	\begin{enumerate}
		\item[(I)] constant orbits on $M\setminus U_{\rho+\epsilon}$ corresponding to $\crit f_W$,
		\item[(II)] 1-periodic orbits near $\partial U_{\rho+\epsilon}$,
		\item[(III)] 1-periodic orbits near $\partial U_\rho$, 
		\item[(IV)] constant orbits in $D$.
	\end{enumerate}
	 We then have the projection $\widetilde{\mathcal{P}}(H_\mu)\to \mathcal{P}(H_\mu)$ sending $[x,\bar{x}]$ to $x$. We write 
	 \[
	 \widetilde{\mathcal{P}}(H_\mu)= \widetilde{\mathcal{P}}_{\I}(H_\mu)	 \cup \widetilde{\mathcal{P}}_{\II}(H_\mu) \cup \widetilde{\mathcal{P}}_{\III}(H_\mu) \cup \widetilde{\mathcal{P}}_{\IV}(H_\mu)
	 \]
	 where $\widetilde{\mathcal{P}}_{\I}(H_\mu)$ is the subset consisting of type I capped orbits, and similarly for other types.
	 A connected component of type $\II$ or $\III$ capped orbits is diffeomorphic to $\partial U_\rho\cong\partial E $ via the evaluation map $[x,\bar{x}]\mapsto x(0)$. Similarly, a connected component of type $\IV$ capped orbits is diffeomorphic to $D$. 
	We additionally choose Morse functions 
	\[
	f_D:D\longrightarrow\R,\qquad f_{\partial E}:\partial E\longrightarrow \R
	\]
	together with Morse-Smale gradient-like vector fields. Recall that $\partial E\to D$ is a circle bundle. 
	We require that $f_{\partial E}$ has exactly two critical points $\hat{p}$ and $\check{p}$ on $\pi^{-1}(p)$ for each $p\in\crit f_D$, satisfying 
	\begin{equation}\label{eq:hat_check_p}
	\ind_{f_{\partial E}}(\hat p)=\ind_{f_D}(p)+1,\qquad \ind_{f_{\partial E}}(\check p)=\ind_{f_D}(p).		
	\end{equation}
	We refer to \cite[Section 5.1]{BKK24} for details. When this distinction is unnecessary, we simply write $q\in \crit f_{\partial E}$. Through the aforementioned diffeomorphisms, we regard these Morse functions as defined on the respective components of $\widetilde{\mathcal{P}}(H_\mu)$. While different choices of $f_{\partial E}$ and $f_D$ may be made on each component diffeomorphic to $\partial E$ or $D$, we write $f_{\partial E}$ and $f_D$  to keep the notation simple. We define
	\[
	\begin{split}
	\widetilde{\mathcal{P}}^\mathrm{crit}(H_\mu)&:=\{[x,\bar{x}] \in \widetilde{\mathcal{P}}(H_\mu) \mid [x,\bar{x}] \text{ is of type I or a critical point of $f_{\partial E}$ or $f_D$} \}\\[.5ex]
	&=\widetilde{\mathcal{P}}_{\I}(H_\mu)	 \cup \widetilde{\mathcal{P}}_{\II}^\mathrm{crit}(H_\mu) \cup \widetilde{\mathcal{P}}_{\III}^\mathrm{crit}(H_\mu) \cup \widetilde{\mathcal{P}}_{\IV}^\mathrm{crit}(H_\mu)
	\end{split}
	\]
	where $\widetilde{\mathcal{P}}_{\II}^\mathrm{crit}(H_\mu):=\widetilde{\mathcal{P}}_{\II}(H_\mu)\cap \widetilde{\mathcal{P}}^\mathrm{crit}(H_\mu)$, and similarly for other types.

	Every orbit $x\in {\mathcal{P}}(H_\mu)$ admits a canonical capping disk. If $x$ is of type I or IV, then it has a constant capping. If $x$ is of type II or III, then it is either a constant orbit or a (multiple cover of) a fiber circle of $M\setminus\mathrm{sk}(W,\lambda)\cong E$. In the former case, it has a constant capping, and in the latter case, it bounds a natural fiber disk $\bar{x}^\mathrm{fib}$. Thus, in all four cases, any capping disk can be expressed, up to homotopy, as the connected sum of this canonical capping with a sphere in some class $A\in\pi_2(M)/\ker \omega$. Thus, elements of $\widetilde{\mathcal{P}}^\mathrm{crit}(H_\mu)$ can be expressed as 
	\begin{equation}\label{eq:abbrev}
	[x,\bar{x}]=
	\begin{cases}
		[\mu,\I,q,A] & q\in \crit f_W\\[.5ex]
		[\mu,\II,q,m,A] & q\in \crit f_{\partial E}, \; -\frac{\rho}{\epsilon}\mu < m < -1 \\[.5ex]
		[\mu,\III,q,m,A] & q\in \crit f_{\partial E}, \; -\frac{\rho}{\epsilon}\mu<m<\mu \\[.5ex]
		[\mu,\IV,q,A] & q\in \crit f_D\\[.5ex]
	\end{cases}		
	\end{equation}
	where $m$ indicates the winding number of $x$ around the zero of the fiber $E_{\pi({q})}$. The allowed range of $m$ follows from \eqref{eq:Ham=-Reeb}.

 Let $\widetilde{\mathcal{P}}^\mathrm{crit}_{(a,b)}(H_\mu)$ be the subset of $\widetilde{\mathcal{P}}^\mathrm{crit}(H_\mu)$ consisting of $[x,\bar x]$ with $a<\mathcal{A}_{H_\mu}([x,\bar x])<b$. We define the Morse-Bott Floer chain complex
	\begin{equation*}
		\CF^*_{(a,b)}(H_\mu)=\bigoplus_{[x,\bar{x}]\in \widetilde{\mathcal{P}}^\mathrm{crit}_{(a,b)}(H_\mu)} \mathbb{Z}\langle [x,\bar{x}]\rangle
	\end{equation*}
	with grading 
	\begin{equation}\label{ind}
		\mu([x,\bar{x}]):=\CZ(x,\bar{x})+\ind_{f_X}([x,\bar{x}])-\left\lfloor \frac{\dim X}{2}\right\rfloor
	\end{equation}
	where $X=W$ if $[x,\bar{x}]\in \widetilde{\mathcal{P}}_{\I}(H_\mu)$, $X=D$ if $[x,\bar{x}]\in \widetilde{\mathcal{P}}_{\IV}^\crit(H_\mu)$, and $X=\partial E$ otherwise. Depending on its type, the action and index of $[x,\bar{x}]$ are computed as follows. 
	\begin{equation}\label{eq:action_index}
\begin{split}		
	\textrm{(I)}&\quad 	\begin{cases}
	\mathcal{A}_{H_\mu}([\mu,\I,q,A])=\omega(A)+H_\mu(q) \approx \omega(A)+\frac{\mu}{k}\rho \\[.5ex]
	\mu([\mu,\I,q,A])=2c_1^{TM}(A)+\ind_{f_W}(q)-n
\end{cases}\\[1ex]
		\textrm{(II)}&\quad  \begin{cases}
	\mathcal{A}_{H_\mu}([\mu,\II,q,m,A])=\frac{m}{k}r^2(x)+\omega(A)+H_\mu(x) \approx \frac{m}{k}(\rho+\epsilon)+\omega(A)+\frac{\mu}{k}\rho \\[.5ex]
	\mu([\mu,\II,q,m,A])=2m+2c_1^{TM}(A)+\ind_{f_{\partial E}}(q)-(n-1)
\end{cases}\\[1ex]
		\textrm{(III)}&\quad  	\begin{cases}
	\mathcal{A}_{H_\mu}([\mu,\III,q,m,A])=\frac{m}{k}r^2(x)+\omega(A)+H_\mu(x) \approx \frac{m}{k}\rho+\omega(A) \\[.5ex]
	\mu([\mu,\III,q,m,A])=2m+2c_1^{TM}(A)+\ind_{f_{\partial E}}(q)-(n-1)-1
\end{cases}\\[1ex]
	\textrm{(IV)}&\quad   \begin{cases}
	\mathcal{A}_{H_\mu}([\mu,\IV,q,A])=\omega(A)+H_\mu(q) \approx \omega(A)+\frac{\mu}{k}\rho \\[.5ex]
	\mu([\mu,\IV,q,A])=2\lfloor\mu\rfloor+1+ 2c_1^{TM}(A)+\ind_{f_D}(q)-(n-1)
\end{cases}
	\end{split}
	\end{equation}
	The term $-1$ in $\mu([\mu,\III,q,m,A])$ arises due to the fact that $r^2\mapsto h_\mu(r^2)$ is strictly convex at type III orbits.

	\begin{definition}\label{def:cascade}
		For $[x_\pm,\bar{x}_\pm]\in\widetilde{\mathcal{P}}^\mathrm{crit}(H_\mu)$, we denote by ${\mathcal{M}}_m([x_-,\bar{x}_-],[x_+,\bar{x}_+],H_\mu,J)$ the moduli space of flow lines with $m$ cascades for $m\in\N\cup\{0\}$. More precisely, it consists of tuples 
		\begin{equation*}
			\mathbf{u}=(c_0,u_1,c_1, u_2, \cdots, u_m, c_m)
		\end{equation*}
		such that
		\begin{enumerate}[(i)]
			\item $u_i$ are nontrivial (i.e.~$\partial_su_i\neq 0$) Floer cylinders for $(H_\mu,J)$ such that  $\displaystyle\ev_\pm(u_i)\in {\mathcal{P}}(H_\mu)$ exist;
			\item $c_i$ are gradient flow lines of $f_{\partial E}$ or $f_D$ such that 
		\begin{itemize}
		\item 	for $1\leq i\leq m-1$,  $c_i(0)=\ev_+(u_{i})$ and $c_i(t_i)=\ev_-(u_{i+1})$ for some $t_i\geq 0$; 
		\item $\displaystyle \lim_{s\to -\infty}c_0(s)=x_-$ and  $c_0(t_0)=\ev_-(u_1)$ for some $t_0\in\R$;
		\item $\displaystyle \lim_{s\to \infty}c_m(s)=x_+$ and  $c_m(t_m)=\ev_+(u_m)$ for some $t_m\in\R$;
		\item when $m=0$, i.e.~$\mathbf{u}=c_0$, $\displaystyle \lim_{s\to \pm\infty}c_0(s)=x_\pm$;
		\end{itemize}
		\item $[x_+,\bar{x}_+]=[x_+,\bar{x}_-\#\mathbf{u}]:=[x_-,\bar{x}_-\#c_0\#u_1\#c_1\#\cdots u_m\#c_m]$.
		\end{enumerate}

	\end{definition}
	For $m\geq 1$, there is a free $\R^m$-action on this moduli space given by translating each $u_i$ in the $\R$-direction. We denote the quotient space by 
		\begin{equation*}
			\overline{\mathcal{M}}_m([x_-,\bar{x}_-],[x_+,\bar{x}_+],H_\mu,J)={\mathcal{M}}_m([x_-,\bar{x}_-],[x_+,\bar{x}_+],H_\mu,J)/\mathbb{R}^m.
		\end{equation*}
		When $m=0$, there is an $\R$-action given by translating the domain $\R$ of $c_0$ unless $[x_-,\bar{x}_-]= [x_+,\bar{x}_+]$. We denote the quotient space by $\overline{\mathcal{M}}_0([x_-,\bar{x}_-],[x_+,\bar{x}_+],H_\mu,J)$. Suppose that $J$ is chosen so that ${\mathcal{M}}_m([x_-,\bar{x}_-],[x_+,\bar{x}_+],H_\mu,J)$ is cut out transversely. Then, the quotient space is a smooth manifold of dimension $\mu([x_+,\bar{x}_+])-\mu([x_-,\bar{x}_-])-1$. 
		 We denote
	\[
	\begin{split}
	{\mathcal{M}}([x_-,\bar{x}_-],[x_+,\bar{x}_+],H_\mu,J)&:=  \bigcup_{m\in\N\cup\{0\}} {\mathcal{M}}_m([x_-,\bar{x}_-],[x_+,\bar{x}_+],H_\mu,J),\\
		\overline{\mathcal{M}}([x_-,\bar{x}_-],[x_+,\bar{x}_+],H_\mu,J)&:= \bigcup_{m\in\N\cup\{0\}} \overline{\mathcal{M}}_m([x_-,\bar{x}_-],[x_+,\bar{x}_+],H_\mu,J).		
	\end{split}
	\]	
	The latter space is compact when $\mu([x_+,\bar{x}_+])-\mu([x_-,\bar{x}_-])=1$, and each element is assigned a sign induced by a coherent orientation, see \cite{BKK24}. 
	We define the boundary operator on this complex by
	\begin{align*}
		d_{(H_\mu,J)}:\CF^*_{(a,b)}(H_\mu,J)&\longrightarrow \CF^{*+1}_{(a,b)}(H_\mu,J)\\[.5ex]
			[x_-,\bar{x}_-]&\longmapsto\sum\#\overline{\mathcal{M}}([x_-,\bar{x}_-],[x_+,\bar{x}_+],H_\mu,J)\cdot[x_+,\bar{x}_+],
	\end{align*}
	where the sum ranges over all the generators $[x_+,\bar{x}_+]$ of $\CF^{*+1}_{(a,b)}(H_\mu,J)$. 
	
	\begin{rmk}\label{rmk:zerocap}
		For any $A\in\pi_2(M)/\ker\omega$, we have		
		\begin{equation*}
			\mathcal{M}([x_-,\bar{x}_-],[x_+,\bar{x}_+],H_\mu,J)=\mathcal{M}([x_-,\bar{x}_-\#A],[x_+,\bar{x}_+\#A],H_\mu,J).
		\end{equation*}
	\end{rmk}

	Continuation homomorphisms are defined analogously. Let $\mu_\pm\in \R_+\setminus\Z$ with $\mu_-<\mu_+$. We choose a smooth homotopy $(\widetilde{H}_s)_{s\in\R}$ of autonomous Hamiltonians such that, for each $s\in\R$, $\widetilde{H}_s$ takes the form in \eqref{eq:Hamiltonian}, $\widetilde{H}_s=H_{\mu_\pm}$ for $\pm s\geq1$ and $\partial_s\widetilde{H}_s\geq 0$. 
	Let $J_\pm$ be compatible almost complex structures associated with $H_{\mu_\pm}$. 
	We choose a homotopy $(\widetilde{J}_s)_{s\in\R}$ of compatible almost complex structures such that, for each $s\in\R$, $J_s$ satisfies the properties required below \eqref{eq:j_fiber} and $\widetilde{J}_s=J_\pm$ for $\pm s\geq 1$.
	For $m\in\N$, the moduli space
	\[
	{\mathcal{M}}_m([x_-,\bar{x}_-],[x_+,\bar{x}_+],\widetilde{H},\widetilde{J}),\qquad [x_\pm,\bar{x}_\pm]\in \widetilde{\mathcal{P}}^\crit(H_{\mu_\pm})
	\]
	consists of $\mathbf{u}=(c_0,u_1,c_1, u_2, \cdots, u_m, c_m)$ as in Definition \ref{def:cascade}, except that there exists $1\leq m_0\leq m$ such that
	\begin{itemize}
		\item $u_1,\dots,u_{m_0-1}$ are Floer cylinders with respect to $(H_{\mu_-},J_-)$ 
		\item  $u_{m_0+1},\dots,u_{m}$ are Floer cylinders with respect to $(H_{\mu_+},J_+)$
		\item $u_{m_0}$ is a Floer continuation cylinder with respect to $(\widetilde{H},\widetilde{J})$. 
	\end{itemize}
	If $\widetilde{J}$ is chosen so that this moduli space is cut out transversely, then it is a smooth manifold of dimension $\mu([x_+,\bar{x}_+])-\mu([x_-,\bar{x}_-])+m-1$. 
	There is a free $\R^{m-1}$-action on this moduli space given by translating each $u_i$, except $u_{m_0}$. We denote the quotient space by $\overline{\mathcal{M}}_m([x_-,\bar{x}_-],[x_+,\bar{x}_+],\widetilde{H},\widetilde{J})$. Let 
	\[
	\begin{split}
	{\mathcal{M}}([x_-,\bar{x}_-],[x_+,\bar{x}_+],\widetilde{H},\widetilde{J})&:=\bigcup_{m\in\N}{\mathcal{M}}_m([x_-,\bar{x}_-],[x_+,\bar{x}_+],\widetilde{H},\widetilde{J}),\\
	\overline{\mathcal{M}}([x_-,\bar{x}_-],[x_+,\bar{x}_+],\widetilde{H},\widetilde{J})&:=\bigcup_{m\in\N}\overline{\mathcal{M}}_m([x_-,\bar{x}_-],[x_+,\bar{x}_+],\widetilde{H},\widetilde{J}).		
	\end{split}
	\]
	We define a continuation homomorphism by
	\[
	\begin{split}
	\CF^*_{(a,b)}(H_{\mu_-},J_-)&\longrightarrow \CF^{*}_{(a,b)}(H_{\mu_+},J_+)\\
	[x_-,\bar{x}_-]&\longmapsto\sum\# \overline{\mathcal{M}}([x_-,\bar{x}_-],[x_+,\bar{x}_+],\widetilde H,\widetilde J)\cdot[x_+,\bar{x}_+]		
	\end{split}
	\]
	where the sum ranges over all the generators $[x_+,\bar{x}_+]$ of $\CF^{*}_{(a,b)}(H_{\mu_+},J_+)$.

	In order to define the relative symplectic cohomology of $\partial U_\rho$ as in Section \ref{sec:rel_sh}, we choose  acceleration data $\mathcal{H}=\{H_i\}_{i\in\N}$, where $H_i=H_{\mu_i}$ with $\mu_i<\mu_{i+1}$ in $\R_+\setminus\Z$ and $\mathcal{J}=\{J_i\}_{i\in\N}$, together with homotopies. We then define
	\begin{equation*}
		\SC^*(\partial U_\rho,\mathcal{H},\mathcal{J}):=\varprojlim_{b\rightarrow\infty}\SC^*_{(-\infty,b)}(\mathcal{H},\mathcal{J})=\varprojlim_{b\rightarrow\infty} \varinjlim_{i\rightarrow\infty}\CF^*_{(-\infty,b)}(H_i,J_i)
	\end{equation*}
	and denote	
	\begin{equation*}
		\SH^*_M(\partial U_\rho)=\H^*(\SC^*(\partial U_\rho,\mathcal{H},\mathcal{J})).
	\end{equation*}
	
	\begin{remark}\label{rmk:transversality}
	For computational purposes, we consider 
	the restricted class of $J$ specified  below \eqref{eq:j_fiber}.
	We now explain how to achieve transversality for ${\mathcal{M}}_m([x_-,\bar{x}_-],[x_+,\bar{x}_+],H_\mu,J)$ within this class. The argument is an adaptation of the methods in \cite{DL19a,DL19b} and also in \cite{BKK24}. Assume that $\mu([x_+,\bar{x}_+])-\mu([x_-,\bar{x}_-])=1$ and that $J$ is sufficiently stretched along $\partial U_{\rho+2\epsilon}$ as described in Section \ref{sec:neckstr}. Let $\mathbf{u}$ be an element in this moduli space, and let $u$ be a Floer cylinder in $\mathbf{u}$. If $u$ enters $M\setminus U_{\rho+2\epsilon}$, then we can achieve transversality by perturbing $J$ in that region. If $u$ is contained in $U_{\rho+2\epsilon}$,  the projection $\pi\circ u:\R\times S^1\to D$ is a $J_D$-holomorphic cylinder. By the monotonicity condition, $\pi\circ u$ is either simple or constant. In the former case, $u$ is cut out transversely in the horizontal direction by the simpleness of $\pi\circ u$ and a generic choice of $J_D$, while transversality in the vertical direction holds automatically. In the latter case, $u$ is a Floer cylinder contained in a fiber $E\cong U_1\to D$, and it is also automatically cut out transversely. This proves that the boundary operator $d_{(H_\mu,J)}$ is well-defined. Transversality for the continuation homomorphisms is achieved in a similar manner. 
	
	We choose a generic time-dependent almost complex structure $J'$, sufficiently close to $J$. By standard cobordism arguments, the count
	$\#\overline{\mathcal{M}}([x_-,\bar{x}_-],[x_+,\bar{x}_+],H_\mu,J)$ coincides with the corresponding count for $J'$. Since $d_{(H_\mu,J')}\circ d_{(H_\mu,J')}=0$, we also have $d_{(H_\mu,J)}\circ d_{(H_\mu,J)}=0$.
	\end{remark}

	The specific forms of $H_\mu$, $\widetilde{H}_s$, and $J$ near $D$ play a crucial role in the following lemmas. 
	\begin{lem}\label{lem:poi}
	Let $u$ be a Floer cylinder for $(H_\mu,J)$ or a Floer continuation cylinder for $(\widetilde{H},\widetilde{J})$. Then, the local intersection number of $u$ and $D$ at any intersection point is positive. 
	\end{lem}
	We refer to \cite[Lemma 4.3]{GP21} or \cite[Lemma 4.2]{Ton19} for the proof. For the following lemma, see \cite[Equation (7.22)]{Sei21}, \cite[Lemma 4.3]{Ton19}, or \cite[Proposition C.3]{DG26}.
	
	\begin{lem}\label{lem:iaa}
		Let $u:\R_\pm\times S^1\to M$ be a Floer cylinder for $(H_\mu,J)$ such that $\ev_\pm(u)$ is a type IV orbit. Then, the intersection number of (the closure of) $u$ and $D$ at $\ev_\pm(u)$ is greater than or equal to 
		\begin{enumerate}[(a)]
			\item $-\lfloor\mu\rfloor$ if $u:\R_+\times S^1\to M$,
			\item $\lfloor\mu\rfloor+1$  if $u:\R_-\times S^1\to M$.
		\end{enumerate}
	\end{lem}
	
	Using these two lemmas, we have the following topological constraint on $\mathbf{u}$.
	\begin{lem}\label{lem:intersect}
		Let $\mathbf{u}=(c_0,u_1,c_1,\dots,u_m,c_m)$ be an element of ${\mathcal{M}}_m([x_-,\bar{x}_-],[x_+,\bar{x}_+],\widetilde{H},\widetilde{J})$ for $[x_\pm,\bar{x}_\pm]\in\widetilde{\mathcal{P}}^\mathrm{crit}(H_{\mu_\pm})$, where $\mu_-\le \mu_+$.
		\begin{enumerate}
			\item[(a)] If $[x_\pm,\bar{x}_\pm]$ are not of type $\IV$, then
			$\mathbf{u}\cdot D\geq 0$. In addition, if $\mathbf{u}\cap D\neq \emptyset$, then $\mathbf{u}\cdot D\geq 1$.
			\item[(b)] If only $[x_+,\bar{x}_+]$ is of type $\IV$, then
			$\mathbf{u}\cdot D\geq-\lfloor\mu_+\rfloor$. 
			\item[(c)] If $[x_\pm,\bar{x}_\pm]$ are of type $\IV$ and $\mu_+\geq \mu_-+1$, then
			$\mathbf{u}\cdot D\geq\lfloor\mu_-\rfloor+1-\lfloor\mu_+\rfloor$.
		\end{enumerate}
		Here, $\mathbf{u}\cdot D$ denotes the sum of the intersection numbers of the closures of $u_j$ with $D$.
	\end{lem}
	
	\begin{proof}
		We first prove (a). If $u_j$ is entirely contained in $D$, then it extends to a $J_D$-holomorphic sphere, and thus
		\[
		u_j\cdot D=k\,\omega([u_j])\geq 0.
		\]
		If $u_j$ is not contained in $D$, then the local
		intersection numbers of $u_j$ and $D$ at the points of $u_j^{-1}(D)$ are positive
		by Lemma~\ref{lem:poi}. Suppose that $\mathbf{u}$ contains a Floer (continuation) cylinder not contained in $D$ whose positive end converges to a type IV orbit of $H_{\mu'}$ for $\mu'\in \{\mu_\pm\}$. Then it is followed,
		possibly through cylinders contained in $D$, by another Floer (continuation) cylinder not contained in $D$ whose negative asymptotic orbit is a 
		type IV orbit of $H_{\mu''}$, where
		$\mu''\in \{\mu_\pm\}$ and $(\mu',\mu'')\neq(\mu_+,\mu_-)$. By Lemma~\ref{lem:iaa}, the two ends together contribute at least
		$\lfloor\mu''\rfloor+1-\lfloor\mu'\rfloor\geq 1$.	This proves (a).
		
		In case (b), there is a Floer (continuation) cylinder whose positive end converges to a type $\IV$ orbit of $H_{\mu'}$ for $\mu'\in\{\mu_\pm\}$, which  is not followed by a Floer cylinder not contained in $D$. The intersection number of this cylinder with $D$ is at least $-\lfloor \mu'\rfloor$ by Lemma \ref{lem:iaa}. As observed in (a), all the other cylinders contribute to the intersection number nonnegatively to the intersection number. Therefore,
		\begin{equation*}
			\mathbf{u}\cdot D \ge -\lfloor \mu'\rfloor \ge -\lfloor\mu_+\rfloor.
		\end{equation*}
		
		For (c), if $\mathbf{u}$ contains a Floer (continuation) cylinder not contained in $D$, then as in (a), we have
		\begin{equation*}
			\mathbf{u}\cdot D\geq\lfloor\mu_-\rfloor+1-\lfloor\mu_+\rfloor.
		\end{equation*}
		Otherwise, all the cylinders in $\mathbf{u}$ are contained in $D$, and thus we get
		\begin{equation*}
			\mathbf{u}\cdot D\ge 0 \ge \lfloor\mu_-\rfloor+1-\lfloor\mu_+\rfloor.\qedhere
		\end{equation*}
	\end{proof}

		\subsection{Floer continuation cylinders  in fibers}
	
		In this section, we investigate Floer continuation cylinders that are entirely contained in a fiber of $E$. Such cylinders can be identified as Floer continuation cylinders in $(\C,j)$ with respect to the Hamiltonian 
		\[
		G_\mu:\mathbb{C}\to \mathbb{R},\qquad G_\mu(z):=h_\mu(|z|^2),
		\] 
		which can be thought of as the restriction of $H_\mu$ to a fiber of ${E}$. The almost complex structure $j$ is given in \eqref{eq:j_fiber}. For a 1-periodic orbit  $x$ of $G_\mu$ of type II or III, we denote by $\theta(x(0))\in S^1$ the angular coordinate of $x(0)$. We also define $(\tau\cdot x)(t):=x(t+\tau)$ and choose a perfect Morse function 
		\[
		f_x:S_x:=\{\tau\cdot x\mid \tau\in S^1\}\longrightarrow\R.
		\]
		
		\begin{lemma}\label{lem:fiber_d}
			Let $x_-$ and $x_+$ be 1-periodic orbits of $G_\mu$ of type III and  II, respectively. Assume that $x_\pm$ have the same winding number around the origin. Then, the moduli space 
		\begin{equation*}
			\mathcal{M}(S_{x_-},S_{x_+},{G}_\mu,j):=\left\{u:\mathbb{R}\times S^1\rightarrow \mathbb{C}\;\middle|\;
			\begin{matrix*}[l]
				\partial_su+j(u)(\partial_tu-X_{{G}_\mu}(u))=0,\\[.5ex]
				\ev_\pm(u)\in S_{x_\pm}
			\end{matrix*}
			\right\}
		\end{equation*}
		is a smooth 2-dimensional manifold, each of whose connected components is diffeomorphic to $\R\times S^1$:~more precisely, if $u\in \mathcal{M}(S_{x_-},S_{x_+},{G}_\mu,j)$, then $u(\cdot+\sigma,\cdot+\tau)$ also belongs to the moduli space for every $(\sigma,\tau)\in\R\times S^1$. Furthermore, $\theta(\ev_-(u)(0))=\theta(\ev_+(u)(0))$ holds.
		\end{lemma}
		\begin{proof}
		Since $x_\pm$ have the same winding number, $u$ is contained in $\C\setminus\{0\}$.				We follow arguments in \cite[Proposition 4.7]{BKK24}. 
		 Linearizing the Floer equation at $u$, we obtain a Fredholm operator $D_u$ such that $\ker D_u=T_u\mathcal{M}(S_{x_-},S_{x_+},{G}_\mu,j)$. It has index 2 and is surjective due to automatic transversality. In contrast to \cite[Proposition 4.7]{BKK24} where the corresponding operator has index 1, $D_u$ has index 2 since $x_+$ is of type II. 
		 
		 There are three obvious vectors in $\ker D_u$, arising from rotating $u$ in the target and shifting in the domain. Specifically, since the families $e^{i\theta}u(\cdot+\sigma,\cdot+\tau)$ remain in $\mathcal{M}(S_{x_-},S_{x_+},{G}_\mu,j)$, the corresponding tangent vectors  $\partial_\theta|_u$, $\partial_su$, and $\partial_t u$ belong to $\ker D_u$. Thus $\ker D_u=\mathrm{span}\{\partial_su, \partial_tu\}$, and $\partial_\theta|_u$ and $\partial_t u$ are linearly dependent. We write 
		 $u=(a,v):=(r^2(u),\arg u):\R\times S^1\to (0,\infty)\times \R/2\pi\Z$. By the asymptotic condition of $u$, we have
		\begin{equation}\label{eq:p_tu}
			\partial_tu=(\partial_t a,\partial_t v) =2\pi m\cdot \partial_\theta|_u.
		\end{equation}
		where $m$ denotes the winding number of $x_\pm$. In particular, $a$ is independent of $t$.
		The Floer equation then implies $v$ is independent of $s$, see \eqref{eq:floer_eq_in_C} below. Therefore, $v(s,t)=v(t)=v(0)+{2\pi i m t}$ by \eqref{eq:p_tu}. This yields $\theta(\ev_-(u)(0))=\theta(\ev_+(u)(0))$. It also follows from arguments below \eqref{eq:floer_eq_in_C}. This completes the proof. 
		\end{proof}

		 Let $\widetilde{G}_s(z)=h_s(|z|^2)$ be a monotone homotopy between $G_{\mu_-}$ and $G_{\mu_+}$ with $\mu_-\leq \mu_+$.  For  1-periodic orbits $x_\pm$ of $X_{G_{\mu_\pm}}$, we define the moduli space $\mathcal{M}(S_{x_-},S_{x_+},\widetilde{G},j)$ of Floer continuation cylinders as above with $G_\mu$ replaced by $\widetilde{G}$.
		 
	\begin{lem}\label{lem:constconti}
		Assume that 1-periodic orbits $x_\pm$ of $X_{G_{\mu_\pm}}$ are either both of type II or both of type III, and that they have the same winding number around the origin. Then, the moduli space $\mathcal{M}(S_{x_-},S_{x_+},\widetilde{G},j)$
		is a smooth 1-dimensional manifold diffeomorphic to $S^1$: more precisely, $\mathcal{M}(S_{x_-},S_{x_+},\widetilde{G},j)=\{e^{i\theta} u\,|\,\theta\in S^1\}$ for a solution $u$ with $\theta(\ev_-(u)(0))=\theta(\ev_+(u)(0))$.  Moreover, if  both $x_\pm$ are of type II, then $u|_{(-\infty,-1)\times S^1}$ is constant in $s$. If both $x_\pm$ are of type III, then $u|_{(1,\infty)\times S^1}$ is constant in $s$.
	\end{lem}
		The lemma follows from \cite[Proposition 4.7]{BKK24}, where the type III case is proved. Our choice of $j$, which is not necessarily the standard complex structure, causes no issue in the argument. The moreover-part is given in  \cite[Equation (4.11)]{BKK24}, up to a sign reversal coming from our cohomological convention.
	
	\begin{lem}\label{lem:conti_empty}
	Let $x_\pm$ be 1-periodic orbits of $X_{G_{\mu_\pm}}$. If $x_-$ has type $\II$ and $x_+$ has type $\III$ or type $\IV$ (i.e.~the constant orbit at the origin), then $\mathcal{M}(S_{x_-},S_{x_+},\widetilde{G},j)$ is empty. 
	\end{lem}
	\begin{proof}
		By Lemmas \ref{lem:mp} and \ref{lem:ab}, every $u\in\mathcal{M}(S_{x_-},S_{x_+},\widetilde{G},j)$ restricted to $(-\infty,-1)\times S^1$ is constant equal to $\ev_-(u)$. By the Carleman similarity principle, it coincides with the one in Lemma \ref{lem:constconti}, whose positive asymptotic orbit is of type II. This contradiction completes the proof.
	\end{proof}
	
	For 1-periodic orbits $x_-\in \crit f_{x_-}$ of $X_{G_{\mu_-}}$ and $x_+\in \crit f_{x_+}$ of $X_{G_{\mu_+}}$, we write as before 
		$			
			\mathcal{M}(x_-,x_+,\widetilde{G},j)
		$
		 for the moduli space of flow lines with cascades from $ x_-$ to $x_+$. We choose  generic $ f_{x_\pm}$ so that it is cut out transversely.  
			
	\begin{lem}\label{lem:addR2}
		Let $x_-$ be a 1-periodic orbit of $X_{G_{\mu_-}}$ of type III, and let $x_+$ be a 1-periodic orbit of $X_{G_{\mu_+}}$ of type II. Assume that $x_-\in \crit f_{x_-}$ is a maximum point and $x_+\in \crit f_{x_+}$ is a minimum point. Assume also that $x_\pm$ have the same winding number  around the origin. Then, $f_{x_\pm}$ can be chosen so that $\mathcal{M}(x_-,x_+,\widetilde{G},j)$ is empty. 
	\end{lem}
	\begin{proof}
		The moduli space $\mathcal{M}(S_{x_-},S_{x_+},\widetilde{G},j)$ can be analyzed following  arguments in \cite[Proposition 4.7]{BKK24} as in Lemma \ref{lem:fiber_d}; for any $u\in \mathcal{M}(S_{x_-},S_{x_+},\widetilde{G},j)$, the linearized  operator $D_u$ is surjective and has index 2. Accordingly $\mathcal{M}(x_-,x_+,\widetilde{G},j)$ is zero-dimensional.
		
		We show $\theta(\ev_-(u)(0))=\theta(\ev_+(u)(0))$ for any $u\in \mathcal{M}(S_{x_-},S_{x_+},\widetilde{G},j)$. Since the image of $u$ is contained in $\C\setminus\{0\}$, we write 
		\[
		u=(a,v):=(r^2(u),\arg u):\R\times S^1\to (0,\infty)\times \R/2\pi\Z.
		\]
		We recast the Floer equation into 
		\begin{equation}\label{eq:floer_eq_in_C}
		\begin{cases}
		\displaystyle\partial_s a - \frac{1}{2\pi g(r(u))}\partial_t v - \frac{k\,\partial_{r^2}h(r(u))}{g(r(u))}=0\\[2ex]
		\displaystyle \partial_s v + 2\pi g(r(u)) \partial_t a=0	
		\end{cases}
		\end{equation}
		where $g:(0,\infty)\to(0,\infty)$ is a smooth function given in \eqref{eq:j_fiber}. We set $\tilde{g}(R):=\int_0^Rg(\sqrt{\sigma})d\sigma$. Using the second equation in \eqref{eq:floer_eq_in_C}, we compute 
		\begin{equation}\label{eq:theta_const}
		\partial_s\int_0^1 v\, dt =-2\pi\int _0^1 g(r(u))\partial_ta\,dt=-2\pi\int_0^1 \partial_t \big(\tilde{g}(a)\big)dt =0,
		\end{equation}
		where we regard $v$ as a function to $\R$ satisfying $v(s,t+1)=v(s,t)+2\pi$. 
		Since $u$ converges asymptotically to $x_\pm$, which have the same winding number $m$, we have 
		\[
		\lim_{s\to\pm\infty} v(s,t)=2\pi m t+\tau_\pm  
		\]
		for some $\tau_\pm \in \R$. From \eqref{eq:theta_const}, we conclude $\tau_+=\tau_-$\ and  $\theta(\ev_-(u)(0))=\theta(\ev_+(u)(0))$.

		Let $\mathbf{u}\in \mathcal{M}(x_-,x_+,\widetilde{G},j)$. Then, it has one of the following forms.
		
		\medskip 
		
		\noindent Case 1: $\mathbf{u}=(c_0,u,c_1)$: Since $x_-$ is a maximum point and $x_+$ is a minimum point, $c_0$ and $c_1$ are constant flow lines, and $u\in \mathcal{M}(S_{x_-},S_{x_+},\widetilde{G},j)$ satisfies $\ev_\pm(u)=x_\pm$. By the above discussion, such $u$ does not exist for a generic choice of $f_{x_\pm}$.  
		
		\medskip 
		
		\noindent Case 2: $\mathbf{u}=(c_0,u_1,c_1,u_2,c_2)$, where $u_1\in \mathcal{M}(S_{x_-},S_{x'},{G}_{\mu_-},j)$ for some $1$-periodic orbit $x'$ of $X_{G_{\mu_-}}$ of type II with the same winding number  as $x_\pm$, and $u_2 \in \mathcal{M}(S_{x'},S_{x_+},\widetilde{G},j)$. Again, $c_0$ and $c_2$ are constant. By Lemmas \ref{lem:fiber_d} and \ref{lem:constconti} (or by the above discussion), we have 
		\[
		\theta(x_-)=\theta(\ev_-(u_1)(0))=\theta(\ev_+(u_1)(0)),\qquad  \theta(x_+) = \theta(\ev_+(u_2)(0)) =\theta(\ev_-(u_2)(0)).
		\] 
		Thus, $f_{x_\pm}$ can be chosen so that there is no gradient flow line $c_1$ of a perfect Morse function $f_{x'}$ on $S_{x'}$ connecting $\ev_+(u_1)(0)$ and $\ev_-(u_2)(0)$. 
		
		\medskip
		
		\noindent Case 3: $\mathbf{u}=(c_0,u_1,c_1,u_2,c_2)$, where $u_2\in \mathcal{M}(S_{x'},S_{x_+},{G}_{\mu_+},j) $ for some $1$-periodic orbit $x'$ of $X_{G_{\mu_+}}$ of type II with the same winding number  as $x_\pm$, and $u_1\in \mathcal{M}(S_{x_-},S_{x'},\widetilde{G},j)$. Arguing as in Case 2, we can choose $f_{x_\pm}$ so that this case does not occur.  
		
		\medskip
		
		In Cases 2 and 3, $x'$ has the same winding number as $x_\pm$ since the winding number is nondecreasing along $u_1$ and $u_2$.		There are no further cases with more Floer cylinders. Indeed, if $u$ is a Floer cylinder of $G_{\mu_-}$ or $G_{\mu_+}$ such that its asymptotic orbits have the same winding number and are both of type II or  III, then $u$ is necessarily trivial, i.e.~$\partial_su=0$.  This finishes the proof.
	\end{proof}
  	
  	\subsection{Neck-stretching}\label{sec:neckstr}

		To investigate Floer (continuation) cylinders, we perform neck-stretching along the hypersurface $\partial U_{\rho+2\epsilon}$. For this purpose, we have to turn off the perturbation added to $H_\mu$ (see below \eqref{eq:f_W}) and work with the original Hamiltonian $H_\mu$ given in \eqref{eq:Hamiltonian}, which is constant on $M\setminus U_{\rho+\epsilon}$. We elaborate on the  justification for this choice at the end of this section.

		 We modify a compatible almost complex structure $J$ on $M$ near $\partial U_{\rho+2\epsilon}$ to obtain a sequence $\{J_\eta\}_{\eta\in\R_+}$ of almost complex structures having the following property. We take 
		 \[
		 [x_\pm,\bar{x}_\pm]\in\widetilde{\mathcal P}_{\II}^\crit(H_\mu)\cup \widetilde{\mathcal P}_{\III}^\crit(H_\mu)\cup \widetilde{\mathcal P}_{\IV}^\crit(H_\mu),\qquad \mu([x_+,\bar{x}_+])-\mu([x_-,\bar{x}_-])=1.
		 \] 
		 By ${\mathcal{M}}_m([x_-,\bar{x}_-],[x_+,\bar{x}_+],H_\mu,J_\eta)$, we mean the moduli space of flow lines with Floer cascades whose asymptotic orbits are not type I orbits. Then, for every sufficiently large $\eta>0$, there is a bijective correspondence between ${\mathcal{M}}_m([x_-,\bar{x}_-],[x_+,\bar{x}_+],H_\mu,J_\eta)$ and the moduli space of pairs $(\mathbf{u}_1,\mathbf{u}_2)$, modulo reparametrizations, such that
		\begin{enumerate}[(i)]
			\item there is a finite index set $I=\coprod_{j=1}^m I_j$ and a tuple $\mathbf{u}_2=(c_0,{u}_1,\cdots,{u}_m,c_m)$ of curves as in  Definition \ref{def:cascade}, except that $u_j$ are punctured Floer cylinders on $(\widehat{E},J_{\widehat{E}})$,
			\begin{equation*}
				{u}_j:(\mathbb{R}\times S^1)\setminus\{z_i\}_{i\in I_j}\longrightarrow \widehat{E},
			\end{equation*}
			where $u_j$ converges to a periodic Reeb orbit $\gamma_{i}$ of $R_\alpha$ asymptotically at $z_i$ (we allow $I_j=\emptyset$ and in this case $\partial_su_j\neq 0$);
			\item $\mathbf{u}_1=\{v_i\}_{i\in I}$ consists of  ${J}_{\widehat{W}}$-holomorphic planes $v_i$ of $\widehat{W}$ asymptotic to  $\gamma_i$.
		\end{enumerate}

		Similar correspondence holds for ${\mathcal{M}}([x_-,\bar{x}_-],[x_+,\bar{x}_+],\widetilde{H},{J})$ for $[x_\pm,\bar{x}_\pm]\in\widetilde{\mathcal P}_{\II}^\crit(H_{\mu_\pm})\cup \widetilde{\mathcal P}_{\III}^\crit(H_{\mu_\pm})\cup \widetilde{\mathcal P}_{\IV}^\crit(H_{\mu_\pm})$ with $\mu([x_-,\bar{x}_-])=\mu([x_+,\bar{x}_+])$, where $\widetilde{H}$ is a monotone homotopy between $H_{\mu_-}$ and $H_{\mu_+}$. In this case, $\mathbf{u}_2$ consists of one punctured Floer continuation cylinder $u_{m_0}$ for $(\widetilde{H},\widetilde{J})$, punctured Floer cylinders $u_1,\dots,u_{m_0-1}$ for $(H_{\mu_-},J_-)$, and punctured Floer cylinders $u_{m_0+1},\dots,u_{m}$ for $(H_{\mu_+},J_+)$. No additional levels can arise in $(\mathbf{u}_1,\mathbf{u}_2)$ due to the monotonicity condition and index considerations.

		These correspondences follow from the SFT compactness and gluing theorems \cite{BEHWZ03,Par19}. 
		In the context of symplectic homology for $M\setminus D$, such a  correspondence was essential in defining the so-called split symplectic homology in \cite{DL19a,DL19b}; see  \cite[Proposition 5.1]{DL19b}. Adapting this correspondence, a Floer-Gysin exact sequence for $M\setminus D$ was constructed in \cite{BKK24} using Rabinowitz Floer homology. Furthermore, the possible configurations of $(\mathbf{u}_1,\mathbf{u}_2)$ are extremely constrained in the setting of Symplectic homology and Rabinowitz Floer homology for $M\setminus D$ as established in \cite{DL19a,DL19b,BKK24}. This rigidity remains valid in our setting as we shall see in the subsequent sections, and it is used crucially in the proof of our main results.

	By the following lemma from \cite[Lemma 2.6]{DL19b}, $J_{\widehat{W}}$-holomorphic planes in $\mathbf{u}_1$ naturally correspond to $J$-holomorphic spheres in $M$ with an incidence property.
	
	\begin{lem}\label{lem:remsing}
		A compatible almost complex structure $J_{\widehat{W}}$ on $\widehat{W}$ can be chosen so that there is a diffeomorphism $\chi:\mathrm{int}(W)\to M\setminus D$ satisfying $\chi_*J_{\widehat{W}}=J$. Moreover, under this diffeomorphism, we have a bijection between $J_{\widehat{W}}$-holomorphic planes in $\widehat{W}$ asymptotic to Reeb orbits $\gamma$ of multiplicity $m$ and $J$-holomorphic spheres in $M$ intersecting $D$ at a single point $\pi(\gamma)\in D$ with order of contact $m$.
	\end{lem}

	 Due to our choice of $J_{\widehat{E}}$, every punctured Floer (continuation) cylinder ${u}_j:(\mathbb{R}\times S^1)\setminus\{z_i\}_{i\in I_j}\to \widehat{E}$ in $\mathbf{u}_2$ projects to a $J_D$-holomorphic curve $\pi\circ u_j$ in $D$. Let the asymptotic orbits $x_{j,\pm}:=\ev_\pm(u_j)$ have winding numbers $m_{j,\pm}$. Let $\gamma_1,\dots,\gamma_l$ be asymptotic Reeb orbits at the punctures of $u_j$, with  winding numbers $m_1,\dots,m_l$ around the zero section of $\widehat E$.   Since the asymptotic orbits of $u_j$ project to points in $D$, by the removal of singularity theorem, we can regard $\pi\circ u_j$ as a $J_D$-holomorphic sphere. By gluing the fiber capping disks of the asymptotic orbits, denoted by $\bar{x}_{j,\pm}^\mathrm{fib}$ and $\bar\gamma_i^\mathrm{fib}$, we have $[\bar{x}_{j,-}^\mathrm{fib}\#u_j\# (-\bar{x}_{j,+}^\mathrm{fib})\#\bar{\gamma}_1^\mathrm{fib}\#\dots \# \bar{\gamma}_l^\mathrm{fib}]\in \pi_2(\widehat{E})$, which projects to $[\pi\circ u_j]\in \pi_2(D)$. Here $-\bar{x}_{j,+}^\mathrm{fib}$ denotes $\bar{x}_{j,+}^\mathrm{fib}$ with reversed orientation. Therefore, 
	\begin{equation}\label{eq:winding}
	\begin{split}
	k\omega_D([\pi\circ u_j])&= k\omega([\bar{x}_-^\mathrm{fib}\#u_j\# (-\bar{x}_+^\mathrm{fib})\#\bar{\gamma}_1^\mathrm{fib}\#\dots \# \bar{\gamma}_l^\mathrm{fib}])	\\
	&= [\bar{x}_-^\mathrm{fib}\#u_j\# (-\bar{x}_+^\mathrm{fib})\#\bar{\gamma}_1^\mathrm{fib}\#\dots \# \bar{\gamma}_l^\mathrm{fib}]\cdot D\\
	&= m_{j,-}-m_{j,+}+(m_1+\dots + m_l) + u_j\cdot D.
	\end{split}
	\end{equation}
	
	\begin{definition}\label{def:pearl}
		 For $N\in\N\cup\{0\}$ and $p_\pm\in\crit f_D$, let $\mathcal{N}_N^*(p_-,p_+;\mathbf{A};J_D)$ denote the moduli space of simple chains of pearls connecting $p_-$ and $p_+$, where $\mathbf{A}=(A_1,\dots,A_N)\in(\pi_2(D)\setminus\{0\})^N$ for $N\geq1$. Namely, it consists of a tuple $(z_0,v_1,\dots,v_N,z_N)$ such that 
	\begin{enumerate}[(i)]
			\item $v_i:\mathbb{CP}^1\to D$ are simple $J_D$-holomorphic spheres representing $[v_i]=A_i$, 
			\item the image of $v_i$ is not contained in the image of $v_j$ for any $i\neq j$,
			\item $z_i$ are gradient flow lines of $f_D$ such that
		\begin{itemize}
		\item 	for $1\leq i\leq N-1$,  $z_i(0)=v_{i}(\infty)$ and $z_i(t_i)=v_{i+1}(0)$ for some $t_i>0$; 
		\item $\displaystyle \lim_{s\to -\infty}z_0(s)=p_-$ and  $z_0(t_0)=v_1(0)$ for some $t_0\in\R$;
		\item $\displaystyle \lim_{s\to \infty}z_N(s)=p_+$ and  $z_N(t_N)=v_N(\infty)$ for some $t_N\in\R$,
		\item when $N=0$, meaning the tuple contains only  $z_0$, $\displaystyle\lim_{s\to\pm\infty} z_0(s)=p_\pm$.
		\end{itemize}
		\end{enumerate}
	\end{definition}
	
	 For $N\geq 1$, this moduli space is cut out transversely for generic $J_D$  and has dimension 
	 \[
	 \dim \mathcal{N}_N^*(p_-,p_+;\mathbf{A};J_D)=\ind_{f_D}(p_+)-\ind_{f_D}(p_-)+2c_1^{TD}(\mathbf{A})+N-1,
	 \]
	 where $\displaystyle c_1^{TD}(\mathbf{A}):=\sum_{i=1}^N c_1^{TD}(A_i)$. 
	Since $\C^*$ acts freely on each $v_i$, we have 
	\begin{equation}\label{eq:2N}
	 \dim \mathcal{N}_N^*(p_-,p_+;\mathbf{A};J_D)\geq 2N.	
	\end{equation}
	For $N=0$, $\mathcal{N}_{N=0}^*(p_-,p_+;\mathbf{A}=\emptyset;J_D)=\mathcal{N}(p_-,p_+)$ consists of gradient flow lines of $f_D$ connecting $p_-$ and $p_+$. We retain $J_D$ in the notation for consistency although it is  irrelevant in this case.

	Let $(\mathbf{u}_1,\mathbf{u}_2)$ be the limit arising from either an element of ${\mathcal{M}}_m([x_-,\bar{x}_-],[x_+,\bar{x}_+],H_\mu,J_\eta)$ with $\mu([x_+,\bar{x}_+])-\mu([x_-,\bar{x}_-])=1$ or an element of 
	${\mathcal{M}}_m([x_-,\bar{x}_-],[x_+,\bar{x}_+],\widetilde{H},J_\eta)$ for a monotone homotopy $\widetilde{H}$ with $\mu([x_+,\bar{x}_+])=\mu([x_-,\bar{x}_-])$.
	 We write $p_\pm:=\pi(x_\pm(0))\in \crit f_D$. 
	Let $N\in\N\cup\{0\}$ be the number of punctured Floer (continuation) cylinders in $\mathbf{u}_2$ that project to nonconstant  $J_D$-holomorphic spheres. Equivalently, $m-N$ punctured Floer (continuation) cylinders in $\mathbf{u}_2$ are contained in fibers of $\widehat{E}$. The projection of $\mathbf{u}_2$ to $D$ induces an element in $\mathcal{N}_{N}^*(p_-,p_+;\mathbf{A};J_D)$, where $\mathbf{A}=(A_1,\dots,A_N)$ is a tuple of nonzero homotopy classes $[\pi\circ u_j]$. The index condition and the monotonicity condition imply that the projection indeed satisfies the properties (i) and (ii) in Definition \ref{def:pearl}, see \cite[Proposition 6.2]{DL19a} and \cite[Proposition 4.5]{BKK24}. Furthermore, $J_{\widehat{W}}$-holomorphic planes in $\mathbf{u}_1$ or the corresponding $J$-holomorphic spheres according to Lemma \ref{lem:remsing} are simple.

	\begin{remark}\label{rem:index_0_proj}
  		From \eqref{eq:2N}, we observe that if $\mathcal{N}_N^*(p_-,p_+;\mathbf{A};J_D)$ with $N\geq 1$ is nonempty, then 
  	\begin{equation*}\label{eq:index_base}
  	\ind_{f_D}(p_+)-\ind_{f_D}(p_-)+2c_1^{TD}(\mathbf{A})\geq N+1.  		
  	\end{equation*}  	
  Suppose that $\mathbf{u}_2$ projects to an element in $\mathcal{N}_N^*(p_-,p_+;\mathbf{A};J_D)$ satisfying $\ind_{f_D}(p_+)-\ind_{f_D}(p_-)+2c_1^{TD}(\mathbf{A})=0$. Then, $N=0$, $p_+=p_-$, and the projection is the constant trajectory at $p_+$. Thus, all punctured Floer cylinders in $\mathbf{u}_2$ are contained in the fiber $\widehat{E}_{p_+}$. We observe that if $u_j$ in $\mathbf{u}_2$ neither has punctures nor intersects $D$, then $m_{j,-}=m_{j,+}$ by \eqref{eq:winding}. 
	\end{remark}

  			 Below \eqref{eq:f_W}, we added a $C^2$-small perturbation $\epsilon f_W$ to $H_\mu$ to achieve the Morse property for type I orbits. As described in this section, when analyzing Floer (continuation) cylinders connecting orbits of types II, III, and IV, we may turn off this perturbation and assume $H_\mu$ remains constant on $M\setminus U_{\rho+\epsilon}$. This is possible because there is a sign-preserving bijection between $\overline{\mathcal{M}}([x_-,\bar{x}_-],[x_+,\bar{x}_+],H_\mu,J)$ and the moduli defined with a Hamiltonian sufficiently $C^2$-close to $H_\mu$ due to the openness of the surjectivity condition for the associated Fredholm operators. 
		We incorporate this with Floer continuation cylinders as follows. We fix acceleration data:~$(\{H_i\},\{J_i\})$ and homotopies $\widetilde{H}_i$ between $H_i$ and $H_{i+1}$. Then, we can take a sequence $\epsilon_i \searrow 0$ such that the following holds:
		\[
		\#\overline{\mathcal{M}}([x_-,\bar{x}_-],[x_+,\bar{x}_+],H_i,J_i)=\#\overline{\mathcal{M}}([x_-,\bar{x}_-],[x_+,\bar{x}_+],H_i+\epsilon_if_W,J_i)
		\]
		for all $[x_\pm,\bar{x}_\pm]$ of type II, III, or IV with $\mu([x_+,\bar{x}_+])-\mu([x_-,\bar{x}_-])=1$, and 
		\[
		\#\overline{\mathcal{M}}([x_-,\bar{x}_-],[x_+,\bar{x}_+],\widetilde{H}_i,{J}_i)=\#\overline{\mathcal{M}}([x_-,\bar{x}_-],[x_+,\bar{x}_+],\widetilde{H} _i+(\rho \epsilon_i+(1-\rho)\epsilon_{i+1}) f_W,{J}_i)
		\]
		for all $[x_\pm,\bar{x}_\pm]$ of type II, III, or IV with $\mu([x_+,\bar{x}_+])=\mu([x_-,\bar{x}_-])$, where $\rho:\R\to[0,1]$ is a smooth function such that $\rho'\leq 0$, $\rho(-1)=1$, and $\rho(1)=0$. Here, for distinction, we temporarily write $H_i$ for the unperturbed Hamiltonians and $H_i+\epsilon_i f_W$ for the perturbed ones.
		As defined at the beginning of this section, the moduli space associated with $H_i$ consists of flow lines with Floer cascades whose asymptotic limits do not include type I orbits. Likewise, the moduli spaces defined with $H_i+\epsilon_if_W$ cannot involve any type I orbits as these orbits are strictly of Morse type. The same observation applies to $\widetilde{H}_i$ and its perturbation. 
		
		 If $\{J_i\}$ is a sequence of sufficiently stretched almost complex structures, we also have 
		\[
		\#\overline{\mathcal{M}}([x_-,\bar{x}_-],[x_+,\bar{x}_+],{H}_{i+1},{J}_i)=\#\overline{\mathcal{M}}([x_-,\bar{x}_-],[x_+,\bar{x}_+],{H}_{i+1},{J}_{i+1})
		\]
		by the gluing theorem. Both counts equal the count of the limit configurations $(\mathbf{u}_1,\mathbf{u}_2)$ illustrated above. The specific one-ray we employ is given in \eqref{eq:one-ray}.
		
		In what follows, we again simply denote the perturbed Hamiltonians by $H_i$ (or $H_\mu$), with the convention that the perturbation term is turned off whenever neck-stretching is performed.

	\subsection{Restrictions on Floer continuation trajectories}
	
	\begin{prop}\label{prop:empty}
		Let $\widetilde{H}$ be a monotone homotopy between $H_{\mu_-}$ and $H_{\mu_+}$ as above, where $\mu_\pm$ satisfy 
		\begin{equation*}
			0\ll \mu_- < \mu_+ \le \mu_-+C
		\end{equation*}
		for some constant $C>1$. For $[x_\pm,\bar{x}_\pm]\in\widetilde{\mathcal{P}}^\mathrm{crit}(H_{\mu_\pm})$ with the same index, the moduli space ${\mathcal{M}}([x_-,\bar{x}_-],[x_+,\bar{x}_+],\widetilde{H},J)$ is empty in the following cases.		
		\begin{enumerate}
			\item[(a)] $[x_-,\bar{x}_-] \in \widetilde{\mathcal{P}}_{\I}(H_{\mu_-})$, $[x_+,\bar{x}_+]\in\widetilde{\mathcal{P}}_{\III}^\crit(H_{\mu_+})$;
			\item[(b)] $[x_-,\bar{x}_-] \in \widetilde{\mathcal{P}}_{\I}(H_{\mu_-})$, $[x_+,\bar{x}_+]\in \widetilde{\mathcal{P}}_{\IV}^\crit(H_{\mu_+})$.
		\end{enumerate}
		The same statement holds for the following cases when $J=J_\eta$ is sufficiently stretched.
		\begin{enumerate}
			\item[(c)] $[x_-,\bar{x}_-] \in \widetilde{\mathcal{P}}_{\III}^\crit(H_{\mu_-})$, $[x_+,\bar{x}_+]\in \widetilde{\mathcal{P}}_{\II}^\crit(H_{\mu_+})$;
			\item[(d)] $[x_-,\bar{x}_-] \in \widetilde{\mathcal{P}}_{\II}^\crit(H_{\mu_-})$, $[x_+,\bar{x}_+]\in \widetilde{\mathcal{P}}_{\IV}^\crit(H_{\mu_+})$;
			\item[(e)] $[x_-,\bar{x}_-] \in \widetilde{\mathcal{P}}_{\III}^\crit(H_{\mu_-})$, $[x_+,\bar{x}_+]\in \widetilde{\mathcal{P}}_{\IV}^\crit(H_{\mu_+})$;
			\item[(f)] $[x_-,\bar{x}_-] \in \widetilde{\mathcal{P}}_{\IV}^\crit(H_{\mu_-})$, $[x_+,\bar{x}_+]\in \widetilde{\mathcal{P}}_{\IV}^\crit(H_{\mu_+})$ provided $\mu_+\geq \mu_-+1$.	
		\end{enumerate}
	\end{prop}
	\begin{proof}
		Let $\mathbf{u}=(c_0,u_1,c_1, u_2, \cdots, u_m, c_m)\in{\mathcal{M}}([x_-,\bar{x}_-],[x_+,\bar{x}_+],\widetilde{H},{J})$.
		By Remark \ref{rmk:zerocap}, there is no loss of generality in assuming the capping of $x_-$ is a fiber disk. 
		\medskip

			\noindent Proof of (a). As before, we write $[x_-,\bar{x}_-]=[\mu_-,\I,q_-,0]$ and $[x_+,\bar{x}_+]=[\mu_+,\III,q_+,m,A]$. By \eqref{eq:action_index}, we have 
			\[
				\mathcal{A}_{H_{\mu_-}}([x_-,\bar{x}_-])\approx\frac{\mu_-}{k}\rho,\qquad 
				\mathcal{A}_{H_{\mu_+}}([x_+,\bar{x}_+])\approx \frac{m}{k}\rho+\omega(A).
			\]
			and
			\[
				\mu([x_-,\bar{x}_-])\in[-n,n],\qquad 
				\mu([x_+,\bar{x}_+])\in[-n+1,n]+2m+2c_1^{TM}(A)
			\]
			From $\mu([x_-,\bar{x}_-])=\mu([x_+,\bar{x}_+])$, we deduce 
			\begin{equation}\label{eq:m-n}
			c_1^{TM}(A)	= \tau \omega(A) \leq n-m.
			\end{equation}
			Using this and $\mathcal{A}_{H_{\mu_-}}([x_-,\bar{x}_-])<\mathcal{A}_{H_{\mu_+}}([x_+,\bar{x}_+])$, we estimate 
			\begin{equation}\label{eq:mnrho}
			 \frac{\mu_-}{k}\rho < \frac{m}{k}\rho + \omega(A) \leq  \frac{m}{k}\rho+\frac{n-m}{\tau}.
			\end{equation}
			This also implies
			\begin{equation}\label{eq:mrho}
			(\mu_+-C)\rho \leq \mu_-\rho\leq m\rho+\rho_\crit(n-m).
			\end{equation}
			Case 1) Suppose $\rho\geq \rho_\crit =\frac{k}{\tau}$. In this case, using $m<\mu_+$ (see \eqref{eq:abbrev}) and \eqref{eq:mrho},
			\[
			\mu_+\rho_\crit =\mu_+\rho- \mu_+(\rho-\rho_\crit )\leq \mu_+\rho- m(\rho-\rho_\crit )\leq C\rho +n\rho_\crit 
			\]
			This yields a contradiction for all sufficiently large $\mu_+$. 
			
			\medskip
			
			\noindent Case 2) Suppose $\rho<\rho_\crit$. By the asymptotic condition of $\mathbf{u}$, we have $A=[\mathbf{u}\#(-\bar{x}_+^\mathrm{fib})]$, where $-\bar{x}_+^\mathrm{fib}$ denotes the fiber capping disk $\bar{x}_+^\mathrm{fib}$ of ${x}_+$ with reversed orientation. Since $\bar{x}_+^\mathrm{fib}\cdot D=m$ and $\mathbf{u}\cdot D\geq 0$ by Lemma \ref{lem:intersect}.(a), we have
			\[
			k\omega(A)=k\omega([\mathbf{u}\#(-\bar{x}_+^\mathrm{fib})])=D\cdot (\mathbf{u}\#(-\bar{x}_+^\mathrm{fib}))\geq -m.
			\] 
			Combining this and \eqref{eq:m-n}, we obtain  $\tau\omega(A)\leq n-m \leq n+k\omega(A)$ and consequently 
			\begin{equation}\label{eq:contra1}
			\omega(A) \leq \frac{n}{\tau-k}.	
			\end{equation}
			On the other hand, by \eqref{eq:mnrho} and \eqref{eq:m-n},
			\[
			\mu_-\rho < m\rho+k\omega(A) \leq (n-\tau\omega(A))\rho +k\omega(A).
			\]
			This yields 
			\[
			(\mu_--n)\rho\leq k\omega(A)\Big(1-\frac{\rho}{\rho_\crit}\Big)			
			\]
			which contradicts \eqref{eq:contra1} for sufficiently large $\mu_->0$.
			
			\medskip
			
			\noindent Proof of (b). 
			Let $[x_-,\bar{x}_-]=[\mu_-,\I,q_-,0]$ and $[x_+,\bar{x}_+]=[\mu_+,\IV,q_+,A]$. By \eqref{eq:action_index}, we have 
			\[
				\mathcal{A}_{H_{\mu_-}}([x_-,\bar{x}_-])\approx\frac{\mu_-}{k}\rho,\qquad 
				\mathcal{A}_{H_{\mu_+}}([x_+,\bar{x}_+])\approx \frac{\mu_+}{k}\rho+\omega(A).
			\]
			and
			\[
				\mu([x_-,\bar{x}_-])\in[-n,n],\qquad 
				\mu([x_+,\bar{x}_+])\in [-n+2,n]+2\lfloor\mu_+\rfloor+2c_1^{TM}(A)
			\]
			From $\mu([x_-,\bar{x}_-])=\mu([x_+,\bar{x}_+])$, we deduce 
			\begin{equation}\label{eq:tauomegaA}
			\tau\omega(A)=c_1^{TM}(A)\leq -\lfloor\mu_+\rfloor+n-1.
			\end{equation}
			Also from $\mathcal{A}_{H_{\mu_-}}([x_-,\bar{x}_-])< \mathcal{A}_{H_{\mu_+}}([x_+,\bar{x}_+])$, we obtain
			\[
			-\frac{C}{k}\rho \leq \frac{\mu_--\mu_+}{k}\rho <\omega(A).
			\]
			This contradicts \eqref{eq:tauomegaA} for sufficiently large $\mu_+>0$.
			
			\medskip

			\noindent Proof of (c). For $[x_-,\bar{x}_-]=[\mu_-,\III,q_-,m_-,0]$ and $[x_+,\bar{x}_+]=[\mu_+,\II,q_+,m_+,A]$, we have
			\begin{equation}\label{eq:iii_index}
			\begin{split}
			\mu([x_-,\bar{x}_-])&=2m_-+\ind_{f_{\partial E}}(q_-)-(n-1)-1\\[.5ex]
				&=2m_-+\ind_{f_D}(\pi(q_-))-(n-1)+e_-\\[.5ex]
				\mu([x_+,\bar{x}_+])&=2m_++2c_1^{TM}(A)+\ind_{f_{\partial E}}(q_+)-(n-1)\\
				&=2m_++2c_1^{TM}(A)+\ind_{f_{D}}(\pi(q_+))-(n-1)+e_+	
			\end{split}
			\end{equation}
			where $e_-\in\{0,-1\}$ and $e_+\in\{0,1\}$. Since $A=[\bar{x}_-^\mathrm{fib} \# \mathbf{u}\# (-\bar{x}_+^\mathrm{fib})]$, Lemma \ref{lem:intersect}.(a) yields
			\begin{equation}\label{eq:int}
				k\omega(A)=A\cdot[D]\ge m_--m_+.
			\end{equation}
			Now, we perform neck-stretching as described in Section \ref{sec:neckstr}, so that $\mathbf{u}$ breaks into a pair of curves $\mathbf{u}_1$ in $\widehat{W}$ and $\mathbf{u}_2$ in $\widehat{E}$, with the possibility  that $\mathbf{u}_1=\emptyset$. Let $\bar{\mathbf{u}}_1$ be the set of ${J}$-holomorphic spheres in $M$ corresponding to $\mathbf{u}_1$ as explained in Lemma \ref{lem:remsing}. As described below Definition \ref{def:pearl}, the projection of $\mathbf{u}_2$ to $D$ yields an element in  $\mathcal{N}_N^*(p_-,p_+;[\pi\circ\mathbf{u}_2];J_D)$, where $p_\pm:=\pi(q_\pm)$ and $[\pi\circ\mathbf{u}_2]\in(\pi_2(D)\setminus\{0\})^N$ consists of the homotopy classes represented by $J_D$-holomorphic spheres $\pi\circ u_j$. Recall from Remark \ref{rem:index_0_proj} that 			
			\begin{equation*}
				\ind_{f_D}(p_+)-\ind_{f_D}(p_-)+2c_1^{TD}([\pi\circ\mathbf{u}_2])\geq0.
			\end{equation*}
			Moreover from $A=[\bar{x}_-^\mathrm{fib}\#\mathbf{u}_1\#\mathbf{u}_2\#(-\bar{x}_+^\mathrm{fib})]$, we have 
			\begin{equation}\label{eq:class_A}
				\begin{split}
				\omega(A)
				&=\omega(\bar{x}_-^\mathrm{fib}\#\mathbf{u}_1\#\mathbf{u}_2\#(-\bar{x}_+^\mathrm{fib}))\\
				&=\omega(\bar{\mathbf{u}}_1)+\omega(\bar{x}_-^\mathrm{fib}\#\mathbf{u}_2\#(-\bar{x}_+^\mathrm{fib})\#\bar{\gamma}_1^\mathrm{fib}\#\dots \# \bar{\gamma}_l^\mathrm{fib})\\
				&=\omega(\bar{\mathbf{u}}_1)+\omega_D([\pi\circ\mathbf{u}_2]),
			\end{split}
			\end{equation}	
			where $\gamma_1,\dots,\gamma_l$ are the common asymptotic Reeb orbits of $\mathbf{u}_1$ and $\mathbf{u}_2$. Thus,
			\[
			\begin{split}
				c_1^{TM}(A)&=k\omega(A)+(\tau-k)\omega(A)\\
				&=k\omega(A) +(\tau-k)\omega(\bar{\mathbf{u}}_1)+ (\tau-k)\omega_D([\pi\circ\mathbf{u}_2]).
			\end{split}
			\]			
			Combining these with \eqref{eq:iii_index} and \eqref{eq:int}, we estimate 
			\begin{align*}
				0&=\mu([x_+,\bar{x}_+])-\mu([x_-,\bar{x}_-])\\
				&=\ind_{f_{D}}(p_+)-\ind_{f_D}(p_-)	+2c_1^{TM}(A) 
				+2(m_+-m_-) +(e_+-e_-)\\
				&=\underbrace{(\ind_{f_{D}}(p_+)-\ind_{f_D}(p_-)	+2c_1^{TD}([\pi\circ\mathbf{u}_2]) }_{\geq 0}
				+\underbrace{2(m_+-m_-+k\omega(A))}_{\geq 0} \\
				&\quad +\underbrace{2(\tau-k)\omega(\bar{\mathbf{u}}_1)}_{\geq0} +\underbrace{(e_+-e_-)}_{\geq 0}.
			\end{align*}
			Therefore, all four terms in the last line vanish. Since $\mathbf{u}_1=\emptyset$, all Floer (continuation) cylinders in $\mathbf{u}_2$ have no punctures. By Remark \ref{rem:index_0_proj}, $\mathbf{u}_2$ is entirely contained in the fiber $\widehat{E}_{p_+}$ over $p_+=p_-$.
			Moreover,  $A=[\bar{x}_-^\mathrm{fib}\#\mathbf{u}_2\#(-\bar{x}_+^\mathrm{fib})]=0$ as $\bar{x}_-^\mathrm{fib}\#\mathbf{u}_2\#(-\bar{x}_+^\mathrm{fib})\subset \widehat{E}_{p_+}$, and thus $m_+=m_-$. However, such $\mathbf{u}_2$ does not exist due to Lemma \ref{lem:addR2}.

			\medskip
			
			\noindent Proof of (d). Let $[x_-,\bar{x}_-]=[\mu_-,\II,q_-,m_-,0]$ and $[x_+,\bar{x}_+]=[\mu_+,\IV,q_+,A]$. Then, we have
			\begin{equation*}
				\begin{split}
					\mu([x_-,\bar{x}_-])&=2m_-+\ind_{f_D}(\pi(q_-))-(n-1)+e\\[.5ex]
					\mu([x_+,\bar{x}_+])&=2\lfloor\mu_+\rfloor+1+\ind_{f_D}(q_+)-(n-1) +2c_1^{TM}(A)
				\end{split}
			\end{equation*}
			where $e\in\{0,1\}$. Since $A=[\bar{x}_-^\mathrm{fib} \# \mathbf{u}]$, Lemma \ref{lem:intersect}.(b) yields
			\begin{equation*}
				k\omega(A)=A\cdot[D]\ge m_--\lfloor\mu_+\rfloor.
			\end{equation*}
			We perform neck-stretching and obtain a pair of curves $(\mathbf{u}_1,\mathbf{u}_2)$ corresponding to $\mathbf{u}$. We argue as in the proof of (c) and adopt the same notation. Using the above inequalities, we estimate
			\begin{align*}
				0&=\mu([x_+,\bar{x}_+])-\mu([x_-,\bar{x}_-])\\
				&=\ind_{f_{D}}(q_+)-\ind_{f_D}(\pi(q_-))	+2c_1^{TM}(A) 
				+2(\lfloor\mu_+\rfloor -m_-) +1-e\\
				&=\underbrace{(\ind_{f_{D}}(q_+)-\ind_{f_D}(\pi(q_-))	+2c_1^{TD}([\pi\circ\mathbf{u}_2]) }_{\geq 0}
				+\underbrace{2(\lfloor\mu_+\rfloor-m_-+k\omega(A))}_{\geq 0} \\
				&\quad +\underbrace{2(\tau-k)\omega(\bar{\mathbf{u}}_1)}_{\geq0} +\underbrace{1-e}_{\geq 0}.
			\end{align*}
			Therefore, all four terms in the last line vanish. In particular, $\mathbf{u}_1=\emptyset$, and  $\pi\circ \mathbf{u}_2$ is constant at $q_+=\pi(q_-)$. Thus, $\mathbf{u}_2$ is contained in the fiber $\widehat{E}_{q_+}$. By Lemmas \ref{lem:mp}, \ref{lem:ab}, and \ref{lem:conti_empty}, there is no Floer (continuation) cylinder in $\widehat{E}_{q_+}$ asymptotic to a type II orbit at the negative end and a type III or IV orbit at the positive end. Thus, such a configuration $\mathbf{u}_2$ does not exist.
			
			\medskip
			
			\noindent Proof of (e). 
			 Let $[x_-,\bar{x}_-]=[\mu_-,\III,q_-,m_-,0]$ and $[x_+,\bar{x}_+]=[\mu_+,\IV,q_+,A]$. Then, we have
			 \begin{equation*}
				\begin{split}
					\mu([x_-,\bar{x}_-])&=2m_-+\ind_{f_D}(\pi(q_-))-(n-1)+e\\[.5ex]
					\mu([x_+,\bar{x}_+])&=2\lfloor\mu_+\rfloor+1+\ind_{f_D}(q_+)-(n-1) +2c_1^{TM}(A)
				\end{split}
			\end{equation*}
			where $e\in\{0,-1\}$. Since $A=[\bar{x}_-^\mathrm{fib}\#\mathbf{u}]$, Lemma \ref{lem:intersect}.(b) yields
			\begin{equation*}
					k\omega(A)=A\cdot[D]\ge m_--\lfloor\mu_+\rfloor.
			\end{equation*}
			As above, we perform neck-stretching and obtain $(\mathbf{u}_1,\mathbf{u}_2)$ corresponding to $\mathbf{u}$. 
			The following estimate yields a contradiction.
			\begin{align*}
				0&=\mu([x_+,\bar{x}_+])-\mu([x_-,\bar{x}_-])\\
				&=\ind_{f_{D}}(q_+)-\ind_{f_D}(\pi(q_-))	+2c_1^{TM}(A) 
				+2(\lfloor\mu_+\rfloor -m_-) +1-e\\
				&=\underbrace{(\ind_{f_{D}}(q_+)-\ind_{f_D}(\pi(q_-))	+2c_1^{TD}([\pi\circ\mathbf{u}_2]) }_{\geq 0}
				+\underbrace{2(\lfloor\mu_+\rfloor-m_-+k\omega(A))}_{\geq 0} \\
				&\quad +\underbrace{2(\tau-k)\omega(\bar{\mathbf{u}}_1)}_{\geq0} +\underbrace{1-e}_{\geq 1}.
			\end{align*}
		 			
			\medskip
			
			\noindent Proof of (f).
			Let $[x_-,\bar{x}_-]=[\mu_-,\IV,q_-,0]$ and $[x_+,\bar{x}_+]=[\mu_+,\IV,q_+,A]$. Then, we have
			\begin{equation*}
				\begin{split}
					\mu([x_-,\bar{x}_-])&=2\lfloor\mu_-\rfloor+1+\ind_{f_D}(q_-)-(n-1),\\[.5ex]
					\mu([x_+,\bar{x}_+])&=2\lfloor\mu_+\rfloor+1+\ind_{f_D}(q_+)-(n-1) +2c_1^{TM}(A).
				\end{split}
			\end{equation*}
			Also, Lemma \ref{lem:intersect}.(c) yields
			\begin{equation*}
				k\omega(A)=A\cdot[D]=[\mathbf{u}]\cdot[D]\ge \lfloor\mu_-\rfloor+1-\lfloor\mu_+\rfloor.
			\end{equation*}
		As before, let $(\mathbf{u}_1,\mathbf{u}_2)$ denote the limit of $\mathbf{u}$ obtained via neck-stretching. Then, the following estimate yields a contradiction. 
			\begin{align*}
				0&=\mu([x_+,\bar{x}_+])-\mu([x_-,\bar{x}_-])\\
				&=\ind_{f_{D}}(q_+)-\ind_{f_D}(q_-)	+2c_1^{TM}(A) 
				+2(\lfloor\mu_+\rfloor - \lfloor\mu_-\rfloor) \\
				&=\underbrace{(\ind_{f_{D}}(q_+)-\ind_{f_D}(q_-)	+2c_1^{TD}([\pi\circ\mathbf{u}_2]) }_{\geq 0}
				+\underbrace{2(\lfloor\mu_+\rfloor-\lfloor\mu_-\rfloor+k\omega(A))}_{\geq 2} \\
				&\quad +\underbrace{2(\tau-k)\omega(\bar{\mathbf{u}}_1)}_{\geq0}.\qedhere
			\end{align*}
	\end{proof}

	\begin{prop}\label{prop:res}
		There exist constants $C_1,C_2$ such that the following hold. 
		Let $\widetilde{H}$ be a monotone homotopy between $H_{\mu_-}$ and $H_{\mu_+}$ with $\mu_-<\mu_+$. Assume that there exists $\mathbf{u}\in{\mathcal{M}}([x_-,\bar{x}_-],[x_+,\bar{x}_+],\widetilde{H},J )$ for $[x_\pm,\bar{x}_\pm]\in\widetilde{\mathcal P}^\crit(H_{\mu_\pm})$ with the same index.
		\begin{enumerate}[(a)]
			\item For every $[x_-,\bar{x}_-]=[\mu_-,\I,q_-,A_-]$ and $[x_+,\bar{x}_+]=[\mu_+,\II,q_+,m_+,A_+]$, $m_+\ge C_1$ holds;
			\item Assume $J=J_\eta$ is sufficiently stretched. For every $[x_-,\bar{x}_-]=[\mu_-,\II,q_-,m_-,A_-]$ and $[x_+,\bar{x}_+]=[\mu_+,\III,q_+,m_+,A_+]$, $m_-\le -\frac{\rho}{\epsilon}\mu_-+C_2$ holds.
		\end{enumerate}
	\end{prop}
	
	\begin{proof}
		By Remark \ref{rmk:zerocap}, we may assume  $A_-=0$. We write $A_+=A$.
		
		\medskip
		
		\noindent Proof of (a). From
		\begin{align*}
			\mu([x_-,\bar{x}_-])&=\ind_{f_W}(q_-)-n\\[.5ex]
			\mu([x_+,\bar{x}_+])&=2m_++2c_1^{TM}(A)+\ind_{f_{\partial E}}(q_+)-(n-1),
		\end{align*}
		we obtain
		\begin{equation*}
			2\tau\omega(A)=2c_1^{TM}(A)=\ind_{f_W}(q_-)-\ind_{f_{\partial E}}(q_+)-1-2m_+\le 2n-2m_+.
		\end{equation*}
		By $A=[\mathbf{u}\#(-\bar{x}_+^{\text{fib}})]$ and Lemma \ref{lem:intersect}.(a), we have $k\omega(A)=A\cdot [D]\ge -m_+$.
		Putting these two inequalities together, we conclude
		$m_+\ge -\frac{k}{\tau-k}n$.

		\medskip
		
		\noindent Proof of (b). We have
		\begin{align*}
			\mu([x_-,\bar{x}_-])&=2m_-+\ind _{f_D}(\pi(q_-))-(n-1)+e_-\\[.5ex]
			\mu([x_+,\bar{x}_+])&=2m_++2c_1^{TM}(A)+\ind _{f_D}(\pi(q_+))-(n-1)+e_+
		\end{align*}
		where $e_-\in\{0,1\}$ and $e_+\in\{0,-1\}$. By $A=[\bar{x}_-^\text{fib}\# \mathbf{u}\# (-\bar{x}_+^{\text{fib}})]$ and Lemma \ref{lem:intersect}.(a),
		\begin{equation*}
			k\omega(A)=A\cdot[D]\ge m_--m_+.
		\end{equation*}
		We perform neck-stretching as described in Section \ref{sec:neckstr}. Let $(\mathbf{u}_1,\mathbf{u}_2)$ denote the  limit configuration  corresponding to  $\mathbf{u}$. Arguing as before, we have
		\begin{align*}
			0&=\mu([x_+,\bar{x}_+])-\mu([x_-,\bar{x}_-])\\
			&=\ind_{f_{D}}(\pi(q_+))-\ind_{f_D}(\pi(q_-))	+2c_1^{TM}(A_+) 
			+2(m_+ -m_-) +e_+-e_-\\
			&=\underbrace{(\ind_{f_{D}}(\pi(q_+))-\ind_{f_D}(\pi(q_-))	+2c_1^{TD}([\pi\circ\mathbf{u}_2]) }_{:=\mathbf{(a)}\geq 0}
			+\underbrace{2(m_+-m_-+k\omega(A))}_{:=\mathbf{(b)}\geq 0} \\
			&\quad +\underbrace{2(\tau-k)\omega(\bar{\mathbf{u}}_1)}_{:=\mathbf{(c)}\geq0} +\underbrace{e_+-e_-}_{\geq -2}.
		\end{align*}
		Suppose $\mathbf{(b)}=2$. Then, $\mathbf{u}_1=\emptyset$ and  $\mathbf{u}_2$ is contained in the fiber $\widehat{E}_p$ over $p=\pi(q_+)=\pi(q_-)$, see Remark \ref{rem:index_0_proj}. By Lemmas \ref{lem:mp}, \ref{lem:ab}, and \ref{lem:conti_empty}, such a configuration $\mathbf{u}_2$ does not exist.

		Therefore, we have $\mathbf{(b)}=0$. By $\mathbf{(a)} \leq 2$ and $\mathbf{(c)} \leq 2$, we have 
		\begin{equation*}
			\omega(\bar{\mathbf{u}}_1)\le \frac{1}{\tau-k},\qquad 
			\omega_D([\pi\circ \mathbf{u}_2])\le \frac{n}{\tau-k}
		\end{equation*}
		Since the action increases along $\mathbf{u}$, 
		\begin{equation*}
			\frac{\rho+\epsilon}{k}m_- + \frac{\rho}{k}\mu_- \approx  \mathcal{A}_{H_{\mu_-}}([x_-,\bar{x}_-])< \mathcal{A}_{H_{\mu_+}}([x_+,\bar{x}_+]) \approx\frac{\rho}{k}m_++\omega(A).
		\end{equation*}		
		Putting these inequalities together, we obtain
		\begin{align*}
			\frac{\epsilon}{k}\Big(\frac{\rho}{\epsilon}\mu_-+m_-\Big)&\le \frac{\rho}{k}(m_+-m_-)+\omega(A)\\
			&=(1-\rho)\omega(A)\\
			&=(1-\rho)(\omega_D([\pi\circ \mathbf{u}_2])+\omega(\bar{\mathbf{u}}_1))\\
			&\le(1-\rho)\frac{n+1}{\tau-k},		
		\end{align*}
		where the third line follows from \eqref{eq:class_A}. 
		This proves $m_- \leq -\frac{\rho}{\epsilon}\mu_-+ \frac{k(1-\rho)(n+1)}{\epsilon(\tau-k)}$ as claimed.
	\end{proof}
	
	\begin{remark}
		The statement of Proposition \ref{prop:res} is not vacuous. For example, when $(M,D)=(\mathbb{CP}^1,\{\infty\})$, one can see that on $\widehat{M\setminus D}=\C$, there is a Floer continuation cylinder connecting a type II orbit of $H_{\mu_-}$ to a type III orbit of $H_{\mu_+}$ if and only if $m_-= \lceil -\frac{\rho}{\epsilon}\mu_-\rceil$. 
	\end{remark}
	
	\begin{prop}\label{prop:dec}
		Let $\widetilde{H}$ be a monotone homotopy between $H_{\mu_-}$ and $H_{\mu_+}$ with $\mu_-<\mu_+$.  For $[x_\pm,\bar{x}_\pm]\in\widetilde{\mathcal P}^\crit(H_{\mu_\pm})$, suppose that there exists $\mathbf{u}\in{\mathcal{M}}([x_-,\bar{x}_-],[x_+,\bar{x}_+],\widetilde{H},J)$. Then, we have
		\begin{equation*}
			\mathcal{A}_{H_{\mu_+}}([x_+,\bar{x}_+])-\mathcal{A}_{H_{\mu_-}}([x_-,\bar{x}_-])\geq\frac{\rho}{k}(\mu_+-\mu_-)
		\end{equation*}
		if one of the following holds:
		\begin{enumerate}[(a)]
			\item $[x_-,\bar{x}_-] \in \widetilde{\mathcal{P}}_{\I}(H_{\mu_-})$, $[x_+,\bar{x}_+]\in \widetilde{\mathcal{P}}_{\I}(H_{\mu_+})$; 
			\item $[x_-,\bar{x}_-] \in \widetilde{\mathcal{P}}_{\I}(H_{\mu_-})$, $[x_+,\bar{x}_+]\in \widetilde{\mathcal{P}}_{\II}^\crit(H_{\mu_+})$; 
			\item $[x_-,\bar{x}_-] \in \widetilde{\mathcal{P}}_{\II}^\crit(H_{\mu_-})$, $[x_+,\bar{x}_+]\in \widetilde{\mathcal{P}}_{\II}^\crit(H_{\mu_+})$ provided that they have the same index and $J=J_\eta$ is sufficiently stretched. (In fact, we can explicitly determine $\mathbf{u}$ from which the claim also follows, see Proposition \ref{prop:inc}.)
		\end{enumerate}
	\end{prop}
	\begin{proof}
		By Remark \ref{rmk:zerocap}, we may assume  $\bar{x}_-=\bar{x}_-^\mathrm{fib}$.

	\medskip

	\noindent Proof of (a). For $[x_-,\bar{x}_-] =[\mu_-,\I,q_-,0]$ and $[x_+,\bar{x}_+] =[\mu_+,\I,q_+,A]$, we recall
			\[
				\mathcal{A}_{H_{\mu_-}}([x_-,\bar{x}_-])\approx \frac{\rho}{k}\mu_-,\qquad 
				\mathcal{A}_{H_{\mu_+}}([x_+,\bar{x}_+])\approx \frac{\rho}{k}\mu_+ +\omega(A).
			\]
			By Lemma \ref{lem:intersect}.(a), we have
			\begin{equation*}
				k\omega(A)=A\cdot[D]=[\mathbf{u}]\cdot[D]\ge0,
			\end{equation*}
			and this yields
			\begin{equation*}
				\mathcal{A}_{H_{\mu_+}}([x_+,\bar{x}_+])-\mathcal{A}_{H_{\mu_-}}([x_-,\bar{x}_-])\approx\frac{\rho}{k}(\mu_+-\mu_-) +\omega(A)\ge \frac{\rho}{k}(\mu_+-\mu_-).
			\end{equation*}
			
		\medskip
		
		\noindent Proof of (b).	 For $[x_-,\bar{x}_-] =[\mu_-,\I,q_-,0]$ and $[x_+,\bar{x}_+] =[\mu_+,\II,q_+,m_+,A]$, we recall
			\[
				\mathcal{A}_{H_{\mu_-}}([x_-,\bar{x}_-])\approx \frac{\rho}{k}\mu_-,\qquad 
				\mathcal{A}_{H_{\mu_+}}([x_+,\bar{x}_+])\approx \frac{\rho+\epsilon}{k}m_++\frac{\rho}{k}\mu_+ +\omega(A).
			\]
			Since $[\mathbf{u}\#(-\bar{x}_+^\mathrm{fib})]=A$, by Lemma \ref{lem:intersect}.(a), we have 
			\[
			k\omega(A)=A\cdot [D]\geq -m_+.
			\]
			Using this, together with the fact that $m_+<0$ (see \eqref{eq:abbrev}), we conclude
			\begin{align*}
				\mathcal{A}_{H_{\mu_+}}([x_+,\bar{x}_+])-\mathcal{A}_{H_{\mu_-}}([x_-,\bar{x}_-])&\approx\frac{\rho}{k}(\mu_+-\mu_-)+\omega(A)+\frac{\rho+\epsilon}{k}m_+\\[.5ex]
				&\ge \frac{\rho}{k}(\mu_+-\mu_-)+\frac{m_+}{k}(\rho+\epsilon-1)\\[.5ex]
				&\ge \frac{\rho}{k}(\mu_+-\mu_-).
			\end{align*}
			
			\medskip
			
			\noindent Proof of (c). For $[x_-,\bar{x}_-]=[\mu_-,\II,q_-,m_-,0]$ and $[x_+,\bar{x}_+]=[\mu_+,\II,q_+,m_+,A]$, we have
			\begin{align*}
				\mathcal{A}_{H_{\mu_-}}([x_-,\bar{x}_-])&\approx \frac{\rho+\epsilon}{k}m_-+\frac{\rho}{k}\mu_-\\[.5ex]
				\mathcal{A}_{H_{\mu_+}}([x_+,\bar{x}_+])&\approx \frac{\rho+\epsilon}{k}m_++\frac{\rho}{k}\mu_+ +\omega(A)
			\end{align*}
			and
			\begin{align*}
				\mu([x_-,\bar{x}_-])&=2m_-+\ind _{f_D}(\pi(q_-))-(n-1)+e_-\\[.5ex]
				\mu([x_+,\bar{x}_+])&= 2m_++2c_1^{TM}(A)+\ind _{f_D}(\pi(q_+))-(n-1)+e_+
			\end{align*}
			where $e_\pm\in\{0,1\}$. In addition, $A=[\bar{x}_-^\mathrm{fib}\#\mathbf{u}\#(-\bar{x}_+^\mathrm{fib})]$ and Lemma \ref{lem:intersect}.(a) gives
			\begin{equation*}
				k\omega(A)=A\cdot[D]\ge m_--m_+.
			\end{equation*}
			We perform neck-stretching and denote by $(\mathbf{u}_1,\mathbf{u}_2)$ the limit configuration of $\mathbf{u}$  as in Section \ref{sec:neckstr}. Arguing as in the proof of Proposition \ref{prop:empty}.(c), we compute
			\begin{align*}
				0&=\mu([x_+,\bar{x}_+])-\mu([x_-,\bar{x}_-])\\
				&=\ind_{f_{D}}(\pi(q_+))-\ind_{f_D}(\pi(q_-))	+2c_1^{TM}(A) 
				+2(m_+-m_-) +(e_+-e_-)\\
				&=\underbrace{(\ind_{f_{D}}(\pi(q_+))-\ind_{f_D}(\pi(q_-))	+2c_1^{TD} ([\pi\circ\mathbf{u}_2]) }_{\geq 0}
				+\underbrace{2(m_+-m_-+k\omega(A))}_{\geq 0} \\
				&\quad +\underbrace{2(\tau-k)\omega(\bar{\mathbf{u}}_1)}_{\geq0} +\underbrace{(e_+-e_-)}_{\geq -1}.
			\end{align*}
			In particular, $\omega(\bar{\mathbf{u}}_1)=0$ and $ k\omega(A)=m_--m_+$. Thus, $\bar{\mathbf{u}}_1=\emptyset$, and $A$ is represented by the sphere $\bar{x}_-^\mathrm{fib}\#\mathbf{u}_2\#(-\bar{x}_+^\mathrm{fib})$ in $\widehat{E}$. From $\omega=\pi^*\omega_D+\frac{1}{k}d(r^2\alpha)$ on $\widehat{E}$, we deduce
			\[
		   m_--m_+=k\omega(A)=k\omega_D([\pi\circ \mathbf{u}_2])\geq 0.
		   \]
			Using this, we conclude 
			\begin{align*}
				\mathcal{A}_{H_{\mu_+}}([x_+,\bar{x}_+])-\mathcal{A}_{H_{\mu_-}}([x_-,\bar{x}_-])&\approx\frac{\rho}{k}(\mu_+-\mu_-)+\frac{\rho+\epsilon}{k}(m_+-m_-)+\omega(A)\\[.5ex]
				&=\frac{\rho}{k}(\mu_+-\mu_-)+\frac{m_--m_+}{k}(1-\rho-\epsilon)\\[.5ex]
				&\ge \frac{\rho}{k}(\mu_+-\mu_-).
			\end{align*}
	\end{proof}
	
	\subsection{Non-contribution of upper orbits}
	
	Now we specify the acceleration data to be used.
	We take a cofinal family $\mathcal{H}=\{H_i\}_{i\in\N}$, where $H_i=H_{\mu_i}$ with $\mu_i\in\R_+\setminus\Z$ satisfying $\mu_1\gg0$, $\mu_1\notin\mathbb{Q}$, and $\mu_{i+1}=\mu_i+1$. Then, any consecutive pair $\mu_i$ and $\mu_{i+1}$ satisfies the condition in Proposition \ref{prop:empty}. Let $\widetilde{H}_{i}$ be a monotone homotopy between $H_{i}$ and $H_{i+1}$. To apply the results in the previous section, we need to work with a sufficiently stretched almost complex structure. However, it is not clear whether a single stretched $J$ can be used for all $\widetilde{H}_{i}$, $i\in\N$. For this reason, we also increase the stretching parameter as $i\nearrow\infty$. Let $\{J_\eta\}_{\eta\in\R_+}$ be a family of stretched compatible almost complex structures such that the correspondence in Section \ref{sec:neckstr} holds for $(H_{i},J_l)$ and $(\widetilde{H}_{i},J_l)$ for each $i\in\N$ and all $l\geq i$. In particular, $(H_{i},J_l)$ and $(\widetilde{H}_{i},J_l)$ are  regular pairs for all $l\geq i$, see Remark \ref{rmk:transversality}. We denote $\mathcal{J}=\{J_i\}_{i\in\N}$.

	Now we consider the one-ray 	
	\begin{equation}\label{eq:one-ray}
	\CF^*_{(-\infty,b)}(H_{1},J_{1})\stackrel{\Psi_1\circ \jmath_1}{\longrightarrow} \CF^*_{(-\infty,b)}(H_{2},J_{2})\stackrel{\Psi_2\circ \jmath_2}{\longrightarrow}\CF^*_{(-\infty,b)}({H}_{3},J_{3})\longrightarrow \cdots		
	\end{equation}
	where 
	\begin{equation}\label{eq:continuation}
	\Psi_i:\CF^*_{(-\infty,b)}(H_{i},J_{i+1})\longrightarrow\CF^*_{(-\infty,b)}({H}_{i+1},J_{i+1})
	\end{equation}
	is a continuation homomorphism induced by $(\widetilde{H}_{i},J_{i+1})$ and 
	\[
	\jmath_i:\CF^*_{(-\infty,b)}(H_{i},J_{i})\longrightarrow\CF^*_{(-\infty,b)}({H}_{i},J_{i+1})
	\]
	is the identity map, which makes sense since the two chain complexes are identical due to the correspondence in Section \ref{sec:neckstr}.
	
	\begin{lem}\label{lem:id_chainmap}
	The map $\jmath_i$ consists of a composition of continuation homomorphisms.
	\end{lem}
	\begin{proof}
		We claim that for each $\eta\in[i,i+1]$, there is a neighborhood $\mathcal{U}\subset [i,i+1]$ of $\eta$ such that, for any $\eta'\in\mathcal{U}$, the continuation homomorphism between $\CF^*_{(-\infty,b)}(H_{i},J_{\eta})$ and $\CF^*_{(-\infty,b)}({H}_{i},J_{\eta'})$ induced by a homotopy between $J_\eta$ and $J_{\eta'}$ is the identity map. Assume for contradiction that there are a sequence $\eta_\nu\to\eta$, $\nu\in\N$ and  a sequence of flow lines with Floer continuation cascades for $(H_i,\widetilde{J}_\nu)$, where $\widetilde{J}_\nu$  is a homotopy between $J_\eta$ and $J_{\eta_\nu}$, connecting $[x_\pm,\bar{x}_\pm]\in \widetilde{\mathcal{P}}^\crit(H_i)$ satisfying
		\begin{equation}\label{eq:contradiction}
		[x_-,\bar{x}_-]\neq [x_+,\bar{x}_+],\qquad \mu([x_-,\bar{x}_-])=\mu([x_+,\bar{x}_+]).			
		\end{equation}
		 However, this sequence of flow lines with Floer continuation cascades converges to a collection of flow lines with Floer cascades for $(H_i,J_\eta)$ connecting $[x_-,\bar{x}_-]$ and $[x_+,\bar{x}_+]$. This contradicts \eqref{eq:contradiction}. Thus, the claim is proved, and the lemma follows since $[i,i+1]$ is compact.
		\end{proof}

	From now on, we use the chain complex 
	\begin{equation*}
		\SC^*_M(\partial U_\rho ,\mathcal{H},\mathcal{J}):=\varprojlim_{b\rightarrow\infty}\varinjlim_{i\rightarrow\infty}\CF^*_{(-\infty,b)}(H_{i},J_{i}),
	\end{equation*}
	where the direct limit is taken via $\Psi_i\circ \jmath_i$. By Remark \ref{rmk:Jchange} and Lemma \ref{lem:id_chainmap}, it is quasi-isomorphic to the chain complex defined with another acceleration data. We denote by $d_\SC$ the induced boundary operator and by $\SH^*_M(\partial U_\rho)$ the resulting homology. We show below that it is enough to consider type III orbits to compute $\SH^*_M(\partial U_\rho)$.

	From the sequence of natural injective maps
	\begin{align*}
		\iota_i:\widetilde{\mathcal{P}}^{\crit}_{\III}(H_{i})&\longrightarrow \widetilde{\mathcal{P}}^{\crit}_{\III}(H_{i+1}) \\[.5ex]
		[\mu_i,\III,q,m,A]&\longmapsto[\mu_{i+1},\III,q,m,A],
	\end{align*}
	we have
	\begin{equation*}
	\begin{split}
		\iota_{\infty,i}:\widetilde{\mathcal{P}}^{\crit}_{\III}(H_{i}) &\longrightarrow \widetilde{\mathcal{P}}^{\crit}_{\III}:=\varinjlim_i\widetilde{\mathcal{P}}^{\crit}_{\III}(H_{i}) \cong \crit f_{\partial E}\times \Z \times \pi_2(M)/\ker \omega	\\	
		[\mu_i,\III,q,m,A]&\longrightarrow [q,m,A].
	\end{split}
	\end{equation*}
	We define the action of $[q,m,A]\in \widetilde{\mathcal{P}}^{\crit}_{\III}$ by
	\begin{equation*}
		\mathcal{A}([ q,m,A]):=\frac{m}{k}\rho+\omega(A).
	\end{equation*}
	It can be thought of as the limit of the action of type III orbits in \eqref{eq:action_index}. 
	
		\begin{theorem}\label{thm:equal}
		Elements of $\SC^*_M(\partial U_\rho ,\mathcal{H},\mathcal{J})$ have the form 
		\[
		\sum_{j=1}^{\infty}a_j \xi_j,\qquad a_j\in \mathbb{Z},\;\; \xi_j=[q_j,m_j,A_j]\in \widetilde{\mathcal{P}}^{\crit}_{\III}, \;\;\lim_{j\rightarrow\infty}\mathcal{A}(\xi_j)=\infty.
		\]
		Moreover, for any $\xi_\pm\in\widetilde{\mathcal{P}}_{\III}^\crit$, there exist $i\in\N$ and $\xi_{i,\pm}\in \widetilde{\mathcal{P}}_{\III}^\crit(H_i)$ such that 
		\begin{equation*}
			\xi_\pm =\iota_{\infty,i}(\xi_{i,\pm}),\qquad \langle d_{\SC}\xi_+,\xi_-\rangle=\langle d_{(H_i,J_i)} \xi_{i,+}, \xi_{i,-}\rangle,
		\end{equation*}
		where $d_{(H_i,J_i)}$ is the boundary operator of $\CF^*(H_{i},J_{i})$.
	\end{theorem}
	\begin{proof}
		The theorem follows from \eqref{eq:eltsofsc}, Proposition \ref{prop:inc}, and Corollary \ref{cor:corresp} below. 
	\end{proof}

	\begin{prop}\label{prop:inc}
		The continuation maps $\Psi_{i}$ in \eqref{eq:continuation} restricted to type III orbits coincide with the injective maps $\iota_i$. More precisely, for any $[x,\bar{x}]\in\widetilde{\mathcal{P}}_{\III}^\crit(H_i)$, $\Psi_i([x,\bar{x}])-\iota_i([x,\bar{x}])$ does not contain type III orbits. An analogous statement holds also for type II orbits, replacing type III by type II in the definition of $\iota_i$.
	\end{prop}
	\begin{proof}
		Let $\mu_-=\mu_i$ and $\mu_+=\mu_{i+1}$, and let $\mathbf{u}\in{\mathcal{M}}([x_-,\bar{x}_-],[x_+,\bar{x}_+],\widetilde{H}_{i},J_{i+1})$ for $[x_\pm,\bar{x}_\pm] \in \widetilde{\mathcal{P}}_{\III}^\crit(H_{\mu_\pm})$ with the same index. We write $[x_-,\bar{x}_-]=[\mu_-,\III,q_-,m_-,0]$ and $[x_+,\bar{x}_+]=[\mu_+,\III,q_+,m_+,A]$. Adapting the proof of Proposition \ref{prop:dec}.(c) to type III orbits,
		we have
		\begin{align*}
				0&=\mu([x_+,\bar{x}_+])-\mu([x_-,\bar{x}_-])\\
				&=\ind_{f_{D}}(p_+)-\ind_{f_D}(p_-)	+2c_1^{TM}(A) 
				+2(m_+-m_-) +(e_+-e_-)\\
				&=\underbrace{(\ind_{f_{D}}(p_+)-\ind_{f_D}(p_-)	+2c_1^{TD}([\pi\circ\mathbf{u}_2]) }_{:=(\mathbf{a})\geq 0}
				+\underbrace{2(m_+-m_-+k\omega(A))}_{:=(\mathbf{b})\geq 0} \\
				&\quad +\underbrace{2(\tau-k)\omega(\bar{\mathbf{u}}_1)}_{:=(\mathbf{c})\geq0} +\underbrace{(e_+-e_-)}_{:=(\mathbf{d})\geq -1}.
			\end{align*}
		where $p_\pm:=\pi(q_\pm)$ and  $e_\pm \in \{0,-1\}$. 
		Therefore, we have $(\mathbf{b})=(\mathbf{c})=0$, and ${\mathbf{u}}_1=\emptyset$.  We write  $A=[\bar{x}_-^\mathrm{fib}\#\mathbf{u}_2\#(-\bar{x}_+^\mathrm{fib})]\in\pi_2(\widehat E)$.

		\medskip
		
		\noindent Case 1) $(\mathbf{a})=0$ and $e_+=e_-$. 
		
		\medskip
			By Remark \ref{rem:index_0_proj}, $p:=p_-=p_+$ and $\mathbf{u}_2$ is contained in the fiber $\widehat{E}_{p}$. Therefore, $A=0$ and $m_+=m_-$. We write $\mathbf{u}_2=(c_0,u_1,\dots ,u_N,c_N)$, where each $u_i$ is a Floer cylinder, except for one, which is a Floer continuation cylinder. The winding number does not decrease along $\mathbf{u}_2$ by positivity of zeroes. Thus, every $u_j$ preserves the wi nding number and is contained in $\widehat{E}_p\setminus\{0\}$. 
			We also observe that $\ev_+(u_j)$ for any $j$ cannot be of type II since otherwise, by Lemmas \ref{lem:mp}, \ref{lem:ab}, and \ref{lem:conti_empty}, $\ev_+(u_{j+1})$ is again of type II, which contradicts $\ev_+(u_N)\in{\mathcal{P}}_{\III}(H_{\mu_+})$. 
			Moreover, $\ev_+(u_j)$ is not of type IV for any $j$. Otherwise $u_j=u_N$ and $\ev_+(u_N)$ has type IV because the constant orbit at the origin is isolated in $\widehat{E}_{p}$. Thus, the asymptotic orbits of $u_j$ are of type III for every $j$. Moreover, if $u_j$ were a Floer cylinder, then it is trivial since the asymptotic orbits have the same winding number. Therefore, we conclude $\mathbf{u}_2=(c_0,u,c_1)$, where $u$ is a Floer continuation cylinder. By Lemma \ref{lem:constconti}, $c_0$ and $c_1$ are constant curves, $[x_+,\bar{x}_+]=\iota_i([x_-,\bar{x}_-])$, and such $u$ uniquely exists, which is positively oriented as shown in \cite[Lemma 4.8]{BKK24}. 
			 
			\medskip
			
		\noindent Case 2) $(\mathbf{a})=1$ and $(e_-,e_+)=(0,-1)$.
		
		\medskip
		
			The projection $\pi\circ \mathbf{u}_2$ gives an element in 
			$\mathcal{N}^*_N(p_-,p_+;[\pi\circ \mathbf{u}_2];J_D)$, where, by Remark \ref{rem:index_0_proj}, $N=0$ and $[\pi\circ \mathbf{u}_2]=0$. Thus, every Floer cylinder in $\mathbf{u}_2$ is contained in a fiber of $\widehat{E}$ and the projection $\pi\circ \mathbf{u}_2$ is a gradient flow line of $f_D$ connecting $p_-$ to $p_+$. We claim that $\mathbf{u}_2$ takes the form $(c_0,u,c_1)$, where $u$ is a Floer continuation cylinder connecting type III orbits. As observed in Case 1, no Floer (continuation) cylinder in $\mathbf{u}_2$ has a type II orbit as an asymptotic orbit. Furthermore, no orbit of type IV is involved in $\mathbf{u}$, and hence none is involved in $\mathbf{u}_2$. Indeed, otherwise, by Lemma \ref{lem:intersect}.(a), the intersection number of $\mathbf{u}$ with $D$ is at least $1$, which contradicts the fact that $(\mathbf{b})=0$ implies that $\mathbf{u}\cdot D=0$ due to  $k\omega(A)=m_--m_++\mathbf{u}\cdot D$. Since any Floer cylinder in a fiber that connects type III orbits and avoids the zero section is trivial, the claim follows. 
			
			The gradient flow line of $f_D$ obtained by projecting $\mathbf{u}_2$ equals the concatenation of $\pi\circ c_0$ and $\pi \circ c_1$. Moreover, $u$ lies in the fiber over the concatenation point and does not intersect the zero section.  
			By Lemma \ref{lem:constconti}, $\theta(\ev_-(u)(0))=\theta(\ev_+(u)(0))$, and therefore, by identifying $\ev_+(u)$ and $\ev_-(u)$, we can  concatenate $c_0$ and $c_1$, which yields a gradient flow line $c_1\#c_0$ of $f_{\partial E}$ from $q_-$ to $q_+$. It is nonconstant since $\pi(c_1\#c_0)$ connects $p_-$ to $p_+$ with $\ind_{f_D}(p_+)-\ind_{f_D}(p_-)=1$. However, 
			\[
			\ind_{f_{\partial E}}(q_+)-\ind_{f_{\partial E}}(q_-)= \ind_{f_D}(p_+)-\ind_{f_D}(p_-)+e_+-e_-=0.
			\]
			This contradiction proves that this case does not occur.
			
			\medskip

		We have established that $\Psi_i([x_-,\bar{x}_-])$ contains a unique  type III orbit equal to $\iota_i([x_-,\bar{x}_-])$. 
	An analogous argument applies to orbits of type II.
	\end{proof}

		We denote by $\widetilde{\mathcal{P}}^\crit_{\II_-}(H_{i})\subset \widetilde{\mathcal{P}}^\crit_{\II}(H_{i})$  the subset consisting of elements with winding number $m$ satisfying the condition $m\le -\frac{\rho}{\epsilon}\mu_i+C_2$ in Proposition \ref{prop:res}. We write $\widetilde{\mathcal{P}}^\crit_{\II_+}(H_{i}):=\widetilde{\mathcal{P}}^\crit_{\II}(H_{i})\setminus\widetilde{\mathcal{P}}^\crit_{\II_-}(H_{i})$. We abbreviate $\Phi_i:=\Psi_i\circ \jmath_i$ and 
		\[
		\Phi_{i_2,i_1}:=\Phi_{i_2-1}\circ\cdots\circ \Phi_{i_1}: \CF^*_{(-\infty,b)}(H_{i_1},J_{i_1})\to \CF^*_{(-\infty,b)}(H_{i_2},J_{i_2}).
		\]
		We also denote the canonical map into the direct limit by
			   \[
    \Phi_{\infty,i_1}:\CF^*_{(-\infty,b)}(H_{i_1},J_{i_1})\longrightarrow \varinjlim_{i\rightarrow\infty}\CF^*_{(-\infty,b)}(H_{i},J_{i}).
   \]

		\begin{cor}\label{cor:12van}
			Let ${i_0}$ be large so that $C_1 > -\frac{\rho}{\epsilon}\mu_{i_0}+C_2$ holds for the constants $C_1,\,C_2$ in Proposition \ref{prop:res}. Assume $\beta_{i_1}\in\CF^*_{(-\infty,b)}(H_{i_1},J_{i_1})$ consists of orbits of type $\I$ and $\II_+$. Then, for every sufficiently large $i_2\geq i_1$, $\Phi_{i_2,i_1}(\beta_{i_1})$ vanishes in $\CF^*_{(-\infty,b)}(H_{i_2},J_{i_2})$. 
		\end{cor}
		\begin{proof}
		By Propositions \ref{prop:empty}, \ref{prop:res}, and \ref{prop:inc}, the continuation maps send type I and $\II_+$ orbits only in the following direction:
		\[
		\I \to \I,\,\II_+,\qquad \II_+ \to \I, \II_+
		\]
			 Hence, by  Proposition \ref{prop:dec}, any orbit appearing in $\Phi_{i_2,i_1}(\beta_{i_1})$ has action greater than $b$.
		\end{proof}

		\begin{cor}\label{cor:corresp}
			For every  element $\beta\in\displaystyle{\varinjlim_{i\rightarrow\infty}}\CF^*_{(-\infty,b)}(H_{i},J_{i})$, there exists $l_0\in \mathbb{N}$ such that, for any $l\geq\l_0$, there exists $\beta_{l}^{\III}\in \CF^*_{(-\infty,b)}(H_{l},J_{l})$ that consists only of type III orbits and satisfies $\Phi_{\infty,l}(\beta_{l}^{\III})=\beta$.
		\end{cor}
		\begin{proof}
			We take any $i_1\in \mathbb{N}$ and $\beta_{i_1}\in \CF^*_{(-\infty,b)}(H_{i_1},J_{{i_1}})$ such that $\Phi_{\infty,i_1}(\beta_{i_1})=\beta$. For any  $i_2>i_1$, by Proposition \ref{prop:empty}, $\beta_{i_2}:=\Phi_{i_2,i_1}(\beta_{i_1})\in \CF^*_{(-\infty, b)}(H_{i_2},J_{i_2})$ does not have type IV orbits. Also, by Propositions \ref{prop:empty}, \ref{prop:res}, and \ref{prop:inc}, any composition of $\Phi_i$ does not send orbits of type $\I$, $\II_+$, or $\III$ to orbits of type $\II_-$. Therefore, for any $i_3>i_2$, any type $\II_-$ orbit in $\Phi_{i_3,i_2}(\beta_{i_2})$ is the image of a type $\II_-$ orbit in $\beta_{i_2}$. This together with Proposition \ref{prop:dec}.(c) implies that, for any sufficiently large $l\geq i_3$, type $\II_-$ orbits in $\Phi_{l,i_2}(\beta_{i_2})$ have action greater than $b$, and consequently, $\Phi_{l,i_2}(\beta_{i_2})\in \CF^*_{(-\infty, b)}(H_{l},J_{l})$ does not have type $\II_-$ orbits. We write $\Phi_{l,i_2}(\beta_{i_2})=\beta_l^{\I}+\beta_l^{\II_+}+\beta_l^{\III}$, where $\beta_l^{\III}$ consists of type III orbits, and similarly for others. By Corollary \ref{cor:12van}, $\Phi_{\infty,l}(\beta_l^{\I}+\beta_l^{\II_+})=0$.   Hence, $\beta_l^{\III}$ satisfies ${\Phi}_{\infty,l}(\beta_{l}^{\III})=\beta$.
		\end{proof}

	\subsection{Explicit description of the boundary operator}\label{sec:boundary}
	
	To describe the boundary operator of $\SC^*_M(\partial U_\rho,\mathcal{H},\mathcal{J})$, we consider the Morse cochain complex of $f_D$ with coefficients in $\Gamma$,
	\begin{equation}\label{eq:morse_D}
		\CM^*(f_D;\Gamma)=\bigoplus_{p\in \crit f_D} \Gamma\langle p\rangle,
		\qquad
			\Gamma = \begin{cases}
		\Z  & \text{if }\; \omega(\pi_2(M))=0,\\
		\Z[T,T^{-1}] &\text{if }\; \omega(\pi_2(M))\neq0.
	\end{cases}
	\end{equation}
	The degree of $T$ is $2(\tau-k)m_M$, where $m_M$ is the positive generator of $\omega(\pi_2(M))=m_M\Z$. Let $m_D$ denote the nonnegative generator of $\omega_D(\pi_2(D))=m_D\Z$. Recall that $c_1^{TD}=(\tau -k)[\omega_D]_\Z$. Thus, $(\tau-k)m_M$ divides $(\tau-k)m_D$, the minimal Chern number of $D$, and they coincide when $\pi_2(D)\to \pi_2(M)$ is an isomorphism. We write 
	 \[
	 d_D:\CM^*(f_D;\Gamma) \longrightarrow \CM^{*+1}(f_D;\Gamma)
	 \]
	 for the boundary operator of this complex. The cohomology of this chain complex is isomorphic to the quantum cohomology $\QH^*(D;\Gamma)\cong\H^*(D;\Z)\otimes \Gamma$  of $(D,\omega_D)$.
	
	Let $[\omega_D]_\Z\in \H^2(D;\Z)$ be the pullback of $[\omega]_\Z$ under the inclusion $D\subset M$.
	The chain map for the quantum cup product with $k[\omega_D]_\Z$ is defined as follows. 
	 Let $W^s(p)$ and $W^u(p)$ be the stable and unstable manifolds of the gradient flow of $f_D$ at $p\in\crit f_D$, respectively. Let $\Sigma\subset D$ be a closed codimension two submanifold representing $\PD(k[\omega_D]_\Z)$. Then, the $\Gamma$-linear map 
	\[
	\begin{split}
	\delta^{k[\omega_D]_\Z}_1:\CM^*(f_D;\Gamma)&\longrightarrow \CM^{*+2}(f_D;\Gamma)\\
	p_-&\longmapsto \sum_{p_+}   \#(W^u(p_-)\cap W^s(p_+)\cap \Sigma)\,p_+ ,
	\end{split}
	\]
	where the sum ranges over all $p_+\in\crit f_D$ with $\ind_{f_D}(p_+)=\ind_{f_D}(p_-)+2$, induces the classical cup product with $k[\omega_D]_\Z$ on cohomology. To include the quantum contribution, let ${\mathcal{N}}^*(p_-,p_+;A;J_D)$ be the moduli space of simple $J_D$-holomorphic spheres $v:\mathbb{CP}^1\to D$ such that $[v]=A\in\pi_2(D)\setminus\{0\}$, and $v(0)\in W^u(p_-)$ and $v(\infty)\in W^s(p_+)$. There is a free $\C^*$-action on the moduli space, and we denote the quotient space by $\overline{\mathcal{N}}^*(p_-,p_+;A;J_D)$. We consider a $\Gamma$-linear map
	\[
	\begin{split}
	\delta^{k[\omega_D]_\Z}_2:\CM^*(f_D;\Gamma)&\longrightarrow \CM^{*+2}(f_D;\Gamma)\\
	p_-&\longmapsto \sum_{(p_+,A)}k\omega_D(A) \cdot \# \overline{\mathcal{N}}^*(p_-,p_+;A;J_D)\,p_+ T^{\frac{c_1^{TD}(A)}{(\tau-k)m_M}},
	\end{split}
	\]
	where the sum ranges over all $(p_+,A)\in\crit f_D\times (\pi_2(D)\setminus\{0\})$ with $\ind_{f_D}(p_+)+2c_1^{TD}(A)=\ind_{f_D}(p_-)+2$. Note that each element in $\overline{\mathcal{N}}^*(p_-,p_+;A;J_D)$ has intersection number with $\Sigma$ equal to $k\omega_D(A)$. Then 
	\[
	\delta^{k[\omega_D]_\Z}:=\delta^{k[\omega_D]_\Z}_1+\delta^{k[\omega_D]_\Z}_2
	\] 
	is a chain map that defines the quantum cup product with $k[\omega_D]_\Z$ on cohomology. 

	To state the main result of this section, we also need the moduli space $\overline{\mathcal{M}}_1(p,A)$ of equivalence classes of simple $J$-holomorphic spheres $v:\mathbb{C}P^1\rightarrow M$ such that $[v]=A\in\pi_2(M)$ and $v$ intersects $D$ at a single point $v(\infty)=p\in D$ with multiplicity $v\cdot D=k\omega(A)$. Two curves are equivalent if one is a reparametrization of the other by a biholomorphic map on $\mathbb{CP}^1$ fixing $\infty$. 
			The dimension of this moduli space equals $2n+2c_1^{TM}(A)-2k\omega(A)-4-(2n-2)$, which is zero exactly when $(\tau-k)\omega(A)=1$. Since $(\tau-k)m_M$ is an integer, $(\tau-k)\omega(A)=1$ yields $\omega(A)=m_M$. We denote the count of such curves by 
			\begin{equation}\label{eq:N_1}
				N_1(M,D):= 
				\begin{cases}
					0 &  (\tau-k)m_M\geq 2\\[.5ex]
					\displaystyle\sum_{A}\#\overline{\mathcal{M}}_1(p,A) & (\tau-k)m_M=1,
				\end{cases}
			\end{equation}
	where the sum ranges over all $A\in\pi_2(M)$ with $\omega(A)=m_M$. This quantity is independent of the choice of $p$. Note that $N_1(M,D)\neq 0$ only when $\deg T=2$. If $\tau>0$ is an integer, $N_1(M,D)\neq 0$ implies $\tau=k+1$ and $m_M=1$. We refer to \cite{DL19a,DL19b,AK23,BKK24} for a detailed account on the discussion above, including the orientations of the relevant moduli spaces.

	Recall that for each critical point $p$ of $f_D$, there are two corresponding critical points $\hat{p},\check{p}$ of $f_{\partial E}$, see \eqref{eq:hat_check_p}. For each $m\in\Z$, we have two canonical inclusions
	\begin{equation}\label{eq:incl}
			\begin{split}
		\check{\iota}_m:\CM^*(f_D;\Gamma)&\rightarrow \SC^{*-n+2m}_M(\partial U_\rho,\mathcal{H},\mathcal{J}), \\ p T^\ell &\mapsto [ \check{p},m-k \ell m_M,\ell A],
		\\[.5ex]
		\hat{\iota}_m:\CM^*(f_D;\Gamma)&\rightarrow \SC^{*-n+2m+1}_M(\partial U_\rho,\mathcal{H},\mathcal{J}),\\
		p T^\ell &\mapsto [ \hat{p},m-k\ell m_M,\ell A].
	\end{split}
	\end{equation}
	where $A\in \pi_2(M)/\ker \omega$ is the class with $\omega(A)=m_M$.
	
	\begin{remark}
		Let $v:\mathbb{CP}^1\to D\subset M$ satisfy $v(\infty)=p$ and $\omega_D(v)=m_D$. Then $2c_1^{TD}(v)=2(\tau-k)m_D$ equals the degree of $T^{\frac{m_D}{m_M}}$. We lift $v|_{\mathbb{CP}^1\setminus\{\infty\}}$ to $\tilde v:\mathbb{CP}^1\setminus\{\infty\}\to E\setminus D$ so that $\tilde v$ is a capping disk of a periodic orbit $x$ that winds around the fiber circle over $p$ with multiplicity $-k\omega_D(v)=-km_D$. Moreover, $\tilde v$ is homotopic to the gluing of $\bar{x}^\mathrm{fib}$ with a sphere in $M$ homotopic to $v$. This lifting corresponds to the maps in \eqref{eq:incl} with $m=0$ and $\ell=\frac{m_D}{m_M}$. To be precise, $(p,v)$ corresponds to $pT^{\frac{m_D}{m_M}}$ and the lift $[x,\tilde v]$ of $(p,v)$ satisfying $x(0)=\check p$ equals $\check{\iota}_m(pT^\frac{m_D}{m_M})=[\check{p},m-km_D,[v]]$.
	\end{remark}
	
	 By the following theorem, $\check{\iota}_m$ is a chain map, whereas $\hat{\iota}_m$  is not necessarily. 
	 
	\begin{theorem}\label{thm:main}
		The boundary operator $d_{\SC}$ in $\SC^*_M(\partial U_\rho,\mathcal{H},\mathcal{J})$ can be written as
		\begin{equation*}
			d_{\SC}=d_0+d_1+d_2+d_W,
		\end{equation*}
		where 
		\[
		\begin{cases}
			d_0[\check{p},m,0]=0 \qquad\qquad\qquad  & d_0[\hat{p},m,0]=[\check{p},m+1,0]\\[.5ex]
			d_1[\check{p},m,0]=-\check{\iota}_m(d_D(p)) &
				d_1[\hat{p},m,0]=\hat{\iota}_m(d_D(p)) \\[.5ex]
			d_2[\check{p},m,0]=0 &
				d_2[\hat{p},m,0]=\check{\iota}_m(\delta^{k[\omega_D]_\Z}(p))\\[.5ex]
			d_W[\check{p},m,0]=0 &
				d_W[\hat{p},m,0]=kN_1(M,D) [\check{p},m-km_M,A],
		\end{cases}
		\]
		where in the last case $A \in \pi_2(M)/\ker \omega$ is the class satisfying $\omega(A) = m_M$. (The precise definition of this decomposition of $d_\SC$ is explained in the proof.)
	\end{theorem}
	\begin{proof}

		By Theorem \ref{thm:equal}, it suffices to investigate elements 
$\mathbf{u} \in {\mathcal{M}}([x_-,\bar{x}_-],[x_+,\bar{x}_+],{H}_i,J_i)$ 
for $[x_\pm,\bar{x}_\pm] \in \widetilde{\mathcal{P}}_{\mathrm{III}}^{\mathrm{crit}}(H_{i})$. 
Furthermore, as discussed in Section \ref{sec:neckstr}, analyzing $\mathbf{u}$ is equivalent 
to studying its corresponding pair $(\mathbf{u}_1,\mathbf{u}_2)$ obtained via neck stretching. 
 As before, we write $[x_-,\bar{x}_-]=[\mu_i,\III,q_-,m_-,0]$ and $[x_+,\bar{x}_+]=[\mu_i,\III,q_+,m_+,A]$, and deduce
			\begin{align*}
				1&=\mu([x_+,\bar{x}_+])-\mu([x_-,\bar{x}_-])\\
				&=\underbrace{(\ind_{f_{D}}(p_+)-\ind_{f_D}(p_-)	+2c_1^{TD}([\pi\circ\mathbf{u}_2]) }_{:=(\mathbf{a}) \geq 0}
				+\underbrace{2(m_+-m_-+k\omega(A))}_{:=\mathbf{(b)}\geq 0} \\
				&\quad +\underbrace{2(\tau-k)\omega(\bar{\mathbf{u}}_1)}_{:=\mathbf{(c)}\geq0} +\underbrace{(e_+-e_-)}_{:=\mathbf{(d)}\geq -1}.
			\end{align*}
	where $p_\pm:=\pi(q_\pm)\in D$.

			The contribution to the boundary operator given by pairs $(\mathbf{u}_1,\mathbf{u}_2)$ with 
$\mathbf{u}_1 \neq \emptyset$ is denoted by $d_W$. In the case where 
$\mathbf{u}_1 = \emptyset$, we project $\mathbf{u}_2$ to $D$. (This can be thought of as 
the projection of $\mathbf{u}$ to $D$, which is well-defined because $\mathbf{u}_1 = \emptyset$ 
implies that $\mathbf{u}$ is contained in $\mathcal{U}_{\rho+2\epsilon} \subset M$.)
This projection induces an element in $\mathcal{N}_N^*(p_-,p_+;[\pi\circ\mathbf{u}_2];J_D)$, 
and thus the index difference 
$(\mathrm{ind}_{f_D}(p_+) + 2c_1^{TD}([\pi\circ\mathbf{u}_2])) - \mathrm{ind}_{f_D}(p_-)$ 
is nonnegative. We denote by $d_0$, $d_1$, and $d_2$ the contributions 
to $d_{\mathrm{SC}}$ arising from pairs $(\mathbf{u}_1 = \emptyset,\mathbf{u}_2)$ with index difference $0$, $1$, and $2$, respectively.

	We now investigate all possible cases separately. 
		
		\medskip
		
		\noindent Case 1) ($\mathbf{a,b,c,d})=(0,2,0,-1)$; this case corresponds to $d_0[\hat{p},m,0]=[\check{p},m+1,0]$.
		
		\medskip
		
			Since $\mathbf{(c)}=0$,  $\mathbf{u}_1=\emptyset$. Since $\mathbf{(a)}=0$, by Remark \ref{rem:index_0_proj}, we have $N=0$, $[\pi\circ\mathbf{u}_2]=0$ and $p:=p_+=p_-$.
			 Moreover, $\mathbf{u}_2$ is contained in the fiber $\widehat{E}_p$ and $A=[\bar{x}_-^\mathrm{fib}\#\mathbf{u}_2\#(-\bar{x}_+^\mathrm{fib})]=0$. Since $\mathbf{(b)}=2$, $m_+=m_-+1$. 	
			Hence, $\mathbf{u}_2$ has the form $(c_0,u,c_1)$, where $c_0$ is constant at $q_-=\hat{p}$ and $c_1$ is constant at $q_+=\check{p}$ by $\mathbf{(d)}=-1$.

			Note that $u$ is a Floer cylinder contributing to the boundary operator of the Floer complex for $H_i|_{\widehat{E}_{p}}$. From the computation of the Floer homology of $H_i|_{\widehat{E}_{p}}$ (e.g.~vanishing of the Rabinowitz Floer homology of $\C$), we see that the count of such $u$ equals 1. This proves  
			$d_0[\hat{p},m,0]=[\check{p},m+1,0]$.
		
			\medskip

		\noindent Case 2) ($\mathbf{a,b,c,d})=(0,0,0,1)$; this case corresponds to $d_0[\check{p},m,0]=0$.
			
			\medskip
			
			Since $\mathbf{(a)=(c)}=0$, as in Case 1, $\mathbf{u}_1=\emptyset$ and  $\mathbf{u}_2$ is contained in the fiber over $\widehat{E}_{p}$, where $p:=p_+=p_-$. Moreover, $\mathbf{(b)}=0$ and $A=[\bar{x}_-^\mathrm{fib}\#\mathbf{u}_2\#(-\bar{x}_+^\mathrm{fib})]=0$ imply that $m_-=m_+$. Therefore, $\mathcal{A}_{H_i}([x_-,\bar{x}_-])=\mathcal{A}_{H_i}([x_+,\bar{x}_+])$, which implies that $\mathbf{u}_2$ does not have any Floer cylinder. Thus $\mathbf{u}_2$ consists of a single gradient flow line of $f_{\partial E}$ in the fiber $\partial E_{p}\cong S^1$ connecting $\check{p}$ to $\hat{p}$. Since such gradient flow lines appear in pairs with opposite signs, they make a trivial contribution to  $d_\SC$.
						
			\medskip
			
			\noindent Case 3) ($\mathbf{a,b,c,d})=(1,0,0,0)$; this case corresponds to $d_1$.
		
			\medskip
			
			We again have $\mathbf{u}_1=\emptyset$ by $\mathbf{(c)}=0$.
			By Remark \ref{rem:index_0_proj}, $\mathbf{(a)}=1$ implies $N=0$. Thus, every Floer cylinder in $\mathbf{u}_2$ is contained in a fiber of $\widehat{E}$. Moreover, the projection of $\mathbf{u}_2$ to $D$ gives a gradient flow line of $f_D$ connecting $p_-$ to $p_+$. Since $\omega_D(\pi(\bar{x}_-^\mathrm{fib}\#\mathbf{u}_2\#(-\bar{x}_+^\mathrm{fib})))=0$, we have $\omega(\bar{x}_-^\mathrm{fib}\#\mathbf{u}_2\#(-\bar{x}_+^\mathrm{fib}))=0$ due to $\omega=\pi^*\omega_D+\frac{1}{k}d(r^2\alpha)$ on $E$. Thus, $\mathbf{(b)}=0$ yields $m_+=m_-$.
			This again implies that $\mathcal{A}_{H_i}([x_-,\bar{x}_-])=\mathcal{A}_{H_i}([x_+,\bar{x}_+])$, and in turn $\mathbf{u}_2$ contains no Floer cylinder.
			 Thus, $\mathbf{u}_2=(c_0)$ consists of a single gradient flow line $c_0$ of $f_{\partial E}$. The condition $\mathbf{(d)} = 0$ implies that the asymptotic ends of $c_0$ are either $\hat{p}_-$ and $\hat{p}_+$, or $\check{p}_-$ and $\check{p}_+$. By \cite[Proposition 5.1]{BKK24},
			 gradient flow lines of $f_{\partial E}$ connecting $\check{p}_-$ to $\check{p}_+$ (resp.~$\hat{p}_-$ to $\hat{p}_+$) are in sign-preserving (resp.~sign-reversing) bijection with gradient flow lines of $f_D$ from $p_-$ to $p_+$. This establishes that Case 3 corresponds to $d_1$.

			\medskip
			
			\noindent Case 4) ($\mathbf{a,b,c,d})=(2,0,0,-1)$; this case corresponds to $d_2$.
			
			\medskip
			
			We have $\mathbf{u}_1=\emptyset$ by $(\mathbf{c})=0$. By Remark \ref{rem:index_0_proj}, $\mathbf{(a)}=2$ implies that $N$ is either 0 or 1. When $N=0$, arguing as in Case 3, we know that $\mathbf{u}_2$  consists of a single gradient flow line of $f_{\partial E}$. We claim that, when $N=1$, $\mathbf{u}_2$ contains only one Floer cylinder. Since the condition $\mathbf{(b)}=0$ implies $\mathbf{u}\cdot D=0$, no orbit of type IV appears in $\mathbf{u}$ by Lemma \ref{lem:intersect}.(a), and hence none appears in $\mathbf{u}_2$. Moreover, $\mathbf{u}_2$ does not have type II orbits as (intermediate) asymptotic orbits since otherwise $\mathbf{u}_2$ contains a Floer cylinder in $\widehat{E}\setminus D$ connecting a type II orbit to a type III orbit, which contradicts  Lemmas \ref{lem:mp} and \ref{lem:ab}. As mentioned above, any Floer cylinder in a fiber that connects type III orbits and does not intersect the zero-section is trivial. This proves the claim. Thus, by $(\mathbf{d})=-1$, we reduce to the following two cases,  according to whether $N=0$ or $N=1$: 
			\begin{enumerate}[(i)]
				\item $\mathbf{u}_2=(c_0)$, where $c_0$ is a gradient flow line of $f_{\partial E}$ connecting $\hat{p}_-$ to $\check{p}_+$ or 
				\item $\mathbf{u}_2=(c_0,u,c_1)$, where $c_0$ and $c_1$ are constant, and  $\ev_-(u)=\hat{p}_-$ and $\ev_+(u)=\check{p}_+$.
			\end{enumerate}  
	 			In the first case, $\mathbf{u}_2$ projects to an element in the moduli space ${\mathcal{N}}(p_-,p_+)$ of gradient flow lines of $f_D$ connecting $p_-$ to $p_+$, whereas in the second case, $\mathbf{u}_2$ projects to an element of ${\mathcal{N}}^*_1(p_-,p_+;A;J_D)$. Conversely, among the lifts of elements in ${\mathcal{N}}(p_-,p_+)$ or ${\mathcal{N}}^*_1(p_-,p_+;A;J_D)$ to $E \setminus D$, only finitely many   satisfy the asymptotic conditions $\hat{p}_-$ and $\check{p}_+$. According to \cite[Lemma 5.5]{BKK24}, the  count of such lifts yields the coefficients of the maps $\delta_1^{k\omega_D}$ and $\delta_2^{k\omega_D}$. Indeed, by Lemma \ref{lem:mp}, $\mathbf{u}_2$ is contained in the region where $J_{\widehat{E}}$ satisfies $J_{\widehat{E}} Z = -R_\alpha = R_{\lambda|_{\partial W}}$ (see \eqref{eq:j_fiber}), allowing us to apply the results from \cite{BKK24}. Consequently, the contribution of $\mathbf{u}_2$ in this case corresponds to $d_2$.

			\medskip
			
			\noindent Case 5) ($\mathbf{a,b,c,d})=(0,0,2,-1)$; this case corresponds to $d_W$.
			
			\medskip
			
			Since $\mathbf{(c)}=2$, $(\tau-k)\omega(\bar{\mathbf{u}}_1)=1$. This yields $\omega(\bar{\mathbf{u}}_1)=m_M$ since $(\tau-k)m_M$ is an integer. In particular, $\bar{\mathbf{u}}_1$ consists of a single $J$-holomorphic sphere. Since $\mathbf{(a)}=0$, $\mathbf{u}_2$ is contained in $\widehat{E}_p$, where $p:=p_+=p_-$. Since $A=[\bar{x}_-^\mathrm{fib}\#\mathbf{u}_1\#\mathbf{u}_2\#(-\bar{x}_+^\mathrm{fib})]$ and $\mathbf{u}_1$ does not intersect $D$, $k\omega(A)=A\cdot[D]=m_--m_++\mathbf{u}_2\cdot D$.  Moreover, $(\mathbf{b})=0$ implies  $\mathbf{u}_2\cdot D=0$. Arguing as in Case 4, we know that $\mathbf{u}_2$ 
			contains exactly  one Floer cylinder that  has a single puncture whose asymptotic orbit coincides with that of $\mathbf{u}_1$.  Thus, $\mathbf{u}_2$ has the form $(c_0,u,c_1)$, where $c_0$ and $c_1$ are constant curves at $\hat{p}_-$ and   $\check{p}_+$ respectively, and $u$ is a Floer cylinder with one puncture contained in $\widehat{E}_{p}\setminus \{0\}$. The asymptotic Reeb orbit at the puncture winds $m_+-m_-$ times around the origin of $\widehat{E}_{p}$, thus has multiplicity $m_--m_+$ (as a periodic orbit of $R_{\lambda|_{\partial W}}$). By \cite[Lemma 7.2]{DL19b}, the number of such punctured Floer cylinders equals $m_--m_+$. 
			
			Therefore, $\bar{\mathbf{u}}_1$ has homotopy class $A$ and intersects $D$ exactly at one point $q$ with order of contact $m_--m_+$, see Lemma \ref{lem:remsing}. Using $\mathbf{(b)}=0$ and $\mathbf{(c)}=2$ again, we deduce that $m_--m_+=k\omega(A)=km_M$. Thus, the  count of such spheres is $N_1(M,D)$, as defined in \eqref{eq:N_1}.		
			This proves that the number of the pairs $(\mathbf{u}_1,\mathbf{u}_2)$ in this case equals $k N_1(M,D)$, as claimed in the formula of $d_W$.

			\medskip

			Note that $d_2[\check{q},m,0]=0$ and $d_W[\check{q},m,0]=0$ follow from the fact that $\mathbf{(a)}=2$ or $\mathbf{(c)}=2$ forces $e_-=0$. The above five cases, together with this observation, complete the proof.
	\end{proof}
	
	\begin{remark}
		If $(\tau-k)\omega(\pi_2(M))\neq\Z$, e.g.~when $\tau-k\geq 2$, then $d_W=0$. 
	\end{remark}

	\section{Non-vanishing criteria}

	Recall from the discussion above Theorem \ref{thm:equal} that an element $[q,m,A] \in \widetilde{\mathcal{P}}^{\crit}_{\III}$ refers to a periodic orbit on $\partial U_\rho$ starting at $q$ with a capping disk $\bar{x}^\mathrm{fib}\#A$, and its index and action are given by	
	\[
	\mu([q,m,A])=2m+2c_1^{TM}(A)+\ind_{f_{\partial E}}(q)-n,\qquad 
	\mathcal{A}([q,m,A]) = \frac{\rho}{k} m+\omega(A).
	\]
	By the monotonicity of $\omega$, we have
	\begin{equation}\label{eq:action_index_III}
			\mathcal {A}([q,m,A])= \Big(1-\frac{\rho \tau}{k}\Big)\omega(A) + \frac{\rho}{2k}(\mu([q,m,A])-\ind_{\partial E}(q)+n).
	\end{equation}
	At the critical radius $\rho_\crit=\frac{k}{\tau}$, we have
	\begin{equation}\label{eq:crit_radius}
			{2m+2c_1^{TM}(A)}=2\tau\mathcal{A}([q,m,A]),
	\end{equation}
	where the left-hand side is the Robbin-Salamon index of $[q,m,A]$, ignoring the contribution from $f_{\partial E}$.	
	By Theorem \ref{thm:equal}, a (homogeneous) element of  $\SC^*_M(\partial U_\rho ,\mathcal{H},\mathcal{J})$ has the form 
		\begin{equation}\label{eq:infinite_sum}
		\sum_{j=1}^{\infty}a_j \xi_j,\qquad a_j\in \mathbb{Z},\;\; \xi_j=[q_j,m_j,A_j]\in \widetilde{\mathcal{P}}^{\crit}_{\III}, \;\;\lim_{j\rightarrow\infty}\mathcal{A}(\xi_j)=\infty.
		\end{equation}
	In view of \eqref{eq:action_index_III}, the allowed type of infinite sum depends on $\rho$.
		For $\rho<\rho_\crit$, we have $\displaystyle \lim_{j\rightarrow\infty}\mathcal{A}(\xi_j)=\infty $ if and only if $\displaystyle\lim_{j\to\infty}\omega(A_j)=\infty$. This in turn implies $\displaystyle\lim_{j\to\infty}m_j=-\infty$ since the degree $\mu$ is fixed.  For $\rho>\rho_\crit$,  $\displaystyle \lim_{j\rightarrow\infty}\mathcal{A}(\xi_j)=\infty $ if and only if $\displaystyle\lim_{j\to\infty}\omega(A_j)=-\infty$. In the critical case $\rho=\rho_\crit$, no infinite sums are allowed.

	To capture this phenomenon, we introduce an extension of the ring $\Gamma$ as follows:
	\begin{equation*}
		\widehat\Gamma_\rho :=\begin{cases}
			\Gamma((S)) & \rho>\rho_{\crit},\\
			\Gamma[S,S^{-1}] & \rho=\rho_{\crit},\\
			 \Gamma((S^{-1})) & \rho<\rho_{\crit},
			\end{cases}
	\end{equation*}
	where the degree of the formal variable $S$ is $2$. We then consider the Morse cochain complex $\CM^*(f_D;\widehat\Gamma_\rho)$ and denote the cohomology of this complex by
	\[
	\QH^*(D;\widehat\Gamma_\rho)\cong \H^*(D;\Z)\otimes_\Z \widehat{\Gamma}_\rho.
	\]
	We equip $\QH^*(D;\widehat\Gamma_\rho)$ with the ring structure obtained by extending the quantum cup product on $\QH^*(D;\Gamma)$.
	To state the main result of this section, we introduce two chain maps. The first one is
	\begin{equation*}
		\delta:= \delta^{k[\omega_D]_\Z} + kN_1(M,D)T +  S :\CM^{*-2}(f_D;\widehat\Gamma_\rho)\rightarrow \CM^{*}(f_D;\widehat\Gamma_\rho)
	\end{equation*}
	where $\delta^{k[\omega_D]_\Z}$ is the chain-level quantum cup product with $k[\omega_D]_\Z$ and $N_1(M,D)\in\Z$ is the count of $J$-holomorphic spheres $v$ in $M$ intersecting $D$ precisely at one point with order of contact $k\omega([v])$, see \eqref{eq:N_1} and above. Here $S$ and $T$ act via the corresponding  module operations. We also consider the chain map 
 	\begin{align*}
		\check{\iota}:\CM^*(f_D;\widehat\Gamma_\rho)&\longrightarrow \SC^{*-n}_M(\partial U_\rho, \mathcal{H}, \mathcal{J})\\
		\beta S^m&\longmapsto\check{\iota}_m(\beta),
	\end{align*}
	where $\beta \in \CM^*(f_D;\Gamma)$ and $\check{\iota}_m$ is defined in \eqref{eq:incl}. By the observation after \eqref{eq:infinite_sum}, the map $\check{\iota}$ is well-defined for all $\rho\in(0,1)$. It is indeed a chain map as $\check{\iota}_m$ is for every $m\in\Z$, see Theorem \ref{thm:main}. One can analogously define a  map $\hat{\iota}$ (not necessarily a chain map) using the maps $\hat{\iota}_m$ in \eqref{eq:incl}. The following theorem proves Theorem \ref{thm:sh}. 
	
	\begin{theorem}\label{thm:gysin}
		For every $\rho\in(0,1)$, we have a short exact sequence
		\[\begin{tikzcd}[column sep=15pt]
			0\ar[r]& \QH^{*-2}(D;\widehat\Gamma_\rho)\ar[r,"\delta^*"] & \QH^*(D;\widehat\Gamma_\rho)\ar[r,"\check{\iota}^*"] & \SH^{*-n}_M(\partial U_\rho)\ar[r] &0.
		\end{tikzcd}\]
	\end{theorem}
	\begin{proof}
		We observe that $\check{\iota}$ is injective and the cokernel of $\check{\iota}$ is isomorphic to the image of $\hat{\iota}$. Moreover, by Theorem \ref{thm:main}, $\SC^{*-n}_M(\partial U_\rho,\mathcal{H},\mathcal{J})$ is isomorphic to the cone complex of $\delta$. This yields a short exact sequence of chain complexes
		\[\begin{tikzcd}[column sep=10pt]
			0\ar[r] & \CM^*(f_D;\widehat\Gamma_\rho)\ar[r,"\check{\iota}"] &\SC^{*-n}_M(\partial U_\rho,\mathcal{H},\mathcal{J}) \ar[r] &\CM^{*-1}(f_D;\widehat\Gamma_\rho)\ar[r] &0.
		\end{tikzcd}\]
		and in turn a long exact sequence
		\[\begin{tikzcd}[column sep=15pt]
			\cdots\ar[r]& \QH^{*-2}(D;\widehat\Gamma_\rho)\ar[r,"\delta^*"] & \QH^*(D;\widehat\Gamma_\rho)\ar[r,"\check{\iota}^*"] & \SH^{*-n}_M(\partial U_\rho)\ar[r] &\cdots.
		\end{tikzcd}\]
		It remains to show that $\delta^*$ is injective. We write
		\begin{equation}\label{eq:e_M}
			e_M:=k[\omega_D]_\Z+{k}N_1(M,D) T\in \QH^2(D; \Gamma).
		\end{equation}
		Suppose $\rho\leq \rho_{\crit}$. Then, $\widehat\Gamma_\rho$ is either $\Gamma((S^{-1}))$ or $\Gamma[S,S^{-1}]$. Thus, any nonzero element of $\QH^*(D;\widehat\Gamma_\rho)$ can be written as 
		\begin{equation*}
			\sum_{m\leq m_0} Z_m S^m ,\qquad Z_m\in \QH^*(D;\Gamma),\; Z_{m_0}\neq 0,
		\end{equation*}
		and it is mapped to a nonzero element by $\delta^*=\star(e_M+S)$:
		\begin{equation*}
			\delta^*\left(\sum_{m\leq m_0} Z_m S^m\right)=Z_{m_0} S^{m_0+1}+\sum_{m\leq m_0}(e_M\star Z_m+Z_{m-1}) S^m\neq 0.
		\end{equation*}
		This proves that $\delta^*$ is injective. 
		
		Suppose $\rho>\rho_{\crit}$. Then $\widehat\Gamma_\rho=\Z[T,T^{-1}]((S))$, and any nonzero element of $\QH^*(D;\widehat\Gamma_\rho)$ has the form 
		\begin{equation*}
			\sum_{m\geq m_0} Z_m S^m,\qquad Z_m\in\QH^*(D;\Gamma),\; Z_{m_0}\neq 0.
		\end{equation*}
		Assume that it is in the kernel of $\delta^*$:
		\begin{equation*}
			\delta^*\left(\sum_{m\geq m_0} Z_m S^m\right)=\sum_{m\geq m_0+1}(e_M\star Z_m+Z_{m-1}) S^m+ (e_M \star Z_{m_0}) S^{m_0}=0.
		\end{equation*}
		Then, we have 
		\begin{equation*}
			e_M \star Z_{m_0}=0, \qquad Z_m=-e_M \star Z_{m+1}  \;\;\;\forall m\geq m_0,
		\end{equation*}
		and the latter yields
		\begin{equation*}
			Z_{m_0}=(-e_M)^{\star(m-m_0)}Z_m \;\;\;\forall m\geq m_0.
		\end{equation*}
		Therefore, $Z_{m_0}$ belongs to the ideal $\displaystyle  \bigcap_{i \in \N}(e_M^{\star i})$ of $\QH^*(D;\Gamma)$.
		However, the map 
		\begin{equation*}
			\star e_M:\bigcap_{i \in \N}(e_M^{\star i})\longrightarrow \bigcap_{i \in \N}(e_M^{\star i})
		\end{equation*}
		is an isomorphism, and this contradicts $e_M \star Z_{m_0}=0$. Hence, $\delta^*$ is injective. 
	\end{proof}
	
	\begin{theorem}\label{thm:SH=0}
	We have
		$\SH^*_M(\partial U_\rho)=0$ if and only if one of the following conditions holds.
		\begin{enumerate}[(a)]
			\item $\rho <\rho_\crit$,
			\item $\rho=\rho_\crit$ and the class $e_M$ is nilpotent,
			\item $\rho>\rho_\crit$ and the descending sequence of ideals
			\begin{equation}\label{ideal}
				\QH^*(D;\Gamma)\supseteq (e_M)\supseteq (e_M^{\star 2}) \supseteq (e_M^{\star 3})\supseteq \cdots
			\end{equation}
			stabilizes, where the class $e_M$ is defined in \eqref{eq:e_M}.
		\end{enumerate}
	\end{theorem}
	\begin{proof}
		By Theorem \ref{thm:gysin}, $\SH^*_M(\partial U_\rho)=0$ if and only if the map $\delta^*$ is surjective. It is also equivalent to the condition that the unit element  $1_D$ of $\QH^0(D;\widehat\Gamma_\rho)$ is contained in the image of $\delta^*$, which equals the ideal $(e_M+S)$.
		\medskip
		
		\noindent Case (a). If $\rho < \rho_\crit$, we always have $1_D\in \mathrm{im\,}\delta^*$ since
			\begin{equation*}
				\delta^*\left(\sum_{m\ge 1} (-e_M)^{\star (m-1)} S^{(-m)}\right)=\sum_{m\ge 1} (-e_M)^{\star (m-1)} S^{(-m+1)}-\sum_{m\ge 1} (-e_M)^{\star m} S^{(-m)} =1_D.
			\end{equation*}
			\medskip
			
			\noindent Case (b).
			 Let $\rho=\rho_\crit$. Suppose that $e_M$ is nilpotent, i.e.~$e_M^{\star m_0}=0$ for some $m_0\in\N$. Then, we have
			\begin{equation*}
				\delta^*\left(\sum_{m=1}^{m_0} (-e_M)^{\star (m-1)} S^{(-m)}\right)=1_D.
			\end{equation*}

			Conversely, suppose that there is an element $\displaystyle \sum_{m=m_0}^{m_1} Z_m S^{(-m)}\in \QH^{-2}(D;\widehat\Gamma_\rho)$ with $Z_{m_0}\neq 0$ that is mapped to $1_D$ by $\delta^*$. 
			This implies
			\begin{equation*}
				Z_{m_0} S^{(-m_0+1)}+\sum_{m=m_0}^{m_1-1}(e_M\star Z_m+Z_{m+1}) S^{(-m)}+ (e_M\star Z_{m_1}) S^{(-m_1)}=1_D,
			\end{equation*}
			and, therefore, we have
			\begin{equation*}
				Z_{m_0}=1_D, \quad m_0=1, \quad Z_{m_1}=(-e_M)^{\star (m_1-1)}, \quad e_M\star Z_{m_1}=0.
			\end{equation*}
			This proves that $e_M$ is nilpotent.
			
			\medskip
			
			\noindent Case (c). Let $\rho>\rho_\crit$. 
			Suppose the sequence (\ref{ideal}) stabilizes, which means 
			\begin{equation*}
				(e_M^{\star m_0})=(e_M^{\star (m_0+1)})=\bigcap_{i\in\N}(e_M^{\star i})
			\end{equation*}
			for some $m_0\in\N$. Then, we have an isomorphism
			\begin{equation*}
				\star e_M:(e_M^{\star m_0})\longrightarrow (e_M^{\star m_0}),
			\end{equation*}
			and we denote its inverse by
			\begin{equation*}
				\mathcal{E}:=(\star e_M)^{-1}:(e_M^{\star m_0})\longrightarrow (e_M^{\star m_0}).
			\end{equation*}
			For
			\begin{equation*}
				Z_m:=\begin{cases}
					(-1)^m\mathcal{E}^{m_0-m+1}(e_M^{\star m_0}) & m\le0\\[.5ex]
					(-1)^m\mathcal{E}^{m_0-m+1}(e_M^{\star m_0})+(-e_M)^{\star (m-1)} & 1\le m \le m_0,
				\end{cases}
			\end{equation*}
			we have
			\begin{equation*}
			\begin{cases}
				e_M\star Z_m+Z_{m+1}=0 \;\text{ for }\; 0\neq m \leq  m_0-1\\[.5ex]
			e_M\star Z_0+Z_1=e_M\star \mathcal{E}^{m_0+1}(e_M^{\star m_0})-\mathcal{E}^{m_0}(e_M^{\star m_0})+1_D=1_D  \\[.5ex]
			e_M\star Z_{m_0}=(-1)^{m_0}e_M\star (\mathcal{E}(e_M^{\star m_0})-e_M^{\star(m_0-1)})=0,
			\end{cases}
			\end{equation*}
			and consequently,		\begin{equation}\label{eq:=1_D}
				\delta^*\left(\sum_{m\le m_0} Z_m S^{(-m)}\right)=\sum_{m\leq m_0-1}(e_M\star Z_{m}+Z_{m+1})S^{(-m)} + (e_M\star Z_{m_0})S^{(-m_0)}= 1_D.
			\end{equation}
			
		Conversely, suppose that there exists an element $\displaystyle\sum_{m\le m_0} Z_m S^{(-m)} \in \QH^*(D;\widehat\Gamma_\rho)$ with $Z_{m_0}\neq 0$ satisfying \eqref{eq:=1_D}.
			Then, $m_0 \in\N\cup\{0\}$. If $m_0=0$, then 
			\[
			e_M\star Z_0=1_D,\qquad e_M\star Z_m+Z_{m+1}=0 \;\;\forall m\leq -1.
			\] 
			The latter condition implies 
			$\displaystyle
				Z_0\in \bigcap_{i \in \N}(e_M^{\star i})
			$,
			and thus by the first condition 
			\[
			1_D = e_M\star Z_0\in \bigcap_{i \in \N}(e_M^{\star i}).
			\]
			
			If $m_0\geq1$, then we have instead
			\begin{equation*}
				\quad e_M\star Z_0+Z_1=1_D,\quad e_M\star Z_{m_0}=0,\quad
				e_M\star Z_m+Z_{m+1}=0 \;\; 0\neq\forall  m \leq  m_0-1.
			\end{equation*}
		 	The last condition again implies 
			$
				\displaystyle Z_0\in \bigcap_{i \in \N}(e_M^{\star i})
			$.
			The first and the last conditions yield
			\begin{equation*}
				Z_{m_0}=(-e_M)^{\star (m_0-1)}\star Z_1=(-e_M)^{m_0-1}\star (1_D-e_M\star Z_0).
			\end{equation*}
			Multiplying both sides by $e_M$ and using $e_M\star Z_{m_0}=0$, we obtain
			\begin{equation*}
				e_M^{\star m_0}=e_M^{\star (m_0+1)}\star Z_0\in \bigcap_{i \in \N}(e_M^{\star i}).
			\end{equation*}
			Therefore,
			$
			\displaystyle	(e_M^{\star m_0})\subset\bigcap_{ i \in \N}(e_M^{\star i}).
			$
			This proves that the sequence (\ref{ideal}) stabilizes.
	\end{proof}

	\begin{lem}\label{lem:non_stab}
		Assume that $e_M=k([\omega_D]_\Z+N_1(M,D)T)$ is not nilpotent in $\QH^*_\mathrm{free}(D;\Gamma)=\H^*(D;\Z)/\mathrm{torsion}\otimes_\Z\Gamma$ and not a primitive class, i.e.~$e_M=a\beta$ for some $a\geq 2$ and $\beta\in\QH^*_\mathrm{free}(D;\Gamma)$ (e.g.~when $k\geq 2$). Then, the sequence of ideals in \eqref{ideal} never stabilizes.
	\end{lem}
	\begin{proof}
		By the assumption, for any $i\in \N$, there exists $l\geq i$ such that 
		\[
		(e_M^{\star i})\subset (a^l 1_D),\qquad (e_M^{\star i})\not\subset(a^{l+1}1_D).
		\]
		However, $(e_M^{\star(l+1)})\subset(a^{l+1}1_D)$, and hence $(e_M^{\star i})\neq(e_M^{\star (l+1)})$. Thus, the image of the sequence \eqref{ideal}  in $\QH^*_\mathrm{free}(D;\Gamma)$ does not stabilize. This proves that the sequence \eqref{ideal} itself does not stabilize.
	\end{proof}

	\section{Some computations}\label{sec:computation}
	
	Now, we calculate $\SH_M^*(\partial U_\rho)$ for some examples.
	\begin{prop}\label{prop:cp_n}
		Let $(M,\omega)=(\mathbb{CP}^n,\omega_\mathrm{FS})$ and $D=\mathbb{CP}^{n-1}$ for $n\geq1$. Then, we have
		\begin{equation*}
			\SH_M^{*-n}(\partial U_\rho)\cong\begin{cases}
				0 & \rho\neq\rho_\crit\\
				\mathbb{Z}[\beta,\beta^{-1}] & \rho=\rho_\crit.
			\end{cases}
		\end{equation*}
		where the formal variable $\beta$ has degree $2$. 
	\end{prop}
	\begin{proof}
		For $n\ge 2$, we have $\QH^*(D;\Gamma)\cong\mathbb{Z}[\beta,\beta^{-1}]$. Since $e_M=[\omega_\mathrm{FS}]=\pm\beta$ is invertible, by Corollary \ref{cor:sh=0}, $\SH_M^*(\partial U_\rho)=0$ for $\rho\neq \rho_\crit$.
		For $\rho=\rho_\crit$, by Theorem \ref{thm:sh},
		\[
			\SH_M^{*-n}(\partial U_{\rho_\crit})\cong\mathbb{Z}[\beta,\beta^{-1}, S, S^{-1}]/(\beta+S)
			\cong\mathbb{Z}[S,S^{-1}].
		\]
		For $n=1$, when $D$ is a point, we have $e_M=T$ and $\QH^*(D;\Gamma)\cong\Gamma=\mathbb{Z}[T,T^{-1}]$. The rest of the computation is the same as the case $n\ge 2$.
	\end{proof}
	
		We fix some notation. For a prime number $p$, let $\Z_p$ denote the ring of $p$-adic integers, and let $\Q_p$ denote the field of $p$-adic numbers. We also need the ring of restricted Laurent series over $\Z_p$:
	\[
	\Z_p\langle\beta,\beta^{-1}\rangle:=\left\{ \sum_{i\in\Z} a_i\beta^i\mid a_i\in\Z_p,\, a_j\stackrel{\text{$p$-adic}}{\longrightarrow}0 \text{ as } j\to\pm\infty\right\}
	\]
	and $\Q_p\langle\beta,\beta^{-1}\rangle$ is defined analogously.
	
	\begin{prop}\label{prop:q_2}
		Let $(M,\omega)=(\mathbb{CP}^1\times\mathbb{CP}^1,\omega_\mathrm{FS}\oplus\omega_\mathrm{FS})$, and $D\subset \mathbb{CP}^1\times\mathbb{CP}^1$ be the diagonal. Then, we have
		\begin{equation*}
			\SH_M^{*-2}(\partial U_\rho)\cong\begin{cases}
				0 &\rho<\rho_\crit\\
				\mathbb{Z}[\beta,\beta^{-1},\frac{1}{2}] &\rho=\rho_\crit\\
				\mathbb{Q}_2\langle\beta,\beta^{-1}\rangle &\rho>\rho_\crit
			\end{cases}
		\end{equation*}
		where $\beta$ has degree 2.
	\end{prop}
	\begin{proof}
		In this case, $k=1$ and $[\omega_D]_\Z=2[\beta]$. First, we have
		\begin{equation*}
			\QH^*(D;\Gamma)=\mathbb{Z}[\beta, T, T^{-1}]/(\beta^2-T^2)\cong \mathbb{Z}[\beta,\beta^{-1},\sigma]/(\sigma^2-1)
		\end{equation*}
		where $\deg\beta=\deg T=2$, $\beta^{-1}=\beta T^{-2}$, and $\sigma:=T\beta^{-1}$. For each of the homology classes $[\mathbb{CP}^1\times\{\mathrm{pt}\}]$ and $[\{\mathrm{pt}\}\times\mathbb{CP}^1]$, there is a unique complex sphere intersecting the diagonal $D$ at a given point in this homology class. Thus, $N_1(M,D)=2$, and 
		\begin{equation*}
			e_M=2\beta+2T=2\beta(1+\sigma).
		\end{equation*}
		For $\rho=\rho_\crit$, by Theorem \ref{thm:sh}, we obtain
		\begin{align*}
			\SH_M^{*-2}(\partial U_{\rho_\crit})&\cong \mathbb{Z}[\beta,\beta^{-1},\sigma,S,S^{-1}]/(\sigma^2-1,2\beta(1+\sigma)+S)\\
			&\cong \mathbb{Z}[\beta,\beta^{-1},\sigma,\tfrac{1}{2(1+\sigma)}]/(\sigma^2-1)\\
			&\cong \mathbb{Z}[\beta,\beta^{-1},\sigma,\tfrac{1}{2},\tfrac{1}{1+\sigma}]/(\sigma-1)\\
			&\cong\mathbb{Z}[\beta,\beta^{-1},\tfrac{1}{2}],
		\end{align*}
		where the third isomorphism follows from the fact that both $2$ and $(1+\sigma)$ are invertible as $2(1+\sigma)$ is invertible.  
		Similarly, for $\rho> \rho_\crit$, using  Theorem \ref{thm:sh}, we compute
		\begin{equation}\label{eq:Q2}
			\begin{split}
			\SH_M^{*-2}(\partial U_\rho)&\cong \mathbb{Z}[\beta,\beta^{-1},\sigma]((S))/(\sigma^2-1,\,2\beta(1+\sigma)+S)\\[.5ex]
			&\cong\mathbb{Z}[\beta,\beta^{-1}]((S))/(4\beta+S)\\[.5ex]
			&\cong\frac{\mathbb{Z}[\beta,\beta^{-1}][[S]]}{(4\beta+S)}[S^{-1}] \\[.5ex]
			&\cong \Z_2\langle\beta,\beta^{-1}\rangle[(4\beta)^{-1}]\\
			&\cong \Z_2\langle\beta,\beta^{-1}\rangle[\tfrac{1}{4}]\\
			&\cong \Q_2\langle\beta,\beta^{-1}\rangle.
			\end{split}
		\end{equation}
		As before, the invertibility of $2\beta(1+\sigma)=-S$ implies that $1+\sigma$ is invertible. This together with $\sigma^2-1=0$ implies $\sigma=1$. This is how we derive the second isomorphism. To obtain the fourth isomorphism, we use 
		 $\displaystyle\Z[\beta,\beta^{-1}][[S]]\cong \varprojlim_{n\to\infty}\Z[\beta,\beta^{-1}][S]/(S^n)$ and derive
		\[
		\frac{\Z[\beta,\beta^{-1}][[S]]}{(4\beta+S)}\cong \varprojlim_{n\to\infty}\frac{\Z[\beta,\beta^{-1}][S]}{(S^n,4\beta+S)} \cong \varprojlim_{n\to\infty}\frac{\Z[\beta,\beta^{-1}]}{(4^n)}\cong \varprojlim_{n\to\infty}\frac{\Z[\beta,\beta^{-1}]}{(2^n)}\cong \Z_2\langle\beta,\beta^{-1}\rangle.
		\]
		The last one follows from $\Z_2[\frac{1}{4}]\cong \Z_2[\frac{1}{2}]\cong \Q_2$. 
	\end{proof}
	
	\begin{prop}\label{prop:quadric}
		Let $(M,\omega)=(\mathbb{CP}^n,\omega_\mathrm{FS})$ and $D=Q^{n-1}$ be a quadric. Then, we have
		\begin{align*}
			&\text{For $n=2$: }\; \SH_M^{*-2}(\partial U_\rho) \cong \begin{cases}
				0 & \rho<\rho_\crit \\
				\mathbb{Z}[\beta,\beta^{-1},\frac{1}{6}]\times\mathbb{Z}[\beta,\beta^{-1},\frac{1}{2}] &\rho=\rho_\crit\\
				\mathbb{Q}_2\langle\beta,\beta^{-1}\rangle\times \mathbb{Q}_2\langle\beta,\beta^{-1}\rangle \times \mathbb{Q}_3\langle\beta,\beta^{-1}\rangle &\rho>\rho_\crit
				\end{cases} 
				\\[1ex]
				&\text{For $n=3$: }\; \SH_M^{*-3}(\partial U_\rho) \cong 
				\begin{cases}
					0 &\rho<\rho_\crit \\
					\mathbb{Z}[\beta,\beta^{-1},\frac{1}{2}] & \rho=\rho_\crit\\
					\mathbb{Q}_2\langle\beta,\beta^{-1}\rangle & \rho>\rho_\crit
				\end{cases} 
		\end{align*}	
		where $\beta$ has degree 2. 
	\end{prop}

	\begin{proof}
		For $n=2$, we have $D=Q^1\cong \mathbb{CP}^1$ and thus
		\begin{equation*}
			\QH^*(D,\Gamma)=\mathbb{Z}[\beta,\beta^{-1}, T, T^{-1}]/(\beta^2-T^2)\cong \mathbb{Z}[\beta,\beta^{-1},\sigma]/(\sigma^2-1)
		\end{equation*}		
		where $\deg\beta=\deg T=2$, and $\sigma:=T\beta^{-1}$. We also have $k=2$ and $[\omega_D]_\Z=2\beta$. Since there is a unique complex line in $\mathbb{CP}^2$ intersecting $Q^1$ at a prescribed point with contact order 2 (i.e. being tangent to $Q^1$), we have $N_1(M,D)=1$. Thus, we have
		\begin{equation*}
			e_M=4\beta+2T=2\beta(2+\sigma).
		\end{equation*}
		For $\rho=\rho_\crit$, we compute
		\begin{align*}
			\SH_M^{*-2}(\partial U_\rho)&\cong\mathbb{Z}[\beta,\beta^{-1},\sigma,S,S^{-1}]/(\sigma^2-1,2\beta(2+\sigma)+S)\\
			&\cong\mathbb{Z}[\beta,\beta^{-1},\sigma,\tfrac{1}{2\beta(2+\sigma)}]/(\sigma^2-1)\\
			&\cong\mathbb{Z}[\beta,\beta^{-1},\sigma,\tfrac{1}{2},\tfrac{1}{2+\sigma}]/(\sigma^2-1)\\
			&\cong\mathbb{Z}[\beta,\beta^{-1},\tfrac{1}{6}]\times\mathbb{Z}[\beta,\beta^{-1},\tfrac{1}{2}].
		\end{align*}
		The second isomorphism follows from the fact that the invertibility of $2\beta(2+\sigma)$ implies that both $2$ and $2+\sigma$ are invertible. To see the fourth line, we observe 
		\[
		\Z[\sigma,\tfrac{1}{2},\tfrac{1}{2+\sigma}]/(\sigma^2-1)\cong \Z[\tfrac{1}{2},\tfrac{1}{2+\sigma}]/(\sigma-1)\times \Z[\tfrac{1}{2},\tfrac{1}{2+\sigma}]/(\sigma+1)\cong \Z[\tfrac{1}{6}]\times \Z[\tfrac{1}{2}].
		\]
		where the first isomorphism follows from the Chinese remainder theorem since the ideals $(\sigma+1)$ and $(\sigma-1)$ are comaximal, i.e.~$1\in (\sigma+1)+(\sigma-1)$.
		
		For $\rho>\rho_\crit$, we compute
		\begin{align*}
			\SH_M^{*-2}(\partial U_\rho)&\cong \Z[\beta,\beta^{-1},\sigma]((S))/(\sigma^2-1,2\beta(2+\sigma)+S)\\
			&\cong \frac{\Z[\beta,\beta^{-1}]((S))}{(6\beta+S)} \times 
			\frac{\mathbb{Z}[\beta,\beta^{-1}]((S))}{(2\beta+S)}\\
			&\cong \frac{\Z[\beta,\beta^{-1}][[S]]}{(6\beta+S)}[S^{-1}]\times
			\frac{\Z[\beta,\beta^{-1}][[S]]}{(2\beta+S)}[S^{-1}]\\
			&\cong (\Z_2\langle\beta,\beta^{-1}\rangle \times\Z_3\langle\beta,\beta^{-1}\rangle)[(6\beta)^{-1}]\times\Z_2\langle\beta,\beta^{-1}\rangle[(2\beta)^{-1}]\\
			&\cong (\Z_2\langle\beta,\beta^{-1}\rangle\times \Z_3\langle\beta,\beta^{-1}\rangle)[\tfrac{1}{6}]\times \Z_2\langle\beta,\beta^{-1}\rangle[\tfrac{1}{2}]\\
			&\cong\Q_2\langle\beta,\beta^{-1}\rangle \times \Q_3\langle\beta,\beta^{-1}\rangle\times\Q_2\langle\beta,\beta^{-1}\rangle.
		\end{align*}
		The second isomorphism follows from the Chinese remainder theorem as before. The fourth line is derived as in \eqref{eq:Q2} and below, using the fact that  $\Z[\beta,\beta^{-1}]/{(6^n)}\cong {\Z[\beta,\beta^{-1}]}/{(2^n)}\times
		{\Z[\beta,\beta^{-1}]}/{(3^n)}$ which also follows from the Chinese remainder theorem.
		\medskip

		For $n=3$, we have $D\cong \mathbb{CP}^1\times \mathbb{CP}^1$ and
		\begin{equation*}
			\QH^*(D;\Gamma)\cong\mathbb{Z}[\beta_1,\beta_2,T,T^{-1}]/(\beta_1^2-T,\beta_2^2-T)\cong\mathbb{Z}[\beta,\beta^{-1},\sigma]/(\sigma^2-1)
		\end{equation*}
		where $\deg\beta_1=\deg\beta_2=2$ and $\deg T=4$. The second isomorphism is given by $\beta_1\mapsto \beta$ and $\sigma=\beta_2\beta_1^{-1}$, where $\beta_1^{-1}=\beta_1T^{-1}$. In this case,
		\[
		e_M=2(\beta_1+\beta_2)=2\beta(1+\sigma).
		\]
		The rest of the computation follows as in the proof of Proposition \ref{prop:q_2} as the formula is identical to the one derived there.
	\end{proof}
	The case where $n\geq 4$ in Proposition \ref{prop:quadric} can be computed analogously. This will be included in the first author's thesis.

	\bibliographystyle{amsalpha}
	\bibliography{./ref.bib}

\end{document}